\documentclass[twoside,11pt]{article}

\usepackage{blindtext}

\usepackage[abbrvbib, preprint]{jmlr2e}
\usepackage{amsmath}
\usepackage{tikz}
\usetikzlibrary{patterns}
\usepackage{todonotes}

\newcommand{\de}{\mathrm{d}}

\newcommand{\indiq}{1\!\! 1}

\newtheorem{assumption}{Assumption}

\newcommand{\norm}[1]{\left\|#1\right\|}
\newcommand{\abs}[1]{\left|#1\right|}
\newcommand{\ip}[2]{\left(#1\right)\cdot\left(#2\right)}

\newcommand{\E}{\mathbb{E}}

\newcommand{\e}{m}

\newcommand{\mass}{m}
\newcommand{\friction}{\alpha}

\newcommand{\R}{\mathbb{R}}

\usepackage{lastpage}
\jmlrheading{23}{2024}{1-\pageref{LastPage}}{1/21; Revised 5/22}{9/22}{21-0000}{Kexin Jin, Jonas Latz, Chenguang Liu, Alessandro Scagliotti}

\ShortHeadings{Losing Momentum in Stochastic Optimisation}{Jin et al.}
\firstpageno{1}

\begin{document}

\title{Losing Momentum in Continuous-time Stochastic Optimisation}

\author{\name Kexin Jin \email kexinj@math.princeton.edu \\
       \addr Department of Mathematics\\
       Princeton University\\
       Princeton, NJ 08544-1000, USA
       \AND
        \name Jonas Latz \email jonas.latz@manchester.ac.uk \\
       \addr Department of Mathematics \\
       The University of Manchester\\
       Manchester, M13 9PL, United Kingdom
       \AND
       \name Chenguang Liu \email C.Liu-13@tudelft.nl \\
       \addr Delft Institute of Applied Mathematics\\
       Technische Universiteit Delft\\
       2628 Delft, The Netherlands
       \AND
       \name Alessandro Scagliotti \email scag@ma.tum.de \\
       \addr  CIT School, Technische Universit\"at M\"unchen, and \\
       Munich Center for Machine Learning (MCML)\\
       85748  Garching bei M\"unchen, Germany
       }

\editor{My editor}

\maketitle

\begin{abstract}
The training of modern machine learning models often consists in solving high-dimensional non-convex optimisation problems that are subject to large-scale data.
In this context, momentum-based stochastic optimisation algorithms have become particularly widespread. The stochasticity arises from data subsampling which reduces computational cost. Both, momentum and stochasticity  help the algorithm to 
converge globally. 
In this work, we propose and analyse a continuous-time model for stochastic gradient descent with momentum. This model is a piecewise-deterministic Markov process that represents the optimiser by an underdamped dynamical system and the data subsampling through a stochastic switching. 
We investigate longtime limits, the subsampling-to-no-subsampling limit, and the momentum-to-no-momentum limit. We are particularly interested in the case of reducing the momentum over time. 
Under convexity assumptions, we show convergence of our dynamical system to the global minimiser when reducing  momentum over time and letting the subsampling rate go to infinity.
We then propose a stable, symplectic discretisation scheme to construct an algorithm from our continuous-time dynamical system. In experiments, we study our scheme in convex and non-convex test problems. Additionally, we train a convolutional neural network in an image classification problem. Our algorithm  {attains} competitive results compared to stochastic gradient descent with momentum.
\end{abstract}

\begin{keywords}
  stochastic optimisation, momentum-based optimisation, piecewise-deterministic Markov processes, stochastic stability, deep learning
\end{keywords}

\section{Introduction}
Machine learning and artificial intelligence play a fundamental role in modern scientific research and modern life. In many instances, the underlying learning process consists of solving a high-dimensional non-convex optimisation problem with respect to large-scale data. These problems have been approached by  applicants and researchers using a large range of different algorithms. Methods are based on, e.g.\ stochastic approximation, statistical mechanics, and ideas from biological evolution. Many of these methods deviate from simply solving the optimisation problem, but are actually also used for regularisation or approximate  uncertainty quantification. Unfortunately, many successfully employed methods are theoretically {hardly} understood.

In this article, we investigate momentum-based stochastic optimisation methods in a contin\-uous-time framework. We now introduce the optimisation problem, motivate stochastic and momen\-tum-based optimisation, review past work, and summarise our contributions { and main results}.

\subsection{Optimisation and continuous dynamics}
We study  optimisation problems of the form 
\begin{equation} \label{eq:optprob}
    \min_{\theta \in X} \bar{\Phi}(\theta),
\end{equation}
on the space $X := \mathbb{R}^K$, where
$$\bar\Phi(\theta) := \frac{1}{N}\sum_{i=1}^N \Phi_i(\theta). $$
We denote the number of terms {by} $N \in \mathbb{N}:=\{1,2,\ldots\}$
 and assume that the functions $\Phi_i$, with $i \in I:=\{1,\ldots,N\}$, are continuously differentiable with Lipschitz gradients. 
 We assume that \eqref{eq:optprob} is well-defined and denote $\theta^* :\in \mathrm{argmin}\,\bar\Phi$. In this work, we study continuous{-time} dynamical systems that can be used to optimise functions of type $\bar\Phi$ and to represent the dynamics of optimisation algorithms. The most basic of these dynamical systems is the \emph{gradient flow}
\begin{equation} \label{eq_gradient_flow_full}
    \frac{\mathrm{d}\zeta_t}{\mathrm{dt}} = -\nabla \bar\Phi(\zeta_t), \qquad \zeta_0  \in X,
\end{equation}
a continuous-time version of the \emph{gradient descent method}.
The ODE solution $(\zeta(t))_{t \geq 0}$ converges to a stationary point of $\bar\Phi$ under{, e.g.\ certain} convexity conditions. In practice, especially in modern machine learning, we often encounter optimisation problems that are non-convex{, where the solution} $(\zeta(t))_{t \geq 0}$ may converge to a saddle point or a local minimiser. When discretising  \eqref{eq_gradient_flow_full}, the iterates actually only converge to (local) minimisers, see, e.g.
\cite{LeeSimJord, PanPil}.
However, as observed in \cite{DuJin}, the {algorithm} could take a lot of iterations to escape a saddle point. 

The \emph{stochastic gradient descent (SGD)} method going back to \cite{RobbinsMonro} has been a popular alternative to the {gradient descent method}. 
In SGD, we randomly select one of the $\Phi_i$ $(i \in I)$ and optimise with respect to that function before switching to another $\Phi_j$ $(j \in I\backslash\{i\})$. {This process is repeated until a stationary state is reached.} A continuous-time model for stochastic gradient descent has recently proposed by \cite{Jonas} and extended by \cite{Jin} and \cite{Jonas2}.
There, stochastic gradient descent is represented through the dynamical system
\begin{equation} \label{eq:SGP_old}
 \frac{\mathrm{d}\theta_t}{\mathrm{dt}} = -\nabla \Phi_{i(t)}(\theta_t), \qquad  \theta_0 \in X,   
\end{equation}
where the \emph{index process} $(i(t))_{t \geq 0}$ is a suitable homogeneous continuous-time Markov process on $I$. The processes $(i(t))_{t \geq 0}$, $(\theta_t)_{t \geq 0}$ and any other stochastic processes and random variables  throughout this work are defined on the underlying probability space $(\Omega, \mathcal{F}, \mathbb{P})$. The process $(\theta_t)_{t \geq 0}$  represents a randomised gradient flow  whose potential is randomly replaced by another one when $(i(t))_{t \geq 0}$ jumps from one state to another. Both $(\theta_t)_{t \geq 0}$ and $(i(t), \theta_t)_{t \geq 0}$ are called \emph{stochastic gradient process}. Since $\bar\Phi$ often represents the averaging over data subsets, we refer to the process of replacing $\bar\Phi$ by a $\Phi_i$ as \emph{subsampling}.

{SGD has two main advantages:} 
Since we consider only one of the $\nabla \Phi_i$ at a time, we significantly reduce the computational cost when discretising the dynamical system.
Moreover, the perturbation through the randomised sampling in stochastic gradient descent can help to overcome saddle points, see \cite{DanKoh, JinNet}. However, as the method is purely gradient-based, we would expect that the escaping saddle points is still slow.
{Saddle points and local minimisers can sometimes be escaped when using momentum-based optimisation methods.} 
{Here, we replace the gradient flow \eqref{eq_gradient_flow_full}, by the \emph{underdamped gradient flow} given through 
\begin{equation}\label{eq:polyak_ode}
    m\frac{\de^2 q_t}{\de t^2} -\alpha \frac{\de q_t}{\de t} + \nabla \bar\Phi(q_t) =0, \quad \frac{\de q_0}{\de t} \in X, \quad q_0 \in X.
\end{equation}
This dynamical system describes the motion of a ball with mass $m > 0$ constrained to slide on the graph $\bar\Phi$ subject to a constant gravitational field and viscosity friction $\alpha > 0$, see \cite{attouch2000heavy-2} for details.}
Intuitively, momentum is used to vault the optimiser out of saddle points and local minimisers. 


\subsection{Momentum} \label{subsec:mot_exampl}
{We now give a motivating example where the relevance of momentum in non-convex optimisation can be observed and then summarise relevant literature on momentum-based optimisation. The example is based on }
a simple one-dimen\-sional
non-smooth non-convex optimisation problem
inspired by the training of neural networks. 
As we shall see, the introduction of momentum 
can prevent the convergence of the gradient
flow to a stationary point that is not the global 
minimiser. We later work in a framework that does not actually cover this introductory example, we present it only to give an intuition for the methodology.

Let us consider the function
$\Phi:\R\to\R$ defined as 
\begin{equation} \label{eq:1_dim_ex}
\Phi(x) := (\mathrm{ReLU}(x)-1)^2+x^2
= \begin{cases}
x^2+1 &\mbox{if } x\leq 0\\
2x^2-2x+1 &\mbox{if } x> 0
\end{cases} \qquad (x \in \mathbb{R})
\end{equation}
where $\mathrm{ReLU}(\cdot) := \max\{0,\cdot\}$. 
\begin{figure}[htb]
    \centering
    \includegraphics[scale = 0.6]{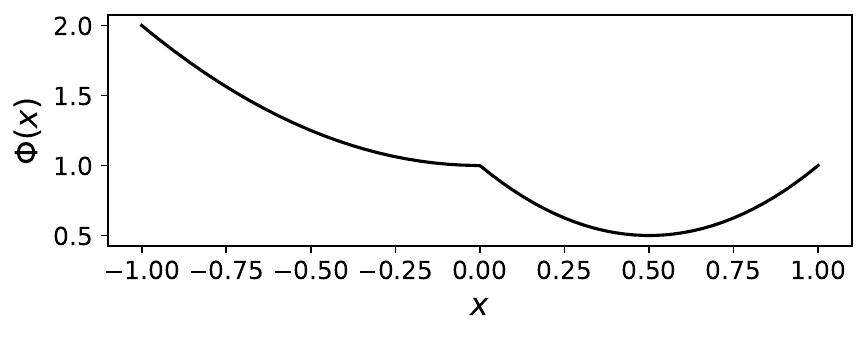}
        \vspace{-3mm}
    \caption{{Plot of the potential $\Phi$.}}
    \label{figure:phi}
\end{figure}
As shown in Figure \ref{figure:phi}, the function $\Phi$ attains its global
minimum at the point $x^* = \frac12$.
However, $\tilde x=0$
is the minimiser of 
$\Phi$ restricted to the 
non-positive half-line $(-\infty, 0]$. Therefore, 
any solution of the gradient flow equation \eqref{eq_gradient_flow_full}
with potential $\bar\Phi := \Phi$ and
{initial value} $\zeta_0<0$ converges to the
{saddle point} $\tilde x$.
In order to avoid the convergence to
$\tilde x$, we can replace the 
gradient flow 
by the associated underdamped gradient flow \eqref{eq:polyak_ode} with mass $m > 0$, viscosity friction $\alpha > 0$, and potential $\bar\Phi$.
If the initial value $q_0$ 
is negative,
then the equation of the motion reduces
to 
\begin{equation} \label{eq:dyn_sys_ex_2}
\mass \frac{\mathrm{d}^2q_t}{\mathrm{d}t^2} + 2q_t = -\alpha \frac{\mathrm{d}q_t}{\mathrm{d}t},
\end{equation}
whenever the dynamical system stays
in the negative half-line.
At this point the natural question is 
whether there exist combinations of
$\mass$ and $\friction$ such that,
assuming that
$q_0<0$, the particle is not confined
in the negative half-line for every
time $t > 0$.
Using elementary theory of linear
second-order differential equations, it turns 
out that the particle manages to overcome the
``false" minimiser $\tilde x = 0$ 
if the following relation is satisfied:
\begin{equation}\label{eq:rel_escape_ex}
    \friction^2-8\mass< 0,
\end{equation}
i.e.\  the friction $\friction$ is sufficiently small or the mass $\mass$ sufficiently large. In summary, momentum can be used to overcome stationary points of the target function.

The use of momentum in optimisation dates back to the 1960s, when Polyak formulated the \emph{heavy ball method} in convex optimisation, see \cite{Pol63,Pol64}. In this context, the aim was to accelerate the convergence of the gradient descent method. {Later, Nesterov  proposed a momentum-based class of discrete accelerated methods for convex problems in his seminal paper \cite{N83}.}
{ Momentum-based optimisation methods have been investigated thoroughly from the perspective of continuous-time dynamical systems. In-depth studies about the mechanical systems underlying Polyak's method can be found in \cite{attouch2000heavy-2} -- a work which has sparked considerable interest in the community. 
We recall, for instance, \cite{attouch2000heavy} which have studied constrained convex optimisation.
In \cite{SBC}, the authors have formulated a dissipative second order ODE that provides a continuous-time counterpart to the Nesterov method. Here, the viscosity friction parameter $\alpha$ is time-dependent and vanishing as $t\to\infty$. This aspect was later further investigated in \cite{ACPR,ShiJord19,ShiJord21,S22}. In \cite{attouch2019scaling-damped, attouch2019scaling-damped-2}, the authors combine vanishing viscosity with time-scaling of trajectories, leading to a dynamics similar to \eqref{eq:polyak_ode}, but with an extra time-dependent coefficient $\beta(t)$ appearing in front of the term $\nabla \bar\Phi$, i.e.
\begin{equation}\label{eq:polyak_scaling}
    \frac{\de^2}{\de t^2}\theta(t) -\alpha(t) \frac{\de}{\de t}\theta (t) + \beta(t) \nabla \bar\Phi(\theta(t)) =0.
\end{equation}
The parameter $\beta$ satisfies $\beta(t)\to\infty$ as $t\to\infty$. 

Finally, we mention more recent advances on second-order dynamics with time scaling that include Hessian-driven damping \cite{attouch2020hessian-rescaling} and averaging \cite{attouch2022scale-averaging}.
}
Next, we see how momentum is used within stochastic optimisation practice.

\subsection{Momentum-based stochastic optimisation and the Adam algorithm}
The most basic momentum-based stochastic optimisation method  combines a {symplectic} Euler discretisation of the underdamped gradient flow \eqref{eq:polyak_ode} with subsampling. This \emph{stochastic gradient descent method with classical momentum} is given by
\begin{equation} \label{eq_SGD_momen}
\begin{cases}
v_n = \alpha v_{n-1} - \eta \nabla \Phi_{i_n}(\theta_{n-1}),\\
\theta_n = \theta_{n-1} + { v_{n}},
\end{cases}
\end{equation}
where $\alpha, \eta > 0$ are hyperparameters and $i_1, i_2,\ldots \sim \mathrm{Unif}(I)$ are independent and identically distributed (i.i.d.) random variables. We refer to $\eta$ as learning rate and refer to \eqref{eq_SGD_momen} just as \emph{classical momentum}. {The connection to the underdamped gradient flow is more obvious when considering the timestep size to be equal to $1$, $\eta$ to be a scaling factor of the potential, and $\alpha$ to be  a parameter that incorporates the velocity at the previous timestep and the friction. Classical momentum has been studied by \cite{Rumelhart1986,pmlr-v28-sutskever13}.}  In deep learning, more advanced momentum-based descent methods have been widely in use, including Adagrad \citep{DHS}; Adadelta \citep{ADADELTA}; RMSprop \citep{RMSprop}; Adam \citep{Adam}, etc. Adam is one of the most commonly used optimisers when training neural networks. The updating rule is the following,
\begin{equation} \label{eq_ADAM}
\begin{cases}
u_n = \beta_1 u_{n-1} + (1-\beta_1)\nabla \Phi_{i_n}(\theta_{n-1}),\\
v_n = \beta_2 v_{n-1} + (1-\beta_2) \big[\nabla \Phi_{i_n}(\theta_{n-1})\big]^2,\\
\theta_n = \theta_{n-1} - \alpha \frac{u_n/(1-\beta_1^n)}{\sqrt{v_n/(1-\beta_2^n)} + \beta_3},
\end{cases}
\end{equation}
where $\beta_1$, $\beta_2$, and $\alpha$ are hyperparameters, $\beta_3$ is a small constant, and $i_1, i_2,\ldots \sim \mathrm{Unif}(I)$ (i.i.d.). Comparing to the classical momentum method, Adam uses not only an estimate of the first moment of the gradient but also an estimate of the second moment as the part of the adaptive learning rate.  
The variable $u_n$ is the biased first moment, similar to the one used in the classical momentum method. The variable $v_n$ is the biased second moment estimated using the gradient squared. When updating $\theta_n$, Adam normalises $u_n$ and $v_n$ so that they become unbiased estimators. Moreover, it was shown  
in \cite{AdamProof} that Adam converges in the sense of \cite{Zinkevich} with speed $O(\log(t)/\sqrt{t}; t \rightarrow \infty)$.

While momentum-based stochastic optimisation methods are popular in machine learning practice, they are overall rather badly understood; see the discussion in \cite{Liu}. In the present work, we analyse a continuous-time stochastic gradient process with momentum and propose an efficient discretisation technique for this dynamical system. Thus, we improve the understanding of momentum-based stochastic optimisation in a theoretical framework and machine learning practice. Our analysis complements the Langevin-based analyses by \cite{Li2019, shi2022hyperparameters}{, where the authors represent SGD in continuous time as a diffusion process.} Before we summarise our contributions more precisely, we introduce our continuous-time stochastic gradient-momentum process.

\subsection{The stochastic gradient-momentum process}
Throughout this work, we discuss a continuous-time dynamical system that represents stochastic gradient descent with momentum -- the stochastic gradient-momentum process. Due to the continuous-time nature of the system, it can equally represent classical momentum \eqref{eq_SGD_momen} and Adam \eqref{eq_ADAM} (although not the adaptive nature of the latter). 
We obtain this system by combining the underdamped dynamical system \eqref{eq:polyak_ode} with the stochastic gradient process \eqref{eq:SGP_old}. 
We now introduce the \emph{index process}, the process that controls the subsampling.
\begin{definition}[Index process]\label{index}
Let $(i(t))_{t \geq 0}$ be a continuous-time Markov process on $I$ with transition rate matrix $$\mathbf{A}_N=\mathbf{\Gamma}_N-N\gamma  \mathbf{I}_{N}.$$
Here, $\gamma>0$ {denotes the jump rate}, $\mathbf{\Gamma}_N$ is an $N\times N$ matrix whose entries are all equal to $\gamma$, and $\mathbf{I}_{N}$ is the identity matrix. We assume that the initial distribution $\mathbb{P}(i(0) \in \cdot)$ is the uniform distribution on $I$. 
Moreover, when given a value $\nu > 0$ or an appropriate, strictly increasing function $\beta: [0, \infty) \rightarrow [0, \infty)$, we define the rescaled index processes $i^\nu(t) := i(t/ \nu)$ and $i^\beta(t) := i(\beta^{-1}(t))$, for $t \geq 0$. We refer to each of $(i(t))_{t \geq 0}$, $(i^\nu(t))_{t \geq 0}$, and $(i^\beta(t))_{t \geq 0}$ as \emph{index process}.
\end{definition}

The function $\beta^{-1}$ and the scalar $1/\nu$ play the role of a time-dependent and constant learning rate, respectively; see \cite{Jin}. Thus, we sometimes refer to them as learning rate.
Next, we make the following regularity assumptions.
\begin{assumption}\label{asSGPf} Let $\Phi_i \in \mathcal{C}^1(X,\R)$, i.e.\  it is continuously differentiable, and $\nabla_x \Phi_i$ be Lipschitz with Lipschitz constant $L >0$, for $i \in I$. 
\end{assumption}
Throughout this work, we study multiple versions of the stochastic gradient-momentum process and introduce them  in Definitions~\ref{def:mSGPC}, \ref{def:DmSGPC}, and \ref{def:DmSGPD}. 
We now introduce the prototype  of the dynamical systems studied in this paper.
Note that we often choose the more compact form $\mathrm{d}z_t = f(z_t) \mathrm{d}t$ to represent $\frac{\mathrm{d}z_t}{\mathrm{d}t} = f(z_t)$. 
The general \emph{stochastic gradient-momentum process} (SGMP) is given by $(p^\dagger_t, q^\dagger_t, i^\beta(t))_{t \geq 0}$ {satisfying}
\begin{equation}\label{eq_gen_SGM}
\left\{ \begin{array}{rl} 
\mathrm{d}q_t^\dagger &= p_t^\dagger \mathrm{d}t,\\
\e(t) \mathrm{d}p_t^\dagger &= - \nabla \Phi_{i^\beta(t)}(q_t^\dagger)\mathrm{d}t- \alpha p_t^\dagger\mathrm{d}t, \\
p^\dagger(t=0) &= p_0,\\ q^\dagger(t=0) &= q_0,
\end{array} \right.
\end{equation}
{where} $(p^\dagger_t)_{t \geq 0}$ describes the velocity of the particle, $(q^\dagger_t)_{t \geq 0}$ its position, $(\e(t))_{t \geq 0}$  its mass (that may depend on time), and $\alpha>0$  the viscocity friction. The function $\beta$ controls the switching of the index process. In the following, we explain the importance behind the functions $\mass$ and $\beta$.

\begin{remark}[Mass $\mass$] The introduction of momentum allows the particle to be vaulted out of stationary points. An effect that is particularly pronounced if $\mass$ is large or $\alpha$ is small. While this behaviour is convenient when escaping local minimisers and saddle points, it can also occur when approaching global minimisers. We give a very simple example in Figure~\ref{figure:losing_momentum}, where we see that the method shows oscillatory behaviour and very slow convergence when the mass $\mass$ is large and constant. As a tuning of the mass $\mass$ is difficult in practice, we suggest an alternative strategy: reducing the mass over time. Rationale: a large mass in the beginning leads to a fast escaping of local minimisers, 
while a small mass later leads to fast convergence once the global minimiser is reached. This intuition is confirmed in Figure~\ref{figure:losing_momentum}.
\end{remark}

\begin{figure}[htb]
    \centering
    \includegraphics[scale = 0.65]{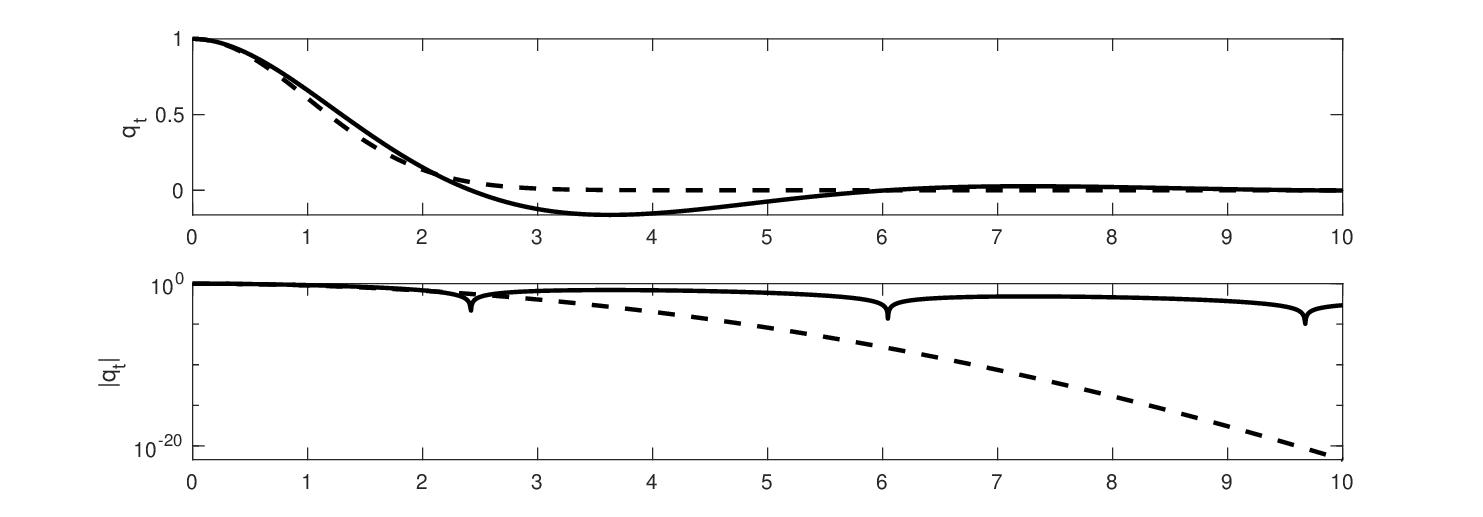}
    \caption{We consider the minimisation of $\bar\Phi(\theta) := \theta^2/2$. Let $\alpha = 1, \e = 1$, $p_0 = 1$, $q_0 = 0$. Then, $q_t = \exp(-t/2)\left( \sin(\sqrt{3}t/2)/\sqrt{3} + \cos(\sqrt{3}t/2)\right)  (t \geq 0)$ (solid line)
oscillates around the solution and converges ultimately to $\theta^* = 0$. If we choose instead $\mass(t) = (1+t)^{-1}$ $(t \geq 0)$ (dashed lines), we have $q_t = \exp(-t^2/2)$, which does not oscillate, but converges very quickly to $\theta^*$.}
    \label{figure:losing_momentum}
\end{figure}
{In \cite{attouch2000heavy-2}, the authors have considered the singular perturbation of \eqref{eq:polyak_ode}, corresponding to letting the mass $m\to 0$, and they established convergence of these dynamics to the gradient flow trajectories. In the present work, we also consider a vanishing mass and carefully bound the Euclidean distance between underdamped gradient flow and gradient flow. Moreover, we let $m$ depend on the time, and investigate the case $m(t)\to 0$ as $t\to\infty$. This setting is conceptually similar to the dynamics studied in  \cite{attouch2019scaling-damped,attouch2019scaling-damped-2} that we give in \eqref{eq:polyak_scaling}. However, $\beta(t) \rightarrow \infty$ as $t \rightarrow \infty$ is not equivalent to the mass going to zero. We can rewrite the underdamped gradient flow to
\begin{equation} \label{eq:our_ode}
    \frac{\de^2}{\de t^2}\theta(t) -\frac{\alpha}{m(t)} \frac{\de}{\de t}\theta (t) +  \underbrace{\frac{1}{m(t)}}_{=\beta(t)} \nabla \Phi(\theta(t)) =0,
\end{equation}
where the factor multiplied by the velocity explodes as well, as $m(t) \rightarrow 0$. Thus, we cannot rely on the results obtained in \cite{attouch2019scaling-damped,attouch2019scaling-damped-2} for the dynamics \eqref{eq:polyak_scaling}.
  Moreover, we shall not actually focus our attention on the deterministic evolution, but we rather investigate a stochastic process related to the ODE \eqref{eq:our_ode}.
However, an interesting similarity between our approach and \cite{attouch2019scaling-damped,attouch2019scaling-damped-2} arises when discretising the respective dynamics. Indeed, the fact that $\beta(t)\to\infty$ in \eqref{eq:polyak_scaling} and $m(t)\to0$ in \eqref{eq:our_ode} make explicit schemes not suitable for deriving discrete approximations. In \cite{attouch2019scaling-damped,attouch2019scaling-damped-2} this issue is addressed by considering a \emph{proximal} discretisation, while in Section~\ref{sec:discr_meth} we opt for a symplectic (semi-implicit) scheme. }

\begin{remark}[Learning rate and $\beta$]
The switching between the potentials leads to a similar effect as the momentum. As the flow approaches the minimisers of different target functions in each step, we are unlikely to converge to a single point. While this helps to escape local minimisers, it prohibits the convergence to a global minimiser. In several stochastic optimisation methods, we need to reduce the learning rate or step-size throughout the algorithm to reduce the time in-between  switches of data sets. We can {attain} this reduction of learning rate by rescaling the time in the index process $(i(t))_{t \geq 0}$ through an appropriate function $\beta$. The rescaling will lead to the waiting times between two switches to decrease over time.
\end{remark}

{We now list and explain more assumptions that are required throughout this work. Indeed, some results of this work require the $(\Phi_i)_{i \in I}$ to be \emph{($\mu$-)strongly convex}, i.e.\  \begin{align*}
    (x-x') \cdot(\nabla \Phi_i(x)-\nabla \Phi_i(x'))\ge \mu\norm{x-x'}^2 \qquad (x, x' \in X, i \in I),
\end{align*} for some strong convexity parameter $\mu >0$. Others do not rely on strong convexity, but still need some restriction on the potential function, as given in the following assumption.} 
{
\begin{assumption}\label{as1.2}
    Let $\Phi_i \in \mathcal{C}^1(X,\R)$ admit a global minimiser $\theta^i_*$, for $i \in I$, and let there be constants $\lambda\in(0, 1/4]$ and $\bar\alpha>0$ such that 
\begin{equation}\label{as:phi}
(x-\theta^i_*)\cdot \nabla \Phi_i(x)/2\ge \lambda (\Phi_i(x)-\Phi_i(\theta^i_*)+\bar\alpha^2\norm{x-\theta^i_*}^2/4),\qquad  (x\in X, i \in I).
\end{equation}
\end{assumption}
We use $\mathcal{A}^{\lambda, \bar\alpha}$ as a short hand to refer to (\ref{as:phi}) with parameters $\lambda$ and $\bar\alpha$.}
{We note that Assumption \ref{as1.2} is also known as the quasi-strong convexity, see, e.g.\ \cite{doi:10.1137/21M1403990, EGZ, Necoara2015LinearCO}.}
    Assumption \ref{as1.2} is implied by strong convexity. We prove this assertion in Lemma~\ref{lem:conx} in the appendix. In the appendix, we also give a counterexample showing that the two assumptions are not equivalent, see Example \ref{example:assp2}.

In stochastic optimisation, it is often necessary to decrease the learning rate over time to obtain a method that converges to the minimiser. In practice, however, this is sometimes purposefully disregarded -- convergence to, e.g.\ a probability distribution is preferred. In stochastic optimisation this leads to an \emph{implicit regularisation} of the optimisation problem. Particularly in machine learning, this implicit regularisation can be necessary to get good generalisation results, see, for example, \cite{Kale, Keskar, Neyshabur, Zhang, Zou}. Thus, we also study this constant learning rate case in the present work. Similarly, whilst we have seen that a decreasing mass can be beneficial, we also study the case of a constant mass.
Indeed, when the mass $\mass$ does not depend on $t$ or the function $\beta$ is linear, i.e.\  $\beta(u) = \nu u$, we refer to mass or learning rate as \emph{homogeneous}, respectively. Since we are otherwise always interested in reducing them over time, we speak of \emph{decreasing} mass or learning rate, respectively.
When decreasing mass and learning rate, an additional assumption -- Assumption \ref{comass} -- is needed. Since this assumption is rather technical, we delay its statement and discuss it instead along with the associated convergence result in Subsection~\ref{sec_losingboth}. 

In order to establish well-posedness of the system, the mass $(m(t))_{t\geq0}$ must not decrease too quickly. Here, we have the following assumption.
\begin{assumption}\label{as1.21}
There exists a constant $\lambda \in (0,1]$, such that $\abs{\e'(t)}\le \lambda\e(t)$, $t>0$. Moreover, we assume that $\e(0)=:\e_0 \le 1$.
\end{assumption}

    It is easy to verify that $\e(t)= \frac{\e_0}{\lambda^{-1}+t}$ and  $\e(t)= \e_0e^{-\lambda t}$ satisfy Assumption~\ref{as1.21}. Setting $m_0 = \lambda = 1$, we retrieve the example in Figure~\ref{figure:losing_momentum}. More generally, Assumption \ref{as1.21} implies  $\e(t)\ge \e_0e^{-\lambda t}$. Hence, the mass decreases at most exponentially with a rate $\geq -1$.

\subsection{Contributions {and main results}}
We have introduced the stochastic gradient-momentum process above. In this work, we analyse this dynamical system, discuss its discretisation, and employ it in numerical experiments. 
We are interested in both, convergence to global minimisers and implicit regularisation. Thus, we study three different stochastic gradient-momentum processes:
(1) homogeneous mass and homogeneous learning rate,
(2) decreasing mass and homogeneous learning rate, and
(3) decreasing mass and  decreasing learning rate, respectively.
    We now summarise our contributions:
   \begin{itemize}
       \item  We study the  connection between SGMP and the gradient flow, the underdamped gradient flow, as well as the stochastic gradient process. These connections allow us to investigate the longtime behaviour of the stochastic gradient-momentum process.
       \end{itemize}
       We obtain those connections by considering  momentum-to-no-momentum asymptotics (letting a homogeneous mass $m$ approach $0$), and random-to-deterministic asymptotics (for a homogeneous $\beta(u) = \nu u$, letting $\nu$ approach $0$). 
   {Understanding the longtime behaviour eventually allows us to study the convergence of the underlying algorithm.}
We summarise our three main results below and also give informal statements of the corresponding theorems -- the complete versions are given later in the text.
    \begin{itemize}
        \item In the fully homogeneous case, we show that the stochastic gradient-momentum process is close to the stochastic gradient process when the mass is small. {This can be understood as the stochastic version of the deterministic result established in \cite{attouch2000heavy-2}.}
        
        {
        \textbf{Theorem \ref{second th} (informal)} \textit{Let $(q_t^\e)_{t \geq 0}$ be the stochastic gradient-momentum process (\ref{eq:AS:pq}) with constant mass $m$ and $(\theta_t^\e)_{t \geq 0}$  be the stochastic gradient process (\ref{Eq_SGPCed}). Under Assumptions \ref{asSGPf} and \ref{as1.2}, we have
        $$ \E\Big[\sup_{0\le t\le T}\norm{q^\e_t-\theta^\e_t}   \Big]\to 0 \ \ \ \ (m\downarrow0). $$}}
        
        \item If the mass decreases over time, the stochastic gradient-momentum process converges to the stochastic gradient process in the longtime limit.

        {
        \textbf{Theorem \ref{thm:dm2p} (informal)} \textit{
        Let $(\tilde q_t)_{t \geq 0}$ be the decreasing-mass stochastic gradient-momentum process (\ref{eq:AS:pqeiet})  with time-dependent momentum $(m(t))_{t\geq0}$ and $(\theta_t)_{t \geq 0}$ be the stochastic gradient process (\ref{eq:SGP_old}). Under strong convexity and Assumptions \ref{asSGPf} and \ref{as1.21}, we have, almost surely and in expectation,
        $$  \norm{\tilde q_t-\theta_t}\to 0 \ \ \ \ (t\to\infty). $$
        }}
        
        \item If mass and learning rate decrease over time, the stochastic gradient-momentum process converges to the minimiser of the full target function in the longtime limit.
        
        {
        \textbf{Corollary \ref{cor:convmi} (informal)} \textit{Let $(\hat q_t)_{t \geq 0}$ be the decreasing-mass, decreasing-learning-rate stochastic gradient-momentum process with time-dependent mass $(m(t))_{t \geq 0}$ and learning rate $(\beta(t))_{t \geq 0}$. Let $\theta_*$ be the minimiser of the full target function.  Under strong convexity and Assumptions \ref{asSGPf}, \ref{as1.21}{, and \ref{comass},} we have
        $$\norm{\hat q_t-\theta_*}  \overset{\text{in probability}}{\xrightarrow{\hspace*{1.5cm}}} 0 \ \ \ \ (t\rightarrow\infty).$$}} 
    \end{itemize}
    The connection to the stochastic gradient process is especially interesting, as its longtime behaviour in convex settings is well-understood, see \cite{Jin, Jonas}. Thus, we can argue that SGMP leads to a similar implicit regularisation when the mass is small or converging to zero. 
In a wider sense, the connections to the other methods allow us to interpret the stochastic gradient-momentum process as a method that can freely interpolate between gradient flow, underdamped gradient flow, and stochastic gradient process -- through adjusting learning rate $\nu^{-1}$ and mass $m$. We depict this relationship in Figure~\ref{fig_tikz_SGMP}. 
   \begin{figure}[htb]
       \centering
       \begin{tikzpicture}
       \draw[->,very thick] (0,0) -- (0,5.25) node[anchor=north east] {mass $m$};
       \draw[->,very thick] (0,0) -- (5.25,0) node[anchor=west] {learning rate $\nu$};
      \draw[pattern=north west lines, pattern color=gray] (0.6,0.6) rectangle (4.75,4.75);
      \draw[pattern=crosshatch, pattern color=gray] (-1.25,-1.25) rectangle (0.4,0.4);
      \draw[pattern=vertical lines, pattern color=gray] (-1.25,0.6) rectangle (0.4,4.75);
      \draw[pattern=horizontal lines, pattern color=gray] (0.6,-1.25) rectangle (4.75,0.4);

      \node[fill=white!80,text width= 2.4cm, rotate = 0] at (2.625,2.625) {\tiny Stochastic Gradient-Momentum Process} ;
       \node[fill=white!80,text width= 3.4cm] at (2.75,-0.6) {\tiny Stochastic Gradient Process} ;
         \node[fill=white!80,text width= 1.1cm] at (-0.4,-0.6) {\tiny Gradient Flow} ;
                \node[fill=white!80,text width= 3.4cm, rotate = 90] at (-0.6,2.75) {\tiny Underdamped Gradient Flow} ;

       \end{tikzpicture}
       \caption{Schematic comparing gradient flow, underdamped gradient flow, stochastic gradient process, and stochastic gradient-momentum process in terms of the particle mass $m$ and learning rate $\nu$.}
       \label{fig_tikz_SGMP}
   \end{figure}

Surprisingly, the continuous relationship shown in Figure~\ref{fig_tikz_SGMP} is lost when, e.g.\ comparing the discrete-time algorithms SGD and classical momentum. This is mainly due to the symplectic Euler discretisation that is employed in classical momentum. Especially problematic is the instability of classical momentum when the mass is small. The step-size needs to be chosen as $O(m; m \downarrow 0)$. From this computational perspective, our contributions are the following:
\begin{itemize}
\item We propose a discretisation strategy for the stochastic gradient-momentum process. The strategy is a semi-implicit method that is explicit in the position $(q_t)_{t \geq 0}$ and  implicit in the velocity $(p_t)_{t \geq 0}$. This discretisation technique allows us to choose the step-size independently of the mass $m$. 
{Thus, the method is stable for small masses as well as masses converging to zero.}
{Similarly to the symplectic Euler method that is conventionally used in momentum-based optimisation, our scheme is symplectic. Thus, it respects (local) conservation laws in the underdamped dynamics.} 

\item We test our discretised algorithm in numerical experiments. We start with  academic convex and non-convex optimisation problems in low and high dimensions. Then we study the training of a convolutional neural network (CNN), which we use to classify images from the CIFAR-10 dataset. The achieved train and test accuracy with our stable discretisation is comparable to stochastic gradient descent with classical momentum. Some of these experiments go beyond the convex setting that we study in our analysis.
   \end{itemize} 
This work is organised as follows. We introduce and investigate the homogeneous-in-time and heterogeneous-in-time stochastic gradient-momentum processes in  Sections~\ref{Sec_Homogene} and \ref{sec:heterogeneous}, respectively. In Section~\ref{sec:discr_meth} we propose discretisation techniques  which we then employ  in the mentioned academic and deep learning problems in Section~\ref{sec:appl}. A conclusion is given in Section~\ref{sec_conc}. Throughout the main part of the work, we require certain results about underdamped gradient flows, which we discuss in Appedices \ref{appendix:underdamped} and \ref{appendix_decreasing_momentum}. Other auxiliary results needed throughout this work, are presented  in  Appendix~\ref{appendix}. Appendix~\ref{appendix:kushner} recalls and summarises a result by \cite{Kushner1}.

\section{Homogeneous-in-time} \label{Sec_Homogene}
In this section, we study the stochastic gradient-momentum process with homogeneous momentum and learning rate. In the first part, Subsection~\ref{sec_momentum_w/w/o_subsamp}, we investigate the interplay between the underdamped gradient flow and the stochastic gradient-momentum process. In the second part, Subsection~\ref{sec_subsamp_w/wo/momentum}, we compare the stochastic gradient-momentum process and the stochastic gradient process. 
We first introduce the dynamical system central to this section. 

\begin{definition} \label{def:mSGPC}
The \emph{homogeneous stochastic gradient-momentum process (hSGMP)} is a solution of the following stochastic differential equation,
\begin{equation}\label{eq:AS:pqe}
\left\{\begin{array}{rl}
\de q^{\nu,m}_t &= p^{\nu,m}_t \de t,\\
m\de p^{\nu,m}_t &= - \nabla \Phi_{i^{\nu}(t)}(q^{\nu,m}_t)\de t-\alpha p^{\nu,m}_t\de t, \\
p^{\nu,m}_0 &= p_0 \in X,\\ 
q^{\nu,m}_0 &= q_0 \in X,
\end{array}\right.
\end{equation}
where $\alpha, m, \nu>0$ are
{the} viscosity friction
{parameter}, {the} mass, and the learning rate, respectively -- {all kept constant.}
{Finally,} $\Phi_j$ satisfies Assumption \ref{asSGPf}, for every {$j \in I$} and the stochastic process $i^\nu(t)=i(t/\nu)$ is given as in Definition \ref{index}. In the following, we denote
\begin{itemize}
    \item[(i)] $(p^\nu_t, q^\nu_t)_{t \geq 0} := (p^{\nu,1}_t, q^{\nu,1}_t)_{t \geq 0}$, i.e.\  choosing (without loss of generality) a constant unit mass $m := 1$ and varying the learning rate $\nu$.
\item[(ii)] $(p^m_t, q^m_t)_{t \geq 0} := (p^{m^\delta,m}_t, q^{m^\delta,m}_t)_{t \geq 0}$, for some $\delta \in [0,1)$, i.e.\  varying the mass, and choosing a constant or a mass-depending learning rate $\nu := m^\delta$; in this case, we additionally choose (without loss of generality) a unit viscosity friction $\alpha := 1$.
\end{itemize}
\end{definition}
The Lipschitz condition in Assumption~\ref{asSGPf} guarantees the well-posedness of this system, which can be shown similarly as in \cite{Jin,Jonas}. In Subsections~\ref{sec_momentum_w/w/o_subsamp} and \ref{sec_subsamp_w/wo/momentum}, we consider the hSGMPs $(p^\nu_t, q^\nu_t)_{t \geq 0}$ and $(p^m_t, q^m_t)_{t \geq 0}$, respectively. 

\subsection{Momentum with and without subsampling} \label{sec_momentum_w/w/o_subsamp}
First, we study  the interplay between subsampling and not subsampling in the momentum dynamics. Indeed, we show that the stochastic gradient-momentum process can approximate the underdamped gradient flow at any accuracy. More precisely, we are interested in the limiting behavior of hSGMP $(p^\nu_t, q^\nu_t)_{t \geq 0}$  as the learning rate approaches zero uniformly, i.e.\  we take $\nu\to0$  in \eqref{eq:AS:pqe}. We prove that the hSGMP converges to the solution to the underdamped gradient flow, which is given by 
\begin{equation}\label{eq:AS:pq}
\left\{ \begin{array}{rl}
\de q_t &= p_t \de t,\\
\de p_t &= - \nabla \bar\Phi(q_t) \de t-\alpha p_t \de t, \\
p_0, q_0 &\in X,
\end{array} \right.
\end{equation}
where we implicitly set the mass $m := 1$. In this case, we show that hSGMP is a stochastic approximation to the underdamped gradient flow.

We now formulate the statement about the convergence of hSGMP to the underdamped gradient flow more particularly.
In principle, we show convergence of a random path to a deterministic path -- both are contained in the space of continuous functions from $[0,\infty)$ to $X^2$ which we denote by $\mathcal{C}([0,\infty),X^2)$ and equip with the metric $$
\rho\Big((\zeta_t)_{t\ge 0},(\xi_t)_{t\ge 0}\Big):= \int_0^\infty e^{-t} \left(1\land\sup_{0\le s\le t}\norm{\zeta_s-\xi_s}\right)\de t,
$$  
where as usual $a \land b := \min\{a,b\}$ for $a, b \in \mathbb{R}$.
Probabilistically, we show convergence in the weak sense. We state and prove the convergence result below.

\begin{theorem}\label{wcovpq}
 Let $(q^\nu_t,p^\nu_t)_{t\ge 0}$ and $(q_t,p_t)_{t\ge 0}$  solve (\ref{eq:AS:pqe}) and (\ref{eq:AS:pq}). Then $(q^\nu_t,p^\nu_t)_{t\ge 0}$  converges weakly to $(q_t,p_t)_{t\ge 0}$ in  $\mathcal{C}([0,\infty),X^2)$ as $\nu\to 0$, i.e.\ 
 for any bounded continuous function $F: \mathcal{C}([0,\infty),X^2) \rightarrow \mathbb{R}$,
 $$\E [F\big((p^\nu_t,q^\nu_t)_{t\ge 0}\big)] \to \E [F\big((p_t,q_t)_{t\ge 0}\big)]=F\big((p_t,q_t)_{t\ge 0}\big).$$
\end{theorem}
\begin{proof}
We apply Theorem 4.3 in \citet[Chapter 7]{Kushner1}, which studies tightness and weak convergence properties of solutions of a certain type of stochastic differential equations. The form of the differential equation is given in (2.2) \citep[Chapter 7]{Kushner1}, we also recall this result in Appendix~\ref{appendix:kushner}. This corresponds to equation (\ref{eq:AS:pq}) with $\xi_t=i(t)$, $\bar{G}(x,y) = (y,-\nabla_x \bar\Phi(x)-\alpha y)$, $\tilde{G}(x,y,i) = (0,- \nabla_x \Phi_i(x) + \nabla_x \bar\Phi(x))$, and $F = 0$. In \citet[Chapter 7, Theorem 4.3]{Kushner1}, the differential operator $A$ in our case is defined in (\ref{eq:AS:pqe}), i.e.\  $Ah(x,y) = - \nabla h(x,y) \cdot (y,-\nabla_x \bar\Phi(x)-\alpha y)$ for any $h$ twice differentiable. If Assumptions (A4.2) to (A4.6) from \cite{Kushner1} were satisfied, applying Theorem 4.3 in \cite[Chapter 7]{Kushner1} would imply that $(q^\nu_t,p^\nu_t)_{t\ge 0}$ is tight and the limit of any weakly convergent subsequence solves (\ref{eq:AS:pqe}). Since $q^\nu_0 = q_0$ and $p^\nu_0 = p_0$, we have $(q^\nu_t,p^\nu_t)_{t\ge 0}$ converges weakly to $(q_t,p_t)_{t\ge 0}$.
Therefore, we just need to verify Assumptions (A4.2) to (A4.6) from \cite{Kushner1}. Assumption (A4.2) follows directly from Assumption~\ref{asSGPf}. Assumption (A4.3) holds since $(i(t))_{t \geq 0}$ is c\`adl\`ag and bounded. Assumptions (A4.4) and (A4.6) trivially hold since $F = 0$. The only non-trivial part is to verify Assumption (A4.5), that is as $t,\tau\to \infty$,
\begin{align*}
    \frac{1}{\tau}\int_t^{t+\tau} \E[- \nabla_x \Phi_{i(s)}(x) + \nabla_x \bar\Phi(x)|(i(s'))_{s'\le t}] \de s \to 0,
\end{align*}
almost surely.
By the Markov property of $(i(t))_{t\ge0}$ and since $i(0)$ is stationary,
\begin{align*}
   \Big| \frac{1}{\tau}\int_t^{t+\tau} \E[- \nabla_x &\Phi_{i(s)}(x) + \nabla_x \bar\Phi(x)|\{i(s')\}_{s'\le t}] \de s\Big|\\
   =& \Big|\frac{1}{\tau}\int_0^{\tau} \E_{i(t)}[- \nabla_x \Phi_{i(s)}(x) + \nabla_x \bar\Phi(x)] \de s\Big|\\
   =& \Big|\frac{1}{\tau}\int_0^{\tau} \E_{i(t)}[ \nabla_x \Phi_{i(s)}(x) - \nabla_x \bar\Phi(x)] \de s\Big|\\
    \le& \sum_{k=1}^N\Big|\frac{1}{\tau}\int_0^{\tau} \E_k[ \nabla_x \Phi_{i(s)}(x) - \nabla_x \bar\Phi(x)] \de s\Big|\\
    =& \frac{1}{\tau}\sum_{k=1}^N\Big|\int_0^{\tau} \Big(\sum_{j=1}^N \frac{1-\exp(-Ns)}{N}\nabla_x \Phi_j(x)\Big)+\exp(-Ns)\Phi_k(x)- \nabla_x \bar\Phi(x)\de s\Big|\\
    = & \frac{1}{\tau}\sum_{k=1}^N\Big|\int_0^{\tau} \exp(-Ns)\Big(\nabla_x\Phi_k(x)- \nabla_x \bar\Phi(x)\Big)\de s\Big|\\
    \le& \frac{1}{N\tau}\sum_{k=1}^N \abs{\nabla_x\Phi_k(x)- \nabla_x \bar\Phi(x)}\to 0.
\end{align*}
as $\tau\to \infty$.
\end{proof}
Hence, hSGMP $(q^\nu_t,p^\nu_t)_{t\ge 0}$ is a stochastic approximation of the underdamped gradient flow  $(q_t,p_t)_{t\ge 0}$. 
{
Assumption \ref{as1.2} on the other hand implies that the position of the underdamped gradient flow $q_t$ converges to the minimiser of $\bar\Phi$ as $t \rightarrow \infty$. This already gives us reason to believe that an SGMP is able to identify the minimiser of $\bar\Phi$. We discuss the longtime behaviour of the underdamped gradient flow in detail in Appendix \ref{appendix:underdamped}.}

\subsection{Subsampling with and without momentum} \label{sec_subsamp_w/wo/momentum}
In a similar way to the last section, we now study the relation between the homogeneous stochastic gradient-momentum process and the stochastic gradient process. Indeed, we let the momentum in hSGMP go to zero by letting the mass $\e$ go to zero. We show that in the limit, we  converge to the stochastic gradient process. We do this while either keeping a constant learning rate or while allowing the learning rate to depend on the mass and thus go to zero as well. Specifically, we consider the process $(p^m_t, q^m_t)_{t \geq 0}$ introduced in Definition~\ref{def:mSGPC}(ii). 
Namely, we set $\delta \in [0,1)$, learning rate $\nu := m^\delta$, and $\alpha :=1$.
We then aim to show that $(q_t^{i,m})_{t \geq 0}$ in \eqref{eq:AS:pq} converges uniformly to some stochastic gradient process as $\e \rightarrow 0$. The associated stochastic gradient process is the following:
\begin{equation} \label{Eq_SGPCed}
\left\{ \begin{array}{rl}
    \de  \theta^\e_t &= - \nabla \Phi_{i(t/m^\delta)}(\theta^\e_t) \de t,\\
    \theta^\e_0 &=\theta_0.
    \end{array} \right.
\end{equation}

Throughout our discussion, we need several auxiliary results regarding momentum-to-no-momentum limit in the deterministic case. We have collected these results, that may be of independent interest, in Appendix~\ref{appendix_decreasing_momentum}.
 We establish the connection between $(q_t^\e)_{t \geq 0}$ and $(\theta_t^\e)_{t \geq 0}$ in the following theorem by iterating over jump times of the index process. In each of the iteration steps, we employ a result regarding the underdamped gradient flow that we summarise and prove in Proposition~\ref{inductionlem} in Appendix~\ref{appendix_decreasing_momentum}.
\begin{theorem}\label{second th}  {Let $\Phi_i$ be convex and satisfy Assumptions \ref{asSGPf} and \ref{as1.2} with $\mathcal{A}^{\lambda_i, 1}$ and critical point $\theta^i_*$ for $i\in I$.} Let $(p^{\e}_t,q^{\e}_t)_{t \geq 0}$ and $(\theta_t^\e)_{t \geq 0}$ solve (\ref{eq:AS:pq}) and (\ref{Eq_SGPCed}), respectively. 
For $0\le\delta<1$ and $T>0$, we have
\begin{align*}
 \E\Big[\sup_{0\le t\le T}\norm{q^\e_t-\theta^\e_t}   \Big]\le \Big(1+C'_L \e(1+T)\Big)e^{\e^{1-\delta}T(1+T)\gamma N}\Big(\norm{q_0-\theta_0}+\e\norm{p_0}\Big)\\
    +\Big(e^{\e^{1-\delta}T(1+T)\gamma N}-1+C'_L\e (T+1)\Big) K_{\Phi,T,\theta_0},
\end{align*}
where $C'_L, K_{\Phi,T,\theta_0}>0$ are constants.
\end{theorem}
\begin{proof}
We now denote by $(\tau_n)_{n\ge 0}$ and $(\tau^{\e,\delta}_n)_{n\ge 0}$ the sequences of the jump times of processes $(i(t))_{t \geq 0}$ and $(i^{\e,\delta}(t))_{t \geq 0}$, respectively. We set $\tau_0=0.$  From the definition, we know that the last jump time before $t$ is $\tau_{N_t},$ where $(N_t)_{t \geq 0}$ is a Poisson process with rate $\gamma N.$ And since $i^{\e,\delta}(t)=i(t/\e^{\delta})$, we have $\tau^{\e,\delta}_n= \e^\delta\tau_n .$ By Proposition \ref{inductionlem}, for $\tau^{\e,\delta}_n<t\le \tau^{\e,\delta}_{n+1},$ we have
\begin{align*}
    \norm{q^\e_t-\theta^\e_t}\leq& \norm{q^\e_{\tau^{\e,\delta}_n}-\theta^\e_{\tau^{\e,\delta}_n}} \\ 
    & \quad
    +C^{(0)}_L\e(1+t-\tau^{\e,\delta}_n) \Big(\norm{q^\e_{\tau^{\e,\delta}_n}-\theta^\e_{\tau^{\e,\delta}_n}}+\norm{p^\e_{\tau^{\e,\delta}_n}}+\norm{\theta^\e_{\tau^{\e,\delta}_n}-\theta^{i^{\e,\delta}(\tau^{\e,\delta}_n)}_*}\Big)
\end{align*}
and from the first inequality in Lemma \ref{lem:bound+Lip}, we obtain
\begin{align*}
    \norm{p^\e_t}\le (1+C_L\e)\norm{p^\e_{\tau^{\e,\delta}_n}}+C_L  \norm{q^\e_{\tau^{\e,\delta}_n}-\theta^\e_{\tau^{\e,\delta}_n}}+\norm{\theta^\e_{\tau^{\e,\delta}_n}-\theta^{i^{\e,\delta}(\tau^{\e,\delta}_n)}_*}.
\end{align*}
     From Lemma \ref{thetabound}, we know
     \begin{align*}
         \norm{\theta^\e_{\tau^{\e,\delta}_n}-\theta^{i^{\e,\delta}(\tau^{\e,\delta}_n)}_*}\le \norm{\theta^\e_{\tau^{\e,\delta}_n}}+\norm{\theta^{i^{\e,\delta}(\tau^{\e,\delta}_n)}_*} \le \norm{\theta_0}+C_\Phi T+K_{\Theta_*}:=K_{\Phi,T,\theta_0}.
     \end{align*}
     So we have
     \begin{align*}
         &\sup_{0\le t\le \tau^{\e,\delta}_{n+1}}\norm{q^\e_t-\theta^\e_t}\le \max\Big\{\sup_{0\le t\le \tau^{\e,\delta}_n}\norm{q^\e_t-\theta^\e_t},\sup_{\tau^{\e,\delta}_n\le t\le \tau^{\e,\delta}_{n+1}}\norm{q^\e_t-\theta^\e_t}\Big\}\\
         &\le \sup_{0\le t\le \tau^{\e,\delta}_n}\norm{q^\e_t-\theta^\e_t}
    +C^{(0)}_L\e(1+\tau^{\e,\delta}_{n+1}-\tau^{\e,\delta}_n) \Big(\norm{q^\e_{\tau^{\e,\delta}_n}-\theta^\e_{\tau^{\e,\delta}_n}}+\norm{p^\e_{\tau^{\e,\delta}_n}}+K_{\Phi,T,\theta_0}\Big).
     \end{align*}
For $T>0,$ if we assume $\tau^{\e,\delta}_{n+1}\le T,$ we have 
\begin{align*}
     \sup_{0\le t\le \tau^{\e,\delta}_{n+1}}\norm{q^\e_t-\theta^\e_t}  \le & \sup_{0\le t\le \tau^{\e,\delta}_n}\norm{q^\e_t-\theta^\e_t}\\
    &+ C^{(0)}_L\e(1+T)\Big(\sup_{0\le t\le \tau^{\e,\delta}_n}\norm{q^\e_t-\theta^\e_t}+\sup_{0\le t\le \tau^{\e,\delta}_n}\norm{p^\e_t}+K_{\Phi,T,\theta_0}\Big)
\end{align*}
and
\begin{align*}
    \sup_{0\le t\le \tau^{\e,\delta}_{n+1}}\norm{p^\e_t}   \le 
     (1+C_L\e)\sup_{0\le t\le \tau^{\e,\delta}_n}\norm{p^\e_t}
    +C_L \Big( \sup_{0\le t\le \tau^{\e,\delta}_n}\norm{q^\e_t-\theta^\e_t}+K_{\Phi,T,\theta_0}\Big).
\end{align*}
We denote
\begin{align*}
A^{\e,\delta}_n:=\sup_{0\le t\le \tau^{\e,\delta}_n}\norm{q^\e_t-\theta^\e_t}   ,\ 
    B^{\e,\delta}_n:=\sup_{0\le t\le \tau^{\e,\delta}_n}\norm{p^\e_t},\ D^{\e,\delta}_n=A^{\e,\delta}_n+\e  B^{\e,\delta}_n.
\end{align*}    
and then rewrite the inequalities  
\begin{align*}
    &A^{\e,\delta}_{n+1}\le A^{\e,\delta}_n+ \e C^{(0)}_L(1+ T)\Big(A^{\e,\delta}_n+B^{\e,\delta}_n+K_{\Phi,T,\theta_0}\Big),\\
    &B^{\e,\delta}_{n+1}\le (1+C_L\e)B^{\e,\delta}_n+C_L\Big(A^{\e,\delta}_n+K_{\Phi,T,\theta_0}\Big).
\end{align*}
Since $\e <1$, there exists a constant  $C'_L$ such that
\begin{align}\label{theorem3 D}
    D^{\e,\delta}_{n+1}
    \le \Big(1+C'_L \e(1+T)\Big)D^{\e,\delta}_n+C'_L  \e(1+T) K_{\Phi,T,\theta_0}.
\end{align}
Hence, for $ n\ge 0$, we have
\begin{align}\label{indudn}
     D^{\e,\delta}_n\le \Big(1+C'_L \e(1+T)\Big)^n\Big(\norm{q_0-\theta_0}+\norm{p_0}\Big) +\Big(\Big(1+C'_L \e(1+T)\Big)^n-1\Big)K_{\Phi,T,\theta_0}.
\end{align}
Let $n= N_{T/\e^\delta}$. From (\ref{indudn}) and by using the moment-generating function of the Poisson distribution, we conclude:
\begin{align*}
    \E[D^{\e,\delta}_{N_{T/\e^\delta}}]
    &\le \Big(\norm{q_0-\theta_0}+m\norm{p_0}\Big)\E \Big[\Big(1+C'_L \e(1+T)\Big)^{N_{T/\e^\delta}}\Big] \\
    &\ \ \ \ \ \ \  +K_{\Phi,T,\theta_0}\E\Big[\Big(1+C'_L \e(1+T)\Big)^{N_{T/\e^\delta}}-1\Big]\\
    &= e^{\e^{1-\delta}T(1+T)\gamma N}\Big(\norm{q_0-\theta_0}+C_L m\norm{p_0}\Big)+ \Big(e^{\e^{1-\delta}T(1+T)\gamma N}-1\Big)K_{\Phi,T,\theta_0}.
\end{align*}
For any $T\ge 0,$ there exists an integer $n\ge 0,$ such that $\tau^{\e,\delta}_n\le T\le \tau^{\e,\delta}_{n+1}$ (i.e.\  $n= N_{T/\e^\delta}$). Hence,
\begin{align*}
    \E\Big[\sup_{0\le t\le T}\norm{q^\e_t-\theta^\e_t}   \Big]&\le \E\Big[\sup_{0\le t\le \tau^{\e,\delta}_{n+1}}\norm{q^\e_t-\theta^\e_t}   \Big]
    \le \E[D^{\e,\delta}_{N_{T/\e^\delta}+1}]\\ &\overset{(\ref{theorem3 D})}{\le} \Big(1+C'_L \e(1+T)\Big)\E[D^{\e,\delta}_{N_{T/\e^\delta}}]+C'_L  \e(1+T) K_{\Phi,T,\theta_0}
    \\ &\le \Big(1+C'_L \e(1+T)\Big)e^{\e^{1-\delta}T(1+T)\gamma N}\Big(\norm{q_0-\theta_0}+\e\norm{p_0}\Big)\\
    &\qquad +\Big(e^{\e^{1-\delta}T(1+T)\gamma N}-1+C'_L\e (T+1)\Big) K_{\Phi,T,\theta_0}.
\end{align*}
\end{proof}
Thus, if we let the mass $\e$ go down to zero and the learning rate parameter $\nu$ is independent of $m$ or $\nu$ goes to zero at a certain rate that is slower than $m$, the stochastic gradient-momentum process  converges to the stochastic gradient process. The speed of convergence is linear in $m$, which is also the case in the deterministic setting.  Note that we show convergence in a possibly weaker sense as compared with the learning rate result in Theorem~\ref{wcovpq}. 

We finish our discussion of the homogeneous-in-time setting with a final remark regarding reducing mass and learning rate at the same time.
\begin{remark}
From \citet[Theorem 1]{Jonas} {and \citet{Jonas_correction}}, we already know that $$(\theta_t^m)_{t \geq 0} \rightarrow (\zeta_t)_{t \geq 0}$$ if $m \rightarrow 0$, i.e.\  the stochastic gradient process converges to the gradient flow as the learning rate approaches zero. Convergence is in the weak sense probabilistically and a weighted $\infty$-norm $$
\rho'\Big((\zeta_t)_{t\ge 0},(\xi_t)_{t\ge 0}\Big):= \sum^\infty_{n= 1}2^{-n}\left(1 \land \sup_{ 0\le s\le n}\norm{\zeta_s-\xi_s}\right),
$$ in space.  Combined with Theorem~\ref{second th}, we can easily see that also $(q_t^m)_{t \geq 0} \rightarrow (\zeta_t)_{t \geq 0}$ in the case $\delta \in (0,1)${, where the} convergence is in the sense of Theorem~\ref{second th}.
\end{remark}

\section{Heterogeneous-in-time} \label{sec:heterogeneous}
In the previous section, we have studied the effect of losing momentum and reducing the learning rate in a uniform fashion. Having a homogeneous mass and learning rate is unquestionably very popular in practice. However, while we could show convergence to the minimiser of the deterministic dynamics, this is not clear for the stochastic dynamics. To obtain convergence in SGMP, we need to reduce both, learning rate and momentum, over time. Hence, we need to discuss a heterogeneous version of the SGMP. Indeed, we now allow the mass $m$ to be a non-constant function of time, especially to be reduced to zero over time, and the index process $(i^{\beta}(t))_{t \geq 0}$ to have waiting times that get smaller as time progresses. Thus we study the case where the stochastic gradient-momentum process converges to the stochastic gradient process and the gradient flow, respectively, at runtime. We have already seen the advantage of reducing mass over time in the deterministic setting in Figure~\ref{figure:losing_momentum}. 

We split this section into two parts. In Subsection~\ref{sec_losingmomentum}, we let only the momentum depend on and decrease over time. We introduce the dynamical system, show well-posedness, and discuss the longtime behaviour of the new dynamical system. We see that asymptotically, the dynamical system behaves like the stochastic gradient process. Then, in Subsection~\ref{sec_losingboth}, we decrease mass and learning rate over time. In this case, we discuss a setting, in which we see convergence of the stochastic dynamical system to the minimiser $\theta_*$ of the target function $\bar\Phi$.
\subsection{Losing momentum over time} \label{sec_losingmomentum}
We consider the stochastic gradient-momentum process with a time-dependent, decreasing mass.
\begin{definition}\label{def:DmSGPC}
The \emph{decreasing-mass stochastic gradient-momentum process} (dmSGMP) is a solution of the following stochastic differential equation,
\begin{equation}\label{eq:AS:pqeiet}
\left\{ \begin{array}{rl}
\de \tilde q_t &= \tilde p_t \de t,\\
\e(t) \de \tilde p_t &= - \nabla \Phi_{i(t)}(\tilde q_t) \de t- \tilde p_t \de t, \\
\tilde p_0, \tilde q_0 &\in X,
\end{array} \right.
\end{equation}
where  $\Phi_j$ satisfies Assumption \ref{asSGPf}, $j = 1, \cdots, N$. The stochastic process $(i(t))_{t\ge0}$ is defined in Definition \ref{index}. The mass $\e(t)>0$ is strictly decreasing and differentiable with  $\lim_{t\to\infty}\e(t)=0.$  
\end{definition}
The formal limit of (\ref{eq:AS:pqeiet}) is the stochastic gradient process without momentum, as given in \eqref{eq:SGP_old}. As in the previous section, we start our discussion  considering the decreasing mass case with a fixed index, i.e.\  the following dynamical system:
\begin{equation}\label{eq:AS:pqeiei}
\left\{ \begin{array}{rl}
\de q^i_t &= p^i_t \de t,\\
\e(t) \de p^i_t &= - \nabla \Phi_i(q^i_t) \de t-  p^i_t \de t, \\
p^i_0, q^i_0 &\in X,
\end{array} \right.
\end{equation}
where the mass  $(\e(t))_{t \geq 0}$ is chosen as before. In addition, we denote $\mathcal{E}(t):=\int_0^t 1/\e(s)\de s.$ The associated limiting equation is given by \eqref{Eq_SGPC}.
Next, we discuss the well-posedness of the system  (\ref{eq:AS:pqeiei}). 

\subsubsection*{Well-posedness}
We study the well-posedness of the deterministic dynamical system \eqref{eq:AS:pqeiei}. Well-posedness of the stochastic dynamical system \eqref{eq:AS:pqeiet} then follows from, e.g.\ Proposition~1 in \cite{Jonas}.
\begin{proposition}
For any fixed $i\in I$,  let $\nabla\Phi_i$ satisfy Assumption \ref{asSGPf} with constant $L$ and $\e(t)$ satisfy Assumption \ref{as1.21}. Then  \eqref{eq:AS:pqeiei} admits a unique solution.
\end{proposition}

\begin{proof}
We first rewrite equation \eqref{eq:AS:pqeiei} as
\begin{equation*} 
\left\{ \begin{array}{rl}
\de q^i_t &= p^i_t\de t,\\
 \de p^i_t &= - \nabla \Phi_i(q^i_t)/\e(t)\de t-  p^i_t/\e(t)\de t, \\
p^i_0, q^i_0 &\in X.
\end{array} \right.
\end{equation*}
We define a map $\mathcal{T}_0: \mathcal{C}([0,t_1],\R)\times \mathcal{C}([0,t_1],\R) \to \mathcal{C}([0,t_1],\R)\times \mathcal{C}([0,t_1],\R),$ as
\begin{align*}
    \mathcal{T}_0(x,y)= \Big(q_0+\int_0^t y_s\de s,\ p_0- \int_0^t \nabla \Phi_i(x_s)/\e(s)\de s-  \int_0^ty_s/\e(s)\de s\Big)
\end{align*} 
where $0\le t\le t_1.$
Next, we show that $\mathcal{T}_0$ is a contraction for $t_1$ small enough. Since $\e(t)$ is strictly decreasing and positive and  using Assumption \ref{as1.21}, we can conclude that 
\begin{align*}
   (e^{\lambda t} \e(t))'\ge 0\ \ \ \ \text{and} \ \ \ \ \e(t)\ge \e_0e^{-\lambda t}.
\end{align*}
We then conclude 
\begin{align*}
    \norm{\mathcal{T}_0(q,p)-\mathcal{T}_0(\hat q,\hat p)}_{\infty}
    &\le t_1\norm{p-\hat p}_\infty+ L\norm{q-\hat q}_{\infty}\int_0^{t_1}\frac{1}{\e(s)}\de s+\norm{p-\hat p}_\infty\int_0^{t_1}\frac{1}{\e(s)}\de s\\
    &\le (\norm{q-\hat q}_{\infty}+\norm{p-\hat p}_\infty)\frac{L+1}{\lambda \e_0}(e^{\lambda t_1}-1+t_1).
\end{align*}
Set  $t_1=\frac{\lambda \e_0}{2(L+1)(2\lambda+1)}=:C(\lambda, \e_0,L)$, and since $e^{\lambda x}-1\le 2\lambda x$ when $\lambda x\le 1,$ we have
\begin{align*}
    \frac{L+1}{\lambda \e_0}(e^{\lambda t_1}-1+t_1)\le \frac{L+1}{\lambda \e_0}(2\lambda+1)t_1\le \frac{1}{2}.
\end{align*}
Then $\mathcal{T}_0$ is contracting with constant smaller than $1/2.$ By the Banach Fixed Point Theorem, we conclude that (\ref{eq:AS:pqeiei}) admits a unique solution on the interval $[0,t_1]$ with initial value $(q_0,p_0).$ Let $\mathcal{T}_n: \mathcal{C}([t_n,t_{n+1}],\R)\times \mathcal{C}([t_n,t_{n+1}],\R)\to  \mathcal{C}([t_n,t_{n+1}],\R)\times \mathcal{C}([t_n,t_{n+1}],\R),$  with
\begin{align*}
    \mathcal{T}_n(x,y)= \Big(q_{t_n}+\int_{t_n}^{t} y_s/\e(s)\de s,\ p_{t_n}- \int_{t_n}^{t} \nabla \Phi_i(x_s)/\e(s)\de s-  \int_{t_n}^{t}y_s/\e(s)\de s\Big).
\end{align*}
Then we have
\begin{align*}
    \norm{\mathcal{T}_n(q,p)-\mathcal{T}_n(\hat q,\hat p)}_{\infty}
    &\le (t_{n+1}-t_n)\norm{p-\hat p}_\infty+ L\norm{q-\hat q}_{\infty}\int_{t_n}^{t_{n+1}}\frac{1}{\e(s)}\de s\\
    & \quad +\norm{p-\hat p}_\infty\int_{t_n}^{t_{n+1}}\frac{1}{\e(s)}\de s\\
    &\le (\norm{q-\hat q}+\norm{p-\hat p}_\infty)\frac{L+1}{\lambda \e_0}e^{\lambda t_n}(e^{\lambda h_n}-1+h_n),
\end{align*}
where $h_n=t_{n+1}-t_n.$ Set $h_n=C(\lambda, \e_0,L) e^{-\lambda t_n}=\frac{\lambda \e_0 e^{-\lambda t_n}}{2(L+1)(2\lambda+1)}$, then 
\begin{align*}
    \frac{L+1}{\lambda \e_0}e^{\lambda t_n}(e^{\lambda h_n}-1+h_n)\le  \frac{L+1}{\lambda \e_0}e^{\lambda t_n}(2\lambda+1)h_n\le \frac{1}{2}.
\end{align*}
This implies $\mathcal{T}_n$ is a contraction, and by the Banach Fixed Point Theorem, we conclude that (\ref{eq:AS:pqeiei}) admits a unique solution on the interval $[t_n,t_{n+1}]$ with initial value $(\tilde q_{t_n},\tilde p_{t_n})$. Therefore, by Lemma \ref{tn}, we have $t_n\to\infty,$ and, thus, (\ref{eq:AS:pqeiei}) admits a unique solution globally.
\end{proof}

As the unique solution to the dynamical system \eqref{eq:AS:pqeiei} is a strong solution, the solution is continuously differentiable. Thus,  the map $(p^i_0, p^i_0) \mapsto (p^i_t, p^i_t)$ is continuous for any $t \geq 0$. Hence, the differential equations \eqref{eq:AS:pqeiet} and  \eqref{eq:AS:pqeiei} are well-posed.

\subsubsection*{Longtime behavior} Having understood the well-posedness of dmSGMP, we are now interested in its longtime behaviour. By reducing the mass over time, we aim to reduce the momentum and {prove convergence} to the stochastic gradient process \eqref{eq:SGP_old}.

To show this convergence result, we now collect a number of auxiliary results that mainly concern the deterministic system  \eqref{eq:AS:pqeiei} that considers a single, fixed sample. Some of these results may remind the reader of similar results shown in Appendix~\ref{appendix_decreasing_momentum} regarding the fixed sample, homogeneous momentum dynamics; we make these connections clear in the titles of the following results. The intermediate goal is to show a bound between the system \eqref{eq:AS:pqeiei} and its limiting equation \eqref{Eq_SGPC}. From there we then go back to the randomised setting and discuss the longtime behaviour and convergence of dmSGMP.
We first show that the deterministic system converges  at exponential speed to its unique stationary point.

\begin{lemma}[cf. Lemma~\ref{lem:dqi}]\label{lyp:cor} 
{Let $(p^{i}_t,q^{i}_t)_{t \geq 0}$ solve (\ref{eq:AS:pqeiei}). Let $\theta_*^i$ be the critical point of $\Phi_i$ for $i\in I$. Under the Assumptions \ref{asSGPf}, \ref{as1.2} with $\mathcal{A}^{\lambda_i, 2}$, and \ref{as1.21} with $\lambda=\min_{i\in I}\lambda_i$}, we have
\begin{align}\label{Lemma4}
    \norm{q^i_t-\theta^i_*}^2\le 8e^{-\lambda t}V^i(0,q^i_0,p^i_0)  \qquad (t \geq 0),
\end{align}
where $V^i(t,x,y)= \e(t)(\Phi_i(x)-\Phi_i(\theta^i_*))+ \frac{1}{4}\Big(\norm{x-\theta^i_*+\e(t)y}^2+\norm{\e(t)y}^2\Big).$
\end{lemma}
\begin{proof}
By Lemma \ref{lyap:witht},  we have
     $\norm{q^i_t-\theta^i_*}^2\le 8V^i(t,q^i_t,p^i_t)\le 8e^{-\lambda t}V^i(0,q^i_0,p^i_0).$
\end{proof}
In the next auxiliary result, we show Lipschitz continuity of the position $(q_t^i)_{t \geq 0}$ and boundedness of the velocity $(p_t^i)_{t \geq 0}$.
\begin{lemma}[cf. Lemma~\ref{lem:bound+Lip}]\label{lem:second p}
{Let $(p^{i}_t,q^{i}_t)_{t \geq 0}$ solve (\ref{eq:AS:pqeiei}). Let $\theta_*^i$ be the critical point of $\Phi_i$ for $i\in I$. Under the Assumptions \ref{asSGPf}, \ref{as1.2} with $\mathcal{A}^{\lambda_i, 2}$, and \ref{as1.21} with $\lambda=\min_{i\in I}\lambda_i$}, for any $0\le s\le t,$ we have 
\begin{align*}
    \norm{q^i_t-q^i_s}&\le  C^{(1)}_L(t-s)\Big(\norm{q^i_0-\theta^i_*}+(\e_0+e^{-\mathcal{E}(s)})\norm{p^i_0}\Big),\\
    \norm{p^i_t}&\le e^{-\mathcal{E}(t)}\norm{p^i_0}+ C^{(1)}_L(\norm{q^i_0-\theta^i_*}+\e_0\norm{p^i_0})  \qquad (t \geq 0),
\end{align*}
where $ C^{(1)}_L=1+8L^2+8L$.
\end{lemma}
\begin{proof}
From (\ref{eq:AS:pqeiei}), we have 
\begin{align*}
    \frac{\de (e^{\mathcal{E}(t)}p^i_t)}{\de t}= -e^{\mathcal{E}(t)}\nabla \Phi_i(q^i_t)/\e(t),
\end{align*}
which implies 
\begin{align}\label{pqeie}
    p^i_t= e^{-\mathcal{E}(t)}p^i_0-e^{-\mathcal{E}(t)} \int_0^te^{\mathcal{E}(s)}\nabla \Phi_i(q^i_s)/\e(s)\de s.
\end{align}
Hence,
\begin{align*}
    \norm{p^i_t}\le&  e^{-\mathcal{E}(t)}\norm{p^i_0}+e^{-\mathcal{E}(t)} \int_0^t\frac{e^{\mathcal{E}(s)}}{\e(s)}\norm{\nabla \Phi_i(q^i_s)}\de s\\
    =& e^{-\mathcal{E}(t)}\norm{p^i_0}+e^{-\mathcal{E}(t)} \int_0^t\frac{e^{\mathcal{E}(s)}}{\e(s)}\norm{\nabla \Phi_i(q^i_s)-\nabla \Phi_i(\theta^i_*)}\de s\\
    \le& e^{-\mathcal{E}(t)}\norm{p^i_0}+Le^{-\mathcal{E}(t)} \int_0^t\frac{e^{\mathcal{E}(s)}}{\e(s)}\norm{q^i_s-\theta^i_*}\de s\\
    \overset{(\ref{Lemma4})}{\le}& e^{-\mathcal{E}(t)}\norm{p^i_0}+8Le^{-\mathcal{E}(t)}(V^i(0,q^i_0,p^i_0))^{\frac{1}{2}} \int_0^t\frac{e^{\mathcal{E}(s)}}{\e(s)}e^{\frac{-\lambda s}{2} }\de s\\
    \le& e^{-\mathcal{E}(t)}\norm{p^i_0}+8L(V^i(0,q^i_0,p^i_0))^{\frac{1}{2}}.
\end{align*}
{The last inequality holds since $\int_0^t\frac{e^{\mathcal{E}(s)}}{\e(s)}e^{\frac{-\lambda s}{2} }\de s\le e^{\mathcal{E}(t)}$, which can be seen by  taking the derivative on both sides and comparing the initial values.}
And since $V^i(0,q^i_0,p^i_0)\le (1+L)(\norm{q^i_0-\theta^i_*}^2+\e^2_0\norm{p^i_0}^2),$ we have 
\begin{align*}
    \norm{p^i_t}\le e^{-\mathcal{E}(t)}\norm{p^i_0}+(1+8L^2+8L)(\norm{q^i_0-\theta^i_*}+\e_0\norm{p^i_0}). 
\end{align*}
We set $C^{(1)}_L=(1+8L^2+8L)$ and have 
\begin{align*}
    \norm{q^i_t-q^i_s}
    &\le \int_s^t\norm{p^i_m}\de m 
    \le C^{(1)}_L(t-s)(\norm{q^i_0-\theta^i_*}+\e_0\norm{p^i_0})+\norm{p^i_0}\int_s^te^{-\mathcal{E}(x)}\de x\\
    &\le C^{(1)}_L(t-s)\left(\norm{q^i_0-\theta^i_*}+(\e_0+e^{-\mathcal{E}(s)})\norm{p^i_0}\right).
\end{align*}
\end{proof}

In the third lemma, we give a bound on the system's acceleration $({\de p^i_t}/{\de t})_{t \geq 0}$.
\begin{lemma}[cf. Lemma~\ref{lem:bound_Accel}]\label{lemmaine1}
{Let $(p^{i}_t,q^{i}_t)_{t \geq 0}$ solve (\ref{eq:AS:pqeiei}). Let $\theta_*^i$ be the critical point of $\Phi_i$ for $i\in I$. Under the Assumptions \ref{asSGPf}, \ref{as1.2} with $\mathcal{A}^{\lambda_i, 2}$, and \ref{as1.21} with $\lambda=\min_{i\in I}\lambda_i$}, we have
\begin{align}\label{Lemma6}
    \norm{\frac{\de p^i_t}{\de t}}\le C^{(2)}_L \left(\frac{ e^{-\mathcal{E}(t)}}{\e(t)}+2\right)(\norm{p^i_0}+\norm{q^i_0-\theta^i_*})  \qquad (t \geq 0),
\end{align}
where $ C^{(2)}_L=8(C^{(1)}_L)^2=8(1+8L^2+8L)^2$. 
\end{lemma}

\begin{proof}
 From (\ref{pqeie}), we have
\begin{align*}
   \frac{\de p^i_t}{\de t}= \frac{-e^{-\mathcal{E}(t)}p^i_0}{\e(t)}-\frac{\nabla \Phi_i(q^i_t)}{\e(t)}+\frac{e^{-\mathcal{E}(t)}}{\e(t)}\int_0^te^{\mathcal{E}(s)}\nabla \Phi_i(q^i_s)/\e(s)\de s.
\end{align*}
Notice that
\begin{align*}
    \frac{\nabla \Phi_i(q^i_t)}{\e(t)}=\frac{e^{-\mathcal{E}(t)}}{\e(t)}\int_0^te^{\mathcal{E}(s)}\nabla \Phi_i(q^i_t)/\e(s)\de s +\frac{e^{-\mathcal{E}(t)}\nabla \Phi_i(q^i_t)}{\e(t)}.
\end{align*}
Then, we can rewrite $\frac{\de p^i_t}{\de t}$ as
\begin{align*}
    \frac{\de p^i_t}{\de t}= \frac{-e^{-\mathcal{E}(t)}(p^i_0+\nabla \Phi_i(q^i_t))}{\e(t)}+\frac{e^{-\mathcal{E}(t)}}{\e(t)}\int_0^te^{\mathcal{E}(s)}(\nabla \Phi_i(q^i_s)-\nabla \Phi_i(q^i_t))/\e(s)\de s.
\end{align*}
Hence,
\begin{align*}
    \norm{\frac{\de p^i_t}{\de t}}
    &\le \frac{e^{-\mathcal{E}(t)}}{\e(t)}\Big(\norm{p^i_0+\nabla \Phi_i(q^i_t)-\nabla \Phi_i(\theta^i_*)}+\int_0^te^{\mathcal{E}(s)}\norm{\nabla \Phi_i(q^i_s)-\nabla \Phi_i(q^i_t)}/\e(s)\de s\Big)\\
    &\le \frac{(1+L) e^{-\mathcal{E}(t)}}{\e(t)}\Big(\norm{p^i_0}+\norm{q^i_t-\theta^i_*}+\int_0^te^{\mathcal{E}(s)}\norm{q^i_s-q^i_t}/\e(s)\de s\Big)\\
    &\underbrace{\le}_{(a_2)}  \frac{4C^{(1)}_L e^{-\mathcal{E}(t)}}{\e(t)}\Big(\norm{q^i_0-\theta^i_*}+\norm{p^i_0} \\ & \qquad \qquad + C^{(1)}_L\int_0^t(t-s)e^{\mathcal{E}(s)}/\e(s)\de s\Big(\norm{q^i_0-\theta^i_*}+(\e_0+e^{-\mathcal{E}(s)})\norm{p^i_0}\Big)\Big)\\
    &= \frac{4C^{(1)}_L e^{-\mathcal{E}(t)}}{\e(t)}\norm{q^i_0-\theta^i_*}\Big(1+C^{(1)}_L\int_0^t(t-s)e^{\mathcal{E}(s)}/\e(s)\de s\Big)\\
    &+ \frac{4C^{(1)}_L e^{-\mathcal{E}(t)}}{\e(t)}\norm{p^i_0}\Big(1+C^{(1)}_L\int_0^t(t-s)e^{\mathcal{E}(s)}(\e_0+e^{-\mathcal{E}(s)})/\e(s)\de s\Big)\\
    &\le 4C^{(1)}_L\Big(\frac{ e^{-\mathcal{E}(t)}}{\e(t)}+2C^{(1)}_L\Big)\norm{q^i_0-\theta^i_*}+4C^{(1)}_L\Big(\frac{ e^{-\mathcal{E}(t)}}{\e(t)}+4C^{(1)}_L\Big)\norm{p^i_0},
\end{align*}
where the last step follows since
\begin{align*}
    \int_0^t(t-s)e^{\mathcal{E}(s)}/\e(s)\de s
   \le 2\e(t)e^{\mathcal{E}(t)}
\end{align*}
{and the inequality $(a_2)$ follows from Lemmas \ref{lyp:cor} and \ref{lem:second p} and the fact that $V^i(0,q^i_0,p^i_0)\le (1+L)(\norm{q^i_0-\theta^i_*}^2+\e^2_0\norm{p^i_0}^2).$}
Hence, 
\begin{align*}
     \norm{\frac{\de p^i_t}{\de t}}
    &\le 4C^{(1)}_L \Big(\frac{ e^{-\mathcal{E}(t)}}{\e(t)}+4C^{(1)}_L\Big)(\norm{p^i_0}+\norm{q^i_0-\theta^i_*})\\
    &\le 8(C^{(1)}_L)^2 \Big(\frac{ e^{-\mathcal{E}(t)}}{\e(t)}+2\Big)(\norm{p^i_0}+\norm{q^i_0-\theta^i_*}).
    \end{align*}
    We set $C^{(2)}_L=8(C^{(1)}_L)^2=8(1+8L^2+8L)^2,$ which completes the proof.
\end{proof}

We {have now attained} the aforementioned intermediate goal where we show a bound between the system \eqref{eq:AS:pqeiei} and its limiting equation \eqref{Eq_SGPC}. We need to additionally ask for some underlying convexity; strong convexity in this case: {  Indeed, we aim to show that $q_{\tau_n}-\theta_{\tau_n}$ is contractive. Assumption \ref{as1.2} would only lead to contractivity of $q_{\tau_n}-\theta_*$.}

\begin{proposition}[cf. Proposition~\ref{inductionlem}]\label{inductionlem2}
Let $(p^{i}_t,q^{i}_t)_{t \geq 0}$ solve (\ref{eq:AS:pqeiei}). Let $\theta_*^i$ be the critical point of $\Phi_i$ for $i\in I$. We assume $\Phi_i$ is strongly convex with constant $\mu$ and let $\Phi_i$ satisfy Assumption \ref{asSGPf} and $(\e(t))_{t \geq 0}$ satisfy Assumption \ref{as1.21}  with $\lambda=(\mu/4)\land (1/4)$. Then we have
\begin{align*}
\norm{q^i_t-\theta^i_t}\le e^{-\mu t/2 }\norm{q^i_0-\theta^i_0}+C^{(1)}_{L,\mu} \e^{\frac{1}{2}}_0  \Big(\norm{p^i_0}+\norm{q^i_0-\theta^i_*}\Big)  \qquad (t \geq 0),
\end{align*}
where $C^{(1)}_{L,\mu}=2C^{(2)}_L\mu^{-1/2}(1+\mu^{-1/2}+\mu^{1/2})+\mu^{1/2},$ { which is larger than $\mu^{1/2}.$}
\end{proposition}
\begin{proof}
By Lemma \ref{lemmaine1} and Lemma \ref{lem:conx}, we have
\begin{align*}
    \frac{1}{2}\frac{\de \norm{q^i_t-\theta^i_t}^2}{\de t}&= \ip{q^i_t-\theta^i_t}{\frac{\de q^i_t}{\de t}-\frac{\de \theta^i_t}{\de t}}\\
    &=-\ip{q^i_t-\theta^i_t}{\nabla \Phi_i(q^i_t)-\nabla \Phi_i(\theta^i_t) }-\e(t)\ip{q^i_t-\theta^i_t}{\frac{\de p^i_t}{\de t}}\\
    &\le -\mu \norm{q^i_t-\theta^i_t}^2+\frac{\mu}{2}\norm{q^i_t-\theta^i_t}^2+\frac{\e^2(t)}{2\mu}\norm{\frac{\de p^i_t}{\de t}}^2\\
    &\overset{(\ref{Lemma6})}{\le} -\frac{\mu}{2}\norm{q^i_t-\theta^i_t}^2+(C^{(2)}_L)^2(\mu)^{-1}\e^2(t)\Big(\frac{ e^{-2\mathcal{E}(t)}}{\e^2(t)}+4\Big)\Big(\norm{p^i_0}^2+\norm{q^i_0-\theta^i_*}^2\Big)\\
    &= -\frac{\mu}{2}\norm{q^i_t-\theta^i_t}^2+(2C^{(2)}_L)^2(\mu)^{-1}\Big( e^{-2\mathcal{E}(t)}+\e^2(t)\Big)\Big(\norm{p^i_0}^2+\norm{q^i_0-\theta^i_*}^2\Big),
\end{align*}
which implies 
\begin{align*}
    \norm{q^i_t-\theta^i_t}^2&\le  e^{-\mu t}\norm{q^i_0-\theta^i_0}^2\\
    &\ \ \ \ \ +(C^{(2)}_L )^2(\mu)^{-1} e^{-\mu t}\int_0^t\Big(e^{-2\mathcal{E}(s)}+\e^2(s)\Big)e^{\mu s}\de s  \Big(\norm{p^i_0}^2+\norm{q^i_0-\theta^i_*}^2\Big).
\end{align*}
From Lemma \ref{epsilong1}, we know that 
$e^{-\mu t}\int_0^t\e^2(s)e^{\mu s}\de s\le 2\mu^{-1}\e^2(t)\le 2\mu^{-1}\e^2_0\le 2\mu^{-1}\e_0,  $ and 
$ e^{-\mu t}\int_0^te^{-2\mathcal{E}(s)}e^{\mu s}\de s\le (1+\mu)\e_0$, which implies
\begin{align*}
    \norm{q^i_t-\theta^i_t}^2\le e^{-\mu t}\norm{q^i_0-\theta^i_0}^2+(C^{(1)}_{L,\mu} )^2 \e_0  \Big(\norm{p^i_0}^2+\norm{q^i_0-\theta^i_*}^2\Big).
\end{align*}

\end{proof}
From Lemma \ref{lem:second p} and Proposition \ref{inductionlem2}, we immediately have the following corollary.
\begin{corollary}\label{corollary noprove}
Under the same condition as Proposition \ref{inductionlem2}. For $0< \e_0 \le 1$ and any fixed $i\in I,$ we have
\begin{align*}
\norm{q^i_t-\theta^i_t}&\le e^{-\mu t/2 }\norm{q^i_0-\theta^i_0}+ C^{(1)}_{L,\mu} \e_0 \Big(\norm{p^i_0}+\norm{q^i_0-\theta^i_0}+\norm{\theta^i_0-\theta^i_*}\Big),\\
 \norm{p^i_t}&\le (e^{-\mathcal{E}(t)}+C^{(1)}_L\e_0)\norm{p^i_0}+C^{(1)}_L  \Big(\norm{q^i_0-\theta^i_0}+\norm{\theta^i_0-\theta^i_*}\Big)  \qquad (t \geq 0).
\end{align*}
\end{corollary}

We can now show the main statement of this section. The convergence of dmSGMP to the stochastic gradient process in the longtime limit. Importantly, we assume a coupling between the processes through the index process $(i(t))_{t \geq 0}$ that is identical in both dynamical systems. Throughout this section, we collected evidence for a contractive behaviour in between the deterministic processes  \eqref{eq:AS:pqeiei} and \eqref{Eq_SGPC}. We now use these results to study the randomised version piece-by-piece. This, of course, is a very similar strategy to that used in Theorem~\ref{second th}.

\begin{theorem}\label{thm:dm2p} {Let $(\tilde p_t,\tilde q_t)_{t \geq 0}$ and $(\theta_t)_{t\geq0}$ solve (\ref{eq:AS:pqeiet}) and (\ref{eq:SGP_old}), respectively. Let $\theta_*^i$ be the critical point of $\Phi_i$ for $i\in I$. We assume that $\Phi_i$ is strongly convex with constant $\mu$ and let $\Phi_i$ satisfy Assumption \ref{asSGPf} and $\e(t)$ satisfy Assumption \ref{as1.21}  with $\lambda=(\mu/4)\land (1/4)$} for $i \in I$. Let $\{\tau_n\}_{n\ge 1}$ be the sequences of the jump times of the process $(i(t))_{t \geq 0}$. If we assume that $\e_0\le ${$\frac{(C^{(2)}_{L,\mu})^{-2}\mu^2}{4(2\gamma N+\mu)^2}$}, where $C^{(2)}_{L,\mu}=C^{(1)}_{L,\mu}+ C^{(1)}_L$, then
$$\E\Big[\norm{\tilde q_{\tau_n}-\theta_{\tau_n}}  \Big]\to 0 \ \ \ \ (n \rightarrow \infty).$$
 Furthermore, if we additionally assume that $\e$ decays at least exponentially, i.e.\  there exist $C,c>0$ such that $\e(t)\le Ce^{-ct},$ we have
 \begin{align*}
     \norm{\tilde q_t-\theta_t}\to 0\ \ \ \ (t \rightarrow \infty)
 \end{align*}
  almost surely and in expectation. 
\end{theorem}

\begin{remark}
After the statement of Assumption \ref{as1.21}, we have concluded that $m(t) \geq m_0 e^{-\lambda t}$. Thus, the constant $c$ mentioned above needs to be smaller than $\lambda$.
\end{remark}

\begin{proof}
  For any $n\ge 0,$ by Corollary \ref{corollary noprove}, and since the initial value of $\e(t)$ in the interval $[\tau_n,\tau_{n+1})$ is exactly $\e(\tau_n)$, we have
\begin{align*}
    \norm{\tilde q_{\tau_{n+1}}-\theta_{\tau_{n+1}}}&\le e^{-\mu (\tau_{n+1}-\tau_n)/2}\norm{\tilde q_{\tau_n}-\theta_{\tau_n}}\\
    &\ \ \ +C^{(1)}_{L,\mu} \e(\tau_n)\Big(\norm{\tilde q_{\tau_n}-\theta_{\tau_n}}+\norm{\tilde p_{\tau_n}}+\norm{\theta_{\tau_n}-\theta^{i(\tau_n)}_*}\Big) 
\end{align*}
and
\begin{align*}
    \norm{\tilde p_{\tau_{n+1}}} 
     \le \Big(e^{-\mathcal{E}(\tau_{n+1} - \tau_{n})}+C^{(1)}_L \e(\tau_{n})\Big)\norm{\tilde p^\e_{\tau_n}}
     +C^{(1)}_L\Big(\norm{\tilde q_{\tau_n}-\theta_{\tau_n}}+\norm{\theta_{\tau_n}-\theta^{i(\tau_n)}_*}\Big).
     \end{align*}
 Lemma \ref{strongthetabound} implies that
     \begin{align*}
        \norm{\theta_{\tau_n}-\theta^{i(\tau_n)}_*}\le \norm{\theta_0}+K_{\Phi}=K_{\Phi,\theta_0}.
     \end{align*}
We denote $a_n:=e^{-\mu (\tau_{n+1}-\tau_n)/2}$ and obtain
\begin{align*}
     \norm{\tilde q_{\tau_{n+1}}-\theta_{\tau_{n+1}}}   \le a_n\norm{\tilde q_{\tau_n}-\theta_{\tau_n}}
    +C^{(1)}_{L,\mu} \e(\tau_n)\Big(\norm{\tilde q_{\tau_n}-\theta_{\tau_n}}+\norm{\tilde p_{\tau_n}}+K_{\Phi,\theta_0}\Big)
\end{align*}
{ as well as $$e^{-\mathcal{E}(\tau_{n+1} - \tau_{n})}\le e^{- (\tau_{n+1}-\tau_n)/\e_0}\le e^{- (\tau_{n+1}-\tau_n)(C^{(2)}_{L,\mu})^2}\le e^{-\mu (\tau_{n+1}-\tau_n)/2}:=a_n.$$} Hence, we also have
\begin{align*}
    \norm{\tilde p_{\tau_{n+1}}}   \le 
   (a_n+C^{(1)}_L\e(\tau_n))\norm{\tilde p_{\tau_n}}
    +C^{(1)}_L \Big(\norm{\tilde q_{\tau_n}-\theta_{\tau_n}}+K_{\Phi,\theta_0}\Big).
\end{align*}
We denote 
$A_n:=\norm{\tilde q_{\tau_n}-\theta_{\tau_n}}   ,\ 
    B_n:=\norm{\tilde p_{\tau_n}}$
and then have the following iteration inequalities:
\begin{align*}
    A_{n+1}\le a_n A_n+C^{(1)}_{L,\mu} \e(\tau_n)\Big(A_n+B_n+K_{\Phi,\theta_0}\Big),\\
    B_{n+1}\le (a_n+C^{(1)}_L\e(\tau_n))B_n+C^{(1)}_L\Big(A_n+K_{\Phi,\theta_0}\Big).
\end{align*}
Moreover, denoting $D_n=A_n+\e^{\frac{1}{2}}(\tau_n) B_n$,  $C^{(2)}_{L,\mu}=C^{(1)}_{L,\mu}+ C^{(1)}_L$, we have
\begin{align}\label{Th3 D }
    D_{n+1}&\le (a_n+C^{(1)}_{L,\mu} \e(\tau_n))A_n+ \e^{\frac{1}{2}}(\tau_{n+1})(a_n+C^{(1)}_L\e(\tau_n))B_n+C^{(1)}_{L,\mu} \e(\tau_n)B_n\\
    &\ \ \ \ \ +\e^{\frac{1}{2}}(\tau_{n+1})C^{(1)}_LA_n\nonumber
   +C^{(1)}_{L,\mu} \e(\tau_n)K_{\Phi,\theta_0} +C^{(1)}_L \e^{\frac{1}{2}}(\tau_{n+1})K_{\Phi,\theta_0}\\
   & \le \Big(a_n+C^{(2)}_{L,\mu} \e^{\frac{1}{2}}(\tau_n)\Big) D_n+C^{(2)}_{L,\mu} \e^{\frac{1}{2}}(\tau_n)K_{\Phi,\theta_0}\nonumber\\
    &\le \Big(a_n+C^{(2)}_{L,\mu} \e^{\frac{1}{2}}_0\Big) D_n+C^{(2)}_{L,\mu} \e^{\frac{1}{2}}(\tau_n)K_{\Phi,\theta_0}
\end{align}
and since $\tau_{n+1}-\tau_n$ are exponentially distributed with parameter $\gamma N$   is independent of $\mathcal{F}_{\tau_n}$, we have
\begin{align}\label{iteration Dn}
   \E[ D_{n+1}]&\le \E[e^{-\mu (\tau_{n+1}-\tau_n)/2}+C^{(2)}_{L,\mu} \sqrt{\e_0}] \E[D_n]+C^{(2)}_{L,\mu} K_{\Phi,\theta_0}\E[ \e^{\frac{1}{2}}(\tau_n)],\\
   \E[ D^2_{n+1}]&\le \E[(e^{-\mu (\tau_{n+1}-\tau_n)/2}+C^{(2)}_{L,\mu} \sqrt{\e_0})^2] \E[D^2_n]+(C^{(2)}_{L,\mu} K_{\Phi,\theta_0})^2\E[ \e(\tau_n)]\\
   & \qquad \qquad\qquad\qquad \qquad\qquad\qquad\qquad+ C^{(3)}_{L,\mu}\sqrt{\e_0}K_{\Phi,\theta_0}\E[D_n]\nonumber,
\end{align}
where $C^{(3)}_{L,\mu}= 2C^{(2)}_{L,\mu}(1+C^{(2)}_{L,\mu}).$
Since $\e_0\le \frac{(C^{(2)}_{L,\mu})^{-2}\mu^2}{4(2\gamma N+\mu)^2},$ we have 
$$\E[e^{-\mu (\tau_{n+1}-\tau_n)/2}+C^{(2)}_{L,\mu}\sqrt{\e_0}]\le c_1,$$  
where $c_1=\frac{\mu+4\gamma N}{2(2\gamma N+\mu)}<1.$ From the iteration inequality and Lemma \ref{lem:iter}, we know that $\lim_{n\to\infty}\E[D_n]=0,$ which finish the proof of the first part.  
If we assume $\e(t)\le Ce^{-ct},$ we then have  
\begin{align*}
    \E[ \e^{\frac{1}{2}}(\tau_n)]\le C\E[e^{-\frac{c\tau_n}{2}}]= C\E\Big[\prod^n_{i=1}e^{-\frac{c(\tau_i-\tau_{i-1})}{2}}\Big]=C\prod^n_{i=1}\E\Big [e^{-\frac{c(\tau_i-\tau_{i-1})}{2}}\Big]= C(c_2)^n
\end{align*}
where $c_2:= \E\Big [e^{-\frac{c(\tau_i-\tau_{i-1})}{2}}\Big]=\frac{2\gamma N}{c+2\gamma N}$ which is a constant does not depend on $i$ and that is smaller than $1$. We then rewrite $(\ref{iteration Dn})$ as 
\begin{align*}
    \E[D_{n+1}]\le  c_1\E[D_n] + CC^{(3)}_{L,\mu}\sqrt{\e_0}K_{\Phi,\theta_0} (c_2)^n,
\end{align*}
which follows from the second part of Lemma \ref{lem:iter}. We then have $\E[D_n]\le \tilde Be^{-\tilde b n}.$ 
This also implies 
\begin{align*}
    \E[D^2_{n+1}]\le  c_3\E[D^2_n] + CC^{(3)}_{L,\mu}\sqrt{\e_0}K_{\Phi,\theta_0} (c_4)^n.
\end{align*}
Furthermore, by using Lemma \ref{lem:iter}, we have $\E[D^2_n]\le \tilde Be^{-\tilde b n}$ for some $\tilde B,\tilde b>0.$ The Markov inequality implies that  we have 
\begin{align*}
    \sum_{n=1}^{+\infty} \mathbb{P}(\abs{D_n}\ge \delta)\le   \delta^{-1}\sum_{n=1}^{+\infty}\E[D_n]\le  \delta^{-1}\sum_{n=1}^{+\infty}\E[D_n]\le  \delta^{-1}\tilde B\sum_{n=1}^{+\infty}e^{-\tilde b n}<+\infty,
\end{align*}
for any $\delta>0$.
This implies that $D_n\to 0$ almost surely as $n\to \infty,$ since $N_t\to \infty$ almost surely as $t\to \infty.$ We then have $D_{N_t}\to 0$  almost surely as $t\to \infty.$
By Corollary \ref{corollary noprove}, for any $t\ge 0,$  we have
\begin{align*}
    \norm{\tilde q_t-\theta_t}  
    &\le e^{-\mu (t-\tau_{N_t})/2}\norm{\tilde q_{\tau_{N_t}}-\theta_{\tau_{N_t}}}\\
    &\ \ \ \ \ +C^{(1)}_{L,\mu}  \e(\tau_{N_t})\Big(\norm{\tilde q_{\tau_{N_t}}-\theta_{\tau_{N_t}}}+\norm{\tilde p_{\tau_{N_t}}}+\norm{\theta_{\tau_{N_t}}-\theta^{i(\tau_{N_t})}_*}\Big) \\
    &\le (1+C^{(3)}_{L,\mu} )D_{N_t}+C^{(1)}_{L,\mu} \e(\tau_{N_t}) K_{\Phi,\theta_0}.
\end{align*}
This implies that $\norm{q_t-\theta_t} \rightarrow 0$ almost surely as $t\to \infty$ since $D_{N_t}\to 0 $ and $\e(\tau_{N_t})\to 0$  almost surely.

Also, in order to prove  $ \E[\norm{\tilde q_t-\theta_t}] \to 0,$ it is sufficient to show $\E[D_{N_t}]+\E[\e(\tau_{N_t})]\to 0$ as $t\to +\infty.$
Since
\begin{align*}
    \E[D^2_{N_t}]= \sum_{k=0}\E[D^2_k\indiq_{N_t=k}]^2\le \sum_{k=0}\E[D^2_k]<+\infty,
\end{align*}
which implies that $\{D_{N_t}\}_{t\ge 0}$ is uniformly integrable. Then, $\E[D_{N_t}]\to 0$ since $D_{N_t}\to 0$  almost surely.
Note that $\lim_{t\to +\infty} \e(\tau_{N_t})=0$. By the Dominated Convergence Theorem, we get 
\begin{align*}
    \lim_{t\to +\infty}  \E[\e(\tau_{N_t})]=0.
\end{align*}
This completes the proof.
\end{proof}

\subsection{Losing, both, momentum and randomness over time} \label{sec_losingboth}
In the previous subsection, we have seen that the stochastic gradient process with decreasing momentum approaches the stochastic gradient process. Asymptotically, the stochastic gradient process converges to a stationary distribution, not necessarily to a single point. 
As discussed before, to attain convergence to a single point, we usually  need to decrease the learning rate over time, see \cite{Jonas}.  We now study the following model.
\begin{definition}\label{def:DmSGPD}The \emph{decreasing-mass, decreasing-learning-rate stochastic gradient-momentum process (ddSGMP)}  is a solution of the following stochastic differential equation,
\begin{equation}\label{eq:AS:pqeiedb}
\left\{ \begin{array}{rl}
\de \hat q_t &= \hat p_t \de t,\\
\e(t) \de \hat p_t &= - \nabla \Phi_{i^\beta(t)}(\hat q_t)\de t- \hat p_t \de t, \\
\hat p_0, \hat q_0 &\in X,
\end{array} \right.
\end{equation}
where $\Phi_j$ satisfies Assumption \ref{asSGPf}, $j = 1, \ldots, N$. The mass $\e(t)>0$ is strictly decreasing and differentiable with  $\lim_{t\to\infty}\e(t)=0.$ The stochastic process $(i(t))_{t\ge0}$ is defined in Definition \ref{index}. For the re-scaled index process $(i^\beta(t))_{t\ge0}=(i(\beta(t)))_{t\ge0}$, we assume that $\beta(t) = \int_0^t \eta(s) \de s,$ $t \geq 0$ with  $\eta:[0,\infty)\to (0,\infty)$ being a non-decreasing continuously differentiable function.
\end{definition}
As opposed to the previous section, we do not consider the stochastic gradient process with constant learning rate as the limiting system, but actually the associated stochastic gradient process with \emph{decreasing} learning rate, denoted  by
\begin{equation} \label{Eq_SGPCxi}
\left\{ \begin{array}{rl}
    \de \xi_t &= - \nabla \Phi_{i^\beta(t)}(\xi_t) \de t,\\
    \xi_0 &\in X.
    \end{array} \right.
\end{equation}
Next, we define the sequence $\{\tau_n\}_{n\ge 1}$  to be the jump times of process $(i(t))_{t \geq 0}$ and  $\tau^\beta_n= \beta^{-1}(\tau_n)$, to be the jump times of  $(i^\beta(t))_{t \geq 0}$. We denote by 
\begin{align*}
    \Omega^\alpha_n:=\left\lbrace\frac{\mu}{2\eta(\tau^\beta_{n+1})}\ge \frac{\alpha_1}{\sqrt{n}} \  \textit{and}\ \e(\tau^\beta_n)\le \alpha_2 e^{-\alpha_3 \sqrt{n}} \right\rbrace,
\end{align*}
for $n \in \mathbb{N}$ an event that is used to impose a growth condition on $\beta$ and $m$. Regarding this event, we now denote the following assumption.
\begin{assumption}\label{comass}
For $n\ge k$, let $W^{\alpha,n}_k= \bigcap^n_{i=k}  \Omega^\alpha_i $. There exist $\alpha_1,\alpha_2,\alpha_3>0$ such that
$\lim_{k\to+\infty} \mathbb{P}(W^{\alpha,\infty}_k)=1.$
\end{assumption}
The event $W^{\alpha,n}_k$ is increasing in $k$ and decreasing in $n.$ Assumption \ref{comass} implies that the complement of $\Omega^\alpha_n$ is eventually small. We describe a setting for which this assumption holds in Example~\ref{Example_commass} below, after stating and proving the main results of this section: the convergence of ddSGMP to the stochastic gradient process with decreasing learning rate and the convergence of ddSGMP to the minimiser of the target function. In neither of these cases, we obtain a convergence rate. We later study the speed of convergence through numerical experiments in Section~\ref{sec:appl}. But now we start with the first of the two aforementioned results.

\begin{theorem}\label{theorem: demassderate} {Let $(\hat p_t,\hat q_t)_{t \geq 0}$ and $(\xi_t)_{t\geq0}$ solve (\ref{eq:AS:pqeiedb}) and (\ref{Eq_SGPCxi}), respectively.}
For $i\in I$, we assume that $\Phi_i$ is strongly convex with constant $\mu$ and let $\Phi_i$ satisfy Assumption \ref{asSGPf}. Let $(\e(t))_{t \geq 0}$ satisfy Assumption \ref{as1.21} with $\lambda=(\mu/4)\land (1/4).$ In addition, we assume that $(\e(t))_{t \geq 0}$ and $(\beta(t))_{t \geq 0}$ satisfy Assumption \ref{comass}. Then, we have  
$$\norm{\hat q_t-\xi_t}  \to 0$$
almost surely, as $t\to\infty$.
\end{theorem}

\begin{proof}
  For any $n\ge 0,$ by Corollary \ref{corollary noprove}, we have the following iteration inequality,
\begin{align*}
    \norm{\hat q_{\tau^\beta_{n+1}}-\xi_{\tau^\beta_{n+1}}}&\le e^{-\mu (\tau^\beta_{n+1}-\tau^\beta_n)/2}\norm{\hat q_{\tau^\beta_n}-\xi_{\tau^\beta_n}}
    \\&\ \ \ +C^{(1)}_{L,\mu} \e(\tau^\beta_n)\Big(\norm{\hat q_{\tau^\beta_n}-\xi_{\tau^\beta_n}}+\norm{\hat p_{\tau^\beta_n}}+\norm{\xi_{\tau^\beta_n}-\theta^{i(\tau^\beta_n)}_*}\Big) 
\end{align*}
and
\begin{align*}
   \norm{\hat p_{\tau^\beta_{n+1}}}\le \Big(e^{-\mu (\tau^\beta_{n+1}-\tau^\beta_n)/2}+C^{(1)}_L \e(\tau^\beta_{n})\Big)\norm{\hat p_{\tau^\beta_n}}+C^{(1)}_{L,\mu} \Big(\norm{\hat q_{\tau^\beta_n}-\xi_{\tau^\beta_n}}+\norm{\xi_{\tau^\beta_n}-\theta^{i(\tau^\beta_n)}_*}\Big). 
\end{align*}
Since $\frac{\de \beta^{-1}}{\de t}(t)= \frac{1}{\eta(\beta^{-1}(t))},$ we have $$\tau^\beta_{n+1}-\tau^\beta_n= \beta^{-1}(\tau_{n+1})-\beta^{-1}(\tau_n)\ge (\tau_{n+1}-\tau_n)/\eta(\beta^{-1}(\tau_{n+1})).$$
Hence, under the event $W^{\alpha,n}_k$, for $n\ge k$, we have 
\begin{align*}
    &\norm{\hat q_{\tau^\beta_{n+1}}-\xi_{\tau^\beta_{n+1}}}\\ &\ \le e^{-\mu (\tau^\beta_{n+1}-\tau^\beta_n)/2}\norm{\hat q_{\tau^\beta_n}-\xi_{\tau^\beta_n}}+C^{(1)}_{L,\mu} \e(\tau^\beta_n)\Big(\norm{\hat q_{\tau^\beta_n}-\xi_{\tau^\beta_n}}+\norm{\hat p_{\tau^\beta_n}}+\norm{\xi_{\tau^\beta_n}-\theta^{i(\tau^\beta_n)}_*}\Big) \Big] \\
    &\  \le \Big(e^{-\frac{\alpha_1(\tau_{n+1}-\tau_n)}{\sqrt{n}}}+C^{(1)}_{L,\mu}\alpha_2 e^{-\alpha_3 \sqrt{n}}\Big)\norm{\hat q_{\tau^\beta_n}-\xi_{\tau^\beta_n}} +C^{(1)}_{L,\mu}\alpha_2 e^{-\alpha_3 \sqrt{n}}\norm{\hat p_{\tau^\beta_n}}+ K_{\Phi,\xi_0}\alpha_2 e^{-\alpha_3 \sqrt{n}}
\end{align*}
and 
\begin{align*}
    \norm{\hat p_{\tau^\beta_n}}\le \Big(e^{-\frac{\alpha_1(\tau_{n+1}-\tau_n)}{\sqrt{n}}}+C^{(1)}_{L,\mu}\alpha_2 e^{-\alpha_3 \sqrt{n}}\Big)\norm{\hat p_{\tau^\beta_n}}+C^{(1)}_{L,\mu} \Big(\norm{\hat q_{\tau^\beta_n}-\xi_{\tau^\beta_n}}+\norm{\xi_{\tau^\beta_n}-\theta^{i(\tau^\beta_n)}_*}\Big).
\end{align*}
Since $\tau_{n+1}-\tau_n$ is independent of $\norm{\hat q_{\tau^\beta_n}-\xi_{\tau^\beta_n}}$ and $W^{\alpha,n-1}_k$, for $n\ge k$, we have 
\begin{align*}
    \E\Big[ \norm{\hat q_{\tau^\beta_{n+1}}-\xi_{\tau^\beta_{n+1}}}\indiq_{W^{\alpha,n}_k}\Big]\le  \Big(\frac{\gamma N}{\gamma N+\frac{\alpha_1}{\sqrt{n}}}+C^{(1)}_{L,\mu}\alpha_2 e^{-\alpha_3 \sqrt{n}}\Big)\E\Big[ \norm{\hat q_{\tau^\beta_n}-\xi_{\tau^\beta_n}}\indiq_{W^{\alpha,n-1}_k}\Big]\\
    +C^{(1)}_{L,\mu}\alpha_2 e^{-\alpha_3 \sqrt{n}}\E\Big[\norm{\hat p_{\tau^\beta_n}}\indiq_{W^{\alpha,n-1}_k}\Big]+ K_{\Phi,\xi_0}\alpha_2 e^{-\alpha_3 \sqrt{n}}\end{align*}
and 
\begin{align*}
    \E\Big[ \norm{\hat p_{\tau^\beta_{n+1}}}\indiq_{W^{\alpha,n}_k}\Big]&\le  \Big(\frac{\gamma N}{\gamma N+\frac{\alpha_1}{\sqrt{n}}}+C^{(1)}_{L,\mu}\alpha_2 e^{-\alpha_3 \sqrt{n}}\Big)\E\Big[ \norm{\hat p_{\tau^\beta_n}}\indiq_{W^{\alpha,n-1}_k}\Big]\\
   &\qquad \qquad  +C^{(1)}_{L,\mu}\E\Big[\norm{\hat q_{\tau^\beta_n}-\xi_{\tau^\beta_n}}\indiq_{W^{\alpha,n-1}_k}\Big]+ K_{\Phi,\xi_0}.
\end{align*}
We denote
\begin{align*}
A^\beta_n:=\E\Big[\norm{\hat q_{\tau^\beta_n}-\xi_{\tau^\beta_n}}\indiq_{W^{\alpha,n}_k}\Big]   ,\ 
    B^\beta_n:=\E\Big[\norm{\hat p_{\tau^\beta_n}}\indiq_{W^{\alpha,n}_k}   \Big].
\end{align*}
Let $D^\beta_n=A^\beta_n+ \alpha_2 e^{-\alpha_3 \sqrt{n}}B^\beta_n.$ We can find some constant $0<C_{\gamma,N,L,\mu}<1$  such that for $n\ge k $  and $k$ large enough
\begin{align}\label{iteration-beta}
    D^\beta_{n+1}\le \Big(1-\frac{C_{\gamma,N,L,\mu}}{\sqrt{n}}\Big)D^\beta_n+ C_{\Phi,\xi_0,L}e^{-\alpha_3 \sqrt{n}}.
\end{align}
Let $G^n_k:= \prod^n_{i=k}\Big(1-\frac{C_{\gamma,N,L,\mu}}{\sqrt{i}}\Big)$ with $G^n_n= 1-\frac{C_{\gamma,N,L,\mu}}{\sqrt{n}}\ge \frac{1}{2}$.  From (\ref{iteration-beta}), we have

\begin{align*}
    D^\beta_{n+1}\le G^n_kD^\beta_k+ 2C_{\Phi,\xi_0,L}\sum^n_{i=k}e^{-\alpha_3 \sqrt{i}}G^n_i.
\end{align*}
From Lemma \ref{twoseq1}, we know that there exist some constant $c>0$ such that $G^n_i\le e^{-c(\sqrt{n}-\sqrt{i})}$.  This implies 
\begin{align*}
     D^\beta_{n+1} &\le Ce^{-c\sqrt{n}}D^\beta_k+ 2C_{\Phi,\xi_0,L}\sum^n_{i=k}e^{-\alpha_3 \sqrt{i}}e^{-c(\sqrt{n}-\sqrt{i})}\\
     &\le Ce^{-c\sqrt{n}}D^\beta_k+2C_{\Phi,\xi_0,L}\sum^n_{i=k}e^{-c\land\alpha_3 \sqrt{i}}e^{-c\land\alpha_3(\sqrt{n}-\sqrt{i})} \\
&\le Ce^{-c\sqrt{n}}D^\beta_k+2C_{\Phi,\xi_0,L}ne^{-c\land\alpha_3\sqrt{n}}.
\end{align*}
Hence,
\begin{align*}
    \sum^\infty_{n=k}\mathbb{P}\Big(\norm{\hat q_{\tau^\beta_n}-\xi_{\tau^\beta_n}}\indiq_{W^{\alpha,n-1}_k}\ge \varepsilon \Big)\le \sum^\infty_{n=k} \frac{D^\beta_n}{\varepsilon}
    \le \frac{C}{\varepsilon} \sum^\infty_{n=k} ne^{-c\land\alpha_3\sqrt{n}}<+\infty.
\end{align*}
Then, the Borel–Cantelli Lemma implies that
\begin{align*}
    \lim_{n\to \infty} \norm{\hat q_{\tau^\beta_n}-\xi_{\tau^\beta_n}}\indiq_{W^{\alpha,n-1}_k}=0,
\end{align*}
almost surely.
Under the Assumption \ref{comass}, we have 
\begin{align*}
    \lim_{k\to \infty}\indiq_{(W^{\alpha,\infty}_k)^c}=0,
\end{align*}
almost surely.
Since 
\begin{align*}
    \norm{\hat q_{\tau^\beta_n}-\xi_{\tau^\beta_n}}&=\norm{\hat q_{\tau^\beta_n}-\xi_{\tau^\beta_n}}\indiq_{W^{\alpha,n-1}_k}+\norm{\hat q_{\tau^\beta_n}-\xi_{\tau^\beta_n}}\indiq_{(W^{\alpha,n}_k)^c}\\
    &\le \norm{\hat q_{\tau^\beta_n}-\xi_{\tau^\beta_n}}\indiq_{W^{\alpha,n-1}_k}+\norm{\hat q_{\tau^\beta_n}-\xi_{\tau^\beta_n}}\indiq_{(W^{\alpha,\infty}_k)^c},
\end{align*}
for $k$ large enough,  we have
\begin{align*}
    \lim_{n\to \infty}\norm{\hat q_{\tau^\beta_n}-\xi_{\tau^\beta_n}}=0,
\end{align*}
almost surely.
Similarly, we have 
\begin{align*}
    \lim_{n\to \infty}e^{-\alpha_3 \sqrt{n}}\norm{\hat p_{\tau^\beta_n}}=0
\end{align*}
almost surely.
Under the event $W^{\alpha,n}_k$, $\e(\tau^\beta_n)$ is smaller than $\alpha_2 e^{-\alpha_3 \sqrt{n}}.$ Hence, we have
\begin{align*}
    \lim_{n\to \infty}\e(\tau^\beta_n)\norm{\hat p_{\tau^\beta_n}}=0,
\end{align*}
almost surely.
Since $\lim_{n\to\infty}\tau^\beta_n=+\infty$ and
\begin{align*}
    \norm{\hat q_t-\xi_t}&\le e^{-\mu (t-\tau^\beta_{N_t})/2}\norm{\hat q_{\tau^\beta_{N_t}}-\xi_{\tau^\beta_{N_t}}}\\
    &\ \ +C^{(1)}_{L,\mu} \e(\tau^\beta_{N_t})\Big(\norm{\hat q_{\tau^\beta_{N_t}}-\xi_{\tau^\beta_{N_t}}}+\norm{\hat p_{\tau^\beta_{N_t}}}+\norm{\xi_{\tau^\beta_{N_t}}-\theta^{i(\tau^\beta_{N_t})}_*}\Big),
\end{align*}
where $\tau^\beta_{N_t}\le t\le \tau^\beta_{{N_t}+1}$, we have
\begin{align*}
    \norm{\hat q_t-\xi_t}\le (1+C^{(1)}_{L,\mu}\e_0)\norm{\hat q_{\tau^\beta_{N_t}}-\xi_{\tau^\beta_{N_t}}}+\e(\tau^\beta_{N_t})\Big(\norm{\hat p_{\tau^\beta_{N_t}}}+\norm{\xi_{\tau^\beta_{N_t}}-\theta^{i(\tau^\beta_{N_t})}_*}\Big).
\end{align*}
This  implies that $\norm{\hat q_t-\xi_t}  \to 0$, almost surely, as $t\to\infty$. 
\end{proof}

Using this result and prior knowledge about the stochastic gradient process, we can now finally show convergence to the minimiser of $\bar\Phi$.  

\begin{corollary}\label{cor:convmi}
Under the same conditions as Theorem \ref{theorem: demassderate}, we have 
\begin{align*}
\norm{\hat q_t-\theta_*}  \to 0
\end{align*}
in probability, as $t\to\infty$,
where $\theta_*$ is the unique minimiser of the function $\bar \Phi$.
\end{corollary}
\begin{proof}
From \citet[Theorem 4]{Jonas}, we know that $\xi_t$ converges to $\theta_*$ weakly, as $t \rightarrow \infty$. This is equivalent to $\norm{\xi_t-\theta_*}  \to 0$  in probability since $\theta_*$ is deterministic. Then the result follows by applying Theorem \ref{theorem: demassderate}.
\end{proof}
Hence, we have shown that when reducing -- indeed, losing -- momentum and decreasing the learning rate over time, we {have} convergence of the stochastic gradient-momentum process to the minimiser $\theta_*$.

We finish this section, by discussing the non-trivial Assumption~\ref{comass}. Indeed, we give an example below for how $(\beta(t))_{t \geq 0}$ and $(m(t))_{t \geq 0}$ can be chosen to satisfy the assumption.
\begin{example} \label{Example_commass}
Let $\beta(t)=t^2$ and $\e(t)= \e_0e^{-\lambda t}$, where $0<m_0<1$ is a constant. In this case $\eta(t)=2t$ and $\beta^{-1}(t)=\sqrt{t}.$ Hence,
\begin{align*}
    \Omega^\alpha_n=\Big\{\frac{\mu}{2\eta(\beta^{-1}(\tau_{n+1}))}\ge \frac{\alpha_1}{\sqrt{n}} \  \textit{and}\ \e(\tau^\beta_n)\le \alpha_2 e^{-\alpha_3 \sqrt{n}} \Big\}\\
    =\Big\{\frac{\mu}{4\sqrt{\tau_{n+1}}}\ge  \frac{\alpha_1}{\sqrt{n}} \  \textit{and}\ \e_0e^{-\lambda \sqrt{\tau_n}}\le \alpha_2 e^{-\alpha_3 \sqrt{n}} \Big\}.
\end{align*}
By a concentration inequality, we have  
\begin{align}\label{probineqexample2}
    \mathbb{P}\Big(\tau_n\le \frac{n}{2\gamma N} \Big)=  \mathbb{P}\Big(e^{ -2\gamma N \tau_n}\ge e^{-n}  \Big)\le e^n\E[e^{ -2\gamma N \tau_n}]=\frac{e^n}{3^n}
\end{align}
and
\begin{align}\label{probineqexample22}
    \mathbb{P}\Big(\tau_n\ge \frac{2n}{\gamma N} \Big)=  \mathbb{P}\Big(e^{ \gamma N \tau_n/2}\ge e^n  \Big)\le e^{-n}\E[e^{ \gamma N \tau_n/2}]=\frac{2^n}{e^n}.
\end{align}
When $n$ is sufficiently large, $\frac{n}{2\gamma N}\le \tau_n\le \frac{2n}{\gamma N}$ already implies that the event $ \Omega^\alpha_n$ occurs for $\alpha_1=\frac{2}{\mu\sqrt{2\gamma N}},\alpha_2=m_0,\alpha_3=\frac{\lambda}{\sqrt{2\gamma N}}$. Hence, when $n$ is sufficiently large, $(\Omega^\alpha_n)^c\subset \{\tau_n\le \frac{n}{2\gamma N}\}\cup \{\tau_n\ge \frac{2n}{\gamma N}\}.$ Therefore, from $(\ref{probineqexample2})$ and $(\ref{probineqexample22}),$ we have
\begin{align*}
     \sum^{\infty}_{n=1}\mathbb{P}(\Omega^c_n)\le \sum^{\infty}_{n=1}\Big(\frac{e^n}{3^n}+\frac{2^n}{e^n}\Big)<+\infty.
\end{align*}
The Borel-Cantelli Lemma then implies that $\lim_{k\to+\infty} \mathbb{P}(\bigcup^\infty_{i=k}  (\Omega^\alpha_i)^c)=0$ which is equivalent to  Assumption \ref{comass}.
\end{example}

\section{Discretisation of the continuous-time system} \label{sec:discr_meth}

After the previous theoretical study of stochastic gradient-momentum processes, we now propose a numerical scheme to discretise those dynamical systems.
This is crucial to turn the continuous-time dynamics into practical algorithms, see the discussion in \cite{S22,ShiJord19,wang2021global, wang2022forward,wang2022continuoustime}: the continuous-time analysis given in Sections~\ref{Sec_Homogene} and \ref{sec:heterogeneous} is only worthwhile, when the discrete algorithm behaves similarly to the continuous-time dynamics. 
The most usual discretisation of SGMP with 
{an explicit symplectic} Euler method would yield an algorithm very similar to stochastic gradient descent with classical momentum. 
However, this is hardly appropriate as 
{explicit schemes are unstable when the mass is small or when it decreases over time}.
In the following, we aim to obtain a stable, symplectic, and efficient discretisation strategy that retains the correct longtime behaviour for the deterministic parts of SGMP. We note that even in the decreasing mass case, we are committed to using a symplectic solver to retain local conversation laws in the dynamical system, see \cite{McLach}.

First, we discuss the discretisation of the underdamped gradient flow with constant or decreasing mass. For convenience, we briefly recall the ODE of interest:
\begin{equation}\label{eq:dyn_sys_discr}
\left\{ \begin{array}{rl}
\de q_t &= p_t \de t,\\
m(t)\de p_t &= - \nabla \bar\Phi(q_t) \de t-\alpha p_t \de t, \\
p_0, q_0 &\in X.
\end{array} \right.
\end{equation}

Since our theoretical analysis investigates
the case $\mass(t) \to 0$
as $t\to\infty$,
we need the discretisation
scheme to especially be stable for small values of $\mass(t)$.
In this framework, a fully explicit method, such as the forward or symplectic Euler method, 
would not be suitable for the purpose. In particular, 
since the evolution of $t\mapsto p_t$ is
directly affected by
the value of the mass $\mass(t)$, it is natural to 
consider an implicit discretisation for this variable:
\[
\mass_n \frac{p_{n+1}-p_n}{h} = - (\nabla \bar\Phi (q_{n}) +
\friction p_{n+1}),
\]
for every $n\geq 0$,
yielding
\begin{equation} \label{eq:p_discr}
p_{n+1}= \frac{1}{{\mass_n}/{h}+\friction}
\left( \frac{\mass_n}{h}p_n -\nabla \bar\Phi (q_n) 
\right),
\end{equation}
where $h>0$ denotes the discretisation step-size and $\mass_n := \mass(nh)$.
For the $q$ variable in \eqref{eq:dyn_sys_discr}
we use the scheme
\begin{equation} \label{eq:q_discr}
q_{n+1} = q_n + h p_{n+1}
\end{equation} 
for every $n\geq 0$. Finally, combining 
\eqref{eq:p_discr} and \eqref{eq:q_discr}, we 
obtain the following update rule for the 
discretised dynamical model:
\begin{equation} \label{eq:semi_impl_rule}
\begin{cases}
p_{n+1} = \frac{1}{{\mass_n}/{h}+\friction}
\left( \frac{\mass_n}{h}p_n -\nabla \bar\Phi (q_n) 
\right),\\
q_{n+1} = q_n + h p_{n+1}, 
\end{cases}    
\end{equation}
for every $n\geq 0$. As we discretise $q$ explicitly and $p$ implicity, this discretisation scheme is a semi-implicit method. Moreover, the discretisation scheme employed here is symplectic \cite[Theorem~VI.3.3]{Verlet}, so that the discretised dynamics satisfy (local) conservation laws. For more details, we refer the reader to the discussion in \cite[Section~VI.7]{Verlet}.

\begin{remark}\label{rem_underGF_GD_disc}
As suggested in the previous sections and depicted in Figure~\ref{fig_tikz_SGMP}, the parameter
$\mass$ can be interpreted as an interpolation
variable between the  gradient
flow and the underdamped gradient flow. We observe that this is still the case
for the discretised dynamical system 
\eqref{eq:semi_impl_rule}. Indeed, if we set
$\friction = 1$ and  let $\mass \to 0$, 
\eqref{eq:semi_impl_rule} turns into 
\begin{equation*}
\begin{cases}
p_{n+1} = -\nabla \bar\Phi (q_n),\\
q_{n+1} = q_n + h p_{n+1},
\end{cases}
\end{equation*}
which is exactly the gradient descent method.
\end{remark}

    {We finish this section with a comparison of the semi-implicit discretisation of the underdamped gradient flow with other momentum-based optimisation methods. First, we note that we can rephrase the update rule \eqref{eq:semi_impl_rule} to obtain a recursive definition of $q_n$ that does not involve the velocity variable $p_n$. Namely, from \eqref{eq:semi_impl_rule} we deduce that
    \begin{equation*}
        q_{n+1} = q_n + \frac{m_n}{m_n/h + \alpha} p_n - \frac{h}{m_n/h+\alpha} \nabla \bar\Phi (q_n).
    \end{equation*}
    Then, recalling that $hp_n = q_n - q_{n-1}$, we obtain that
    \begin{equation}\label{eq:update_rule_q}
        q_{n+1} = q_n + \frac{m_n}{m_n + \alpha h} (q_n - q_{n-1}) - \frac{h}{m_n/h+\alpha} \nabla \bar\Phi (q_n).
    \end{equation}
    This allows us to easily compare our scheme with other momentum-based optimisation schemes. In particular, we recall that Polyak's Heavy Ball method has constant coefficients in front of the momentum term $q_n - q_{n-1}$ and in front of $\nabla \bar\Phi(q_n)$. Moreover, when writing Nesterov's update rule in single-variable form, we obtain:
    \begin{equation*}
        q_{n+1} = q_n + \xi_n (q_n - q_{n-1}) - h \nabla \bar\Phi (q_n) - \xi_n h (\nabla \bar\Phi (q_n)- \nabla \bar\Phi (q_{n-1})),
    \end{equation*}
    with $\xi_n \equiv \frac{1-\sqrt{\mu h}}{1+\sqrt{\mu h}}$ when $\bar \Phi$ is $\mu$-strongly convex or $\xi_n = \frac{n}{n+3}$ when $\bar \Phi$ is just convex.
    Our scheme in \eqref{eq:update_rule_q} has a structure that is reminiscent of the one considered in \cite{attouch2019scaling-damped,attouch2019scaling-damped-2}, where a time rescaling was introduced in the continuous-time dynamics, yielding a non-constant (and diverging) multiplicative factor in front of the gradient of the objective. However, it is important to note that in \cite{attouch2019scaling-damped,attouch2019scaling-damped-2}  the gradient is evaluated at $q_{n+1}$, while in \eqref{eq:update_rule_q} we have $\nabla \bar \Phi(q_n)$.
    }

\subsection{Step-size choice} \label{subsection_stepsizechoice}
The first natural question about the discrete-time
optimisation method \eqref{eq:semi_impl_rule} concerns
the choice of the discretisation step-size $h>0$.
We derive a heuristic rule,
assuming for simplicity that 
$\mass_n=m$ for every $n\geq1$. 
From \eqref{eq:semi_impl_rule},
we obtain 
\begin{align*}
q_{n+1} &= q_n + \frac{h^2}{\mass + \friction h}
\left( \frac{\mass}{h} p_n - \nabla\bar\Phi(q_n)\right) \\
&= q_n -  
\frac{h^2}{\mass + \friction h}
\left( \frac{\mass}{h} 
\frac{h}{\mass + \friction h}
\left( \frac{\mass}{h}p_{n-1}
- \nabla \bar\Phi(q_{n-1})
\right)
- \nabla \bar\Phi(q_n)\right)
\end{align*}
for every $n\geq0$.
With a backward induction
argument, assuming that
$p_0=0$, we deduce that
\begin{equation*}
q_{n+1} =
q_n - 
\frac{h^2}{\mass + \friction h}
\sum_{j=0}^n\left(
\frac{\mass}{\mass + 
\friction h}
\right)^{j}
\nabla \bar\Phi(q_{n-j}).
\end{equation*}
The previous identity 
suggests that the position
at the step $k+1$ is obtained 
through $k+1$ evaluations of the
gradient at the previous points
of the discrete trajectory.
Moreover, each evaluation
is weighted by a power of the
{\it forgetting coefficient}
$\frac{\mass}{\mass + \friction h}<1$.
Therefore, it is reasonable 
to ask that the sum of the 
weights is of order 
$\frac1L$, where 
$L>0$ is 
the  Lipschitz constant of
$\nabla \bar\Phi$, see Assumption~\ref{asSGPf}. In particular,
we require
\begin{equation*}
\frac{h^2}{\mass + \friction h}
\sum_{j=0}^n\left(
\frac{\mass}{\mass + 
\friction h}
\right)^{j}
\simeq \frac{1}{L}.
\end{equation*}
Taking the limit
as $n\to\infty$ in the previous sum, we obtain 
\begin{equation*}
\frac1L \simeq 
\frac{h^2}{\mass + \friction h}
\sum_{j=0}^\infty\left(
\frac{\mass}{\mass + 
\friction h}
\right)^{j}
= \frac{h}{\friction},
\end{equation*}
or equivalently
\begin{equation} \label{eq:step-size}
h \simeq \frac{\friction}{L}.    
\end{equation}
We first note that 
the expression 
at the right-hand side of
\eqref{eq:step-size}
does not depend on
the parameter $\mass$. 
Thus, according to the empirical
argument presented above, 
there is no need to adjust the
magnitude of the step-size
according to $\mass$, also not if $\mass$ changes over time. 
This fact deviates from
the usual discretisation of
the $t\mapsto p_t$ variable in
\eqref{eq:dyn_sys_discr}.
The second observation
is that, we not only see convergence of our discrete method to gradient descent as $m \downarrow 0$, see Remark~\ref{rem_underGF_GD_disc}. We also recover 
the correct step-size for that gradient descent scheme we converge to.

We end this section with a remark on a rescaling of the potential $\Phi$.
\begin{remark}
Let us assume that the potential $\Phi$ is multiplied by
a positive constant $\rho>0$, i.e.\ 
$\Phi' = \rho \bar\Phi$.
A natural question is how we should rescale the parameters
of the method \eqref{eq:semi_impl_rule} such that
the sequence of positions $(q_n')_{n\geq 0}$ coincides with
the sequence $(q_n)_{n\geq 0}$
corresponding to the original objective.
It turns out that it is sufficient to set 
$h'=h/\rho$ and $\mass'_n= \mass_n/\rho$. Indeed, if we
we consider the sequences $(q_n')_{n\geq 0}$
and $(p_n')_{n\geq 0}$ obtained with
\begin{equation*}
\begin{cases}
p_{n+1}' = \frac{1}{\frac{\mass'_n}{h'}+\friction}
\left( \frac{\mass'_n}{h'}p_n' -\nabla \Phi' (q_n') 
\right), & p_0'=0\\
q_{n+1}' = q_n' + h' p_{n+1}', & q_0'= q_0,
\end{cases}
\end{equation*}
then a direct computation yields
\begin{equation*}
    q_n'=q_n,\qquad p_n' = \rho p_n 
\end{equation*}
for every $k\geq 1$, where the sequences 
$(q_n)_{n\geq 0}$ and $(p_n)_{n\geq 0}$ are obtained
using \eqref{eq:semi_impl_rule} with the original
objective $\bar\Phi$ and the parameters $h$ and $\mass_n$.
\end{remark}

\subsection{Randomised version of the method} 
The stochastic dynamical systems discussed throughout this work consist of piecewise deterministic ODEs where the pieces are determined by random waiting times and a subsampling process. After having discussed the ODE discretisation in the previous subsections, we now {incorporate the stochasticity into the discrete-time setting}. 
{In this regard, the random waiting time for the switching of the objective --- despite being a cornerstone of the theoretical analysis --- does not seem natural in view of practical implementations.}
The use of random waiting times in practice within this framework has been discussed by \cite{Jin, Jonas}. In the present work, we replace those random waiting times by deterministic waiting times that coincide with the time steps of the algorithm. This is exactly the {well-established} paradigm of the classical stochastic gradient descent method and 
of the classical momentum method {as well}. The waiting times may either be constant-in-time (homogeneous) or decreasing-in-time (heterogeneous). 

Throughout this work, we have (implicitly) considered the case in which we choose one potential $\Phi_{i(t)}$ at any time $t \geq 0$. Of course, we can also employ \emph{batch subsampling}. There, we have
\begin{equation} \label{eq:stoc_meth}
    \begin{cases}
p_{n+1} = \frac{1}{\frac{\mass_n}{h}+\friction}
\left( \frac{\mass_n}{h}p_n - \mathbf{v}_n
\right), \\
q_{n+1} = q_n + h p_{n+1},
\end{cases}    
\end{equation}
where 
\begin{equation} \label{eq:mini_batch_grad}
    \mathbf{v}_n :=  \frac1\ell
    \sum_{r=1}^\ell \nabla \Phi_{i_r}(q_n),
\end{equation}
and $\{ i_1,\ldots,i_\ell \}$ is a subset
 of $I$ that is sampled uniformly without replacement, i.e.\  we 
do not choose a single but rather a set of $\ell$ potentials at once and optimise with respect to their sample mean. This setting is  very useful in practice, as it reduces the subsampling error. This setting is also already fully contained in our theory. To see this, we  can just define a new set of potentials $(\Phi'_K)_{K \in I'}$, where $I' := \{K \subseteq I : \# K = \ell\}$ and $\Phi_K' = \frac1\ell
    \sum_{i \in K} \Phi_{i}$ for $K \in I'$. Then, the mean of the $(\Phi'_K)_{K \in I'}$ is $\bar\Phi$.
What is not contained in our theory, but still helpful in practice, is \emph{batch subsampling in epochs}, where at any time, we do not replace elements from which we sample in $I$ before we have picked each index once. This setting is not trivially contained in our theory due to its non-Markovian nature.


\section{Numerical Experiments} \label{sec:appl}
We now present numerical experiments in which we test the discretisation strategy proposed in Section~\ref{sec:discr_meth}. We start with academic convex and non-convex examples. Then, we employ SGMP for the training of a convolutional neural network regarding the classification of the CIFAR-10 data set. We aim to show how the method compares with the classical momentum method and standard stochastic gradient descent.

\subsection{One-dimensional non-smooth stationary point}

The first numerical experiment involves the 
one-dimensional example discussed in 
Subsection~\ref{subsec:mot_exampl}.
We recall that we want to minimise the function
\begin{equation*} 
\Phi(x) := (\mathrm{ReLU}(x)-1)^2+x^2
= \begin{cases}
x^2+1 &\mbox{if } x\leq 0\\
2x^2-2x+1 &\mbox{if } x> 0
\end{cases} \qquad (x \in \mathbb{R}),
\end{equation*}
which attains the global minimum at the point
$x^*=\frac12$. The point $\tilde x = 0$ is a non-smooth stationary
point, since $\Phi(x)\geq \Phi(\tilde x)$
for every $x\leq \tilde x$. In Subsection~\ref{subsec:mot_exampl}, we are able to show that the underdamped gradient flow can overcome the local minimiser $\tilde x$ if $\alpha^2 - 8m < 0$.

We now test this property using our discrete-time method \eqref{eq:semi_impl_rule} with the step-size constant along the iterations.  The discretisation scheme asks us to choose the learning rate $h$ according to the Lipschitz constant of $\Phi'$. Actually, the derivative 
of $\Phi$ is not continuous. 
However, in this framework, 
by ``Lipschitz constant" of $\Phi'$
we mean 
$L_0= \max \{ \mathrm{Lip}(\Phi'|_{x\leq 0}), \mathrm{Lip}(\Phi'|_{x\geq 0})\}$.
We test the scheme in this example using the precise Lipschitz constant and overestimated Lipschitz constants. The results are presented in Figure~\ref{fig:conv_mass_frict}.
We see that the condition $\alpha^2 - 8m < 0$ is also relevant for the discrete dynamical system. The inequality appears to be sharp whenever we overestimate the Lipschitz constant -- unsurprisingly as a smaller step-size leads to a more accurate discretisation of the underdamped gradient flow. Using the correct Lipschitz constant lets $\alpha^2 - 8m < 0$ appear quite conservative. Indeed, a much larger range of $\alpha, m$ allow for convergence to the global minimiser. A larger step-size makes the method more robust.

    \begin{figure}[htb]
        \centering
        \includegraphics[scale=0.33]{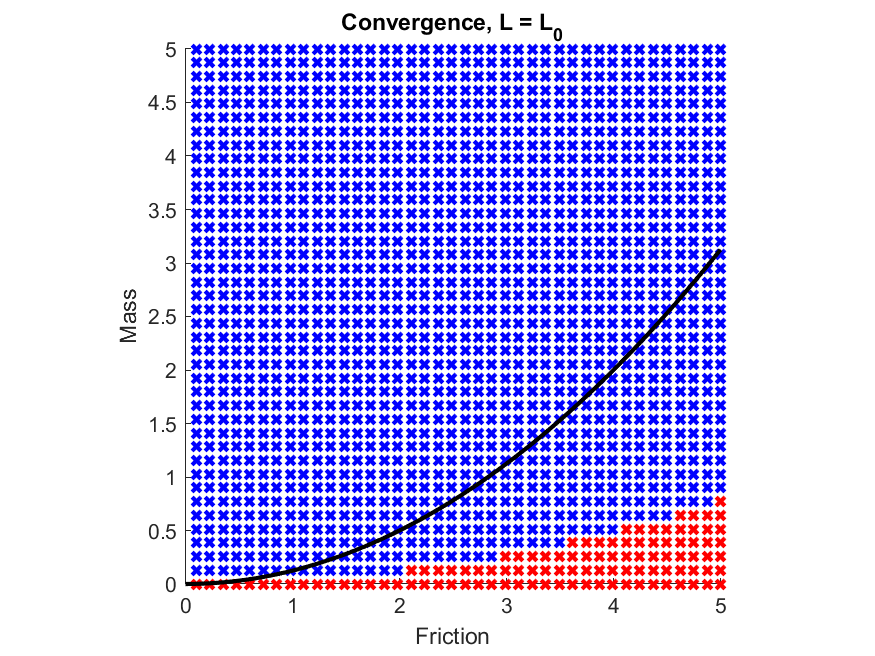}
        \includegraphics[scale=0.33]{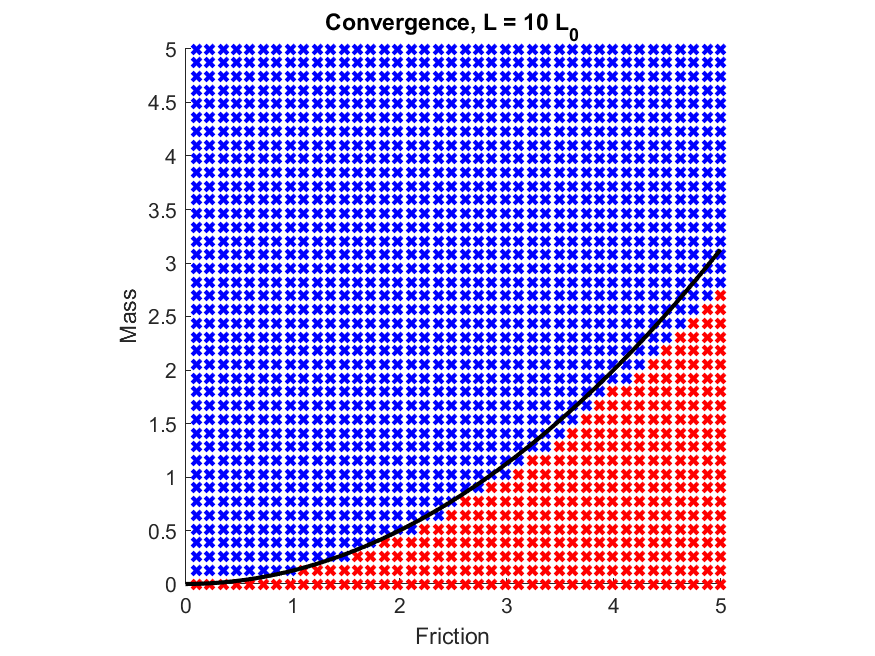}
        \includegraphics[scale=0.33]{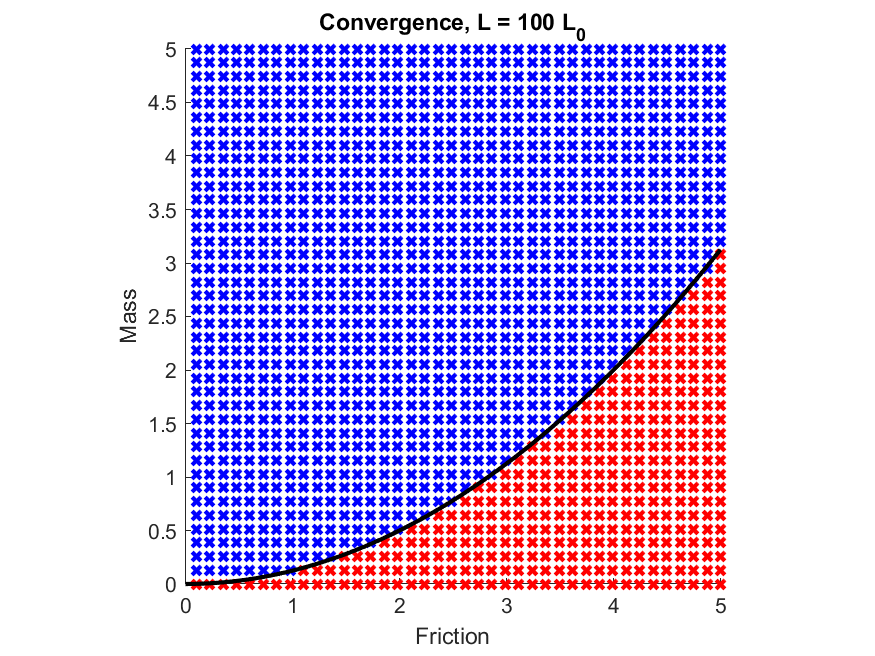}
        \caption{The plots depict
        for which  
        combinations of $\mass,\friction$ 
        the discrete-time 
        version of \eqref{eq:polyak_ode}
        derived in
        \eqref{eq:semi_impl_rule} manages
        to overcome the ``false minimiser".
        The blue crosses represent 
        convergence to the global minimiser
        of \eqref{eq:1_dim_ex},
        the red ones to the origin.
        The black curve divides the 
        $\mass,\friction$ that do and do not
        satisfy \eqref{eq:rel_escape_ex}.
        As the step-size $h=\frac1L$
        gets smaller, the 
        theoretical prediction
        \eqref{eq:rel_escape_ex} becomes more 
        accurate.
        Finally, we observe that the gradient 
        method (that corresponds to 
        $\mass =0$) never converges
        to the global minimiser.
        }
        \label{fig:conv_mass_frict}
    \end{figure}

\subsection{Strongly convex example: quadratic objective} \label{subsec:quadratic}
We now consider a target function of the form
$\bar\Phi =\frac1N \sum_{i=1}^N\Phi_i$, where
\begin{equation*}
    \Phi_i(x) = \frac{N}2 
    x^TA_ix +Nb_i^Tx,
\end{equation*}
and where
$A_i\in \R^{K\times K}$ is a 
symmetric and positive definite
matrix and $b_i\in \R^K$, for $i = 1,\ldots,N$. Importantly, all of the $(\Phi_i)_{i \in I}$ are strongly convex.
We consider $K=500$ (dimension of the 
domain), and $N=100$. 
For every $i=1,\ldots,N$, we have sampled the
eigenvalues of the matrix $A_i$ 
using a uniform distribution in 
$[0.05,15]$, and we have obtained 
$b_i$ using a normal distribution centered
at the origin and with standard deviation
$\sigma=2$. We look at a total of 100 different randomly generated problems and later average over the results.
We compare SGD and our method  \eqref{eq:semi_impl_rule}.
At each iteration we use a mini-batch of 
$\ell = 10$ elements of $\{ \Phi_1,\ldots,\Phi_N\}$
to compute the stochastic 
approximation of the gradient of $\bar\Phi$.
Finally, the discretisation step-size has been chosen 
polynomially decreasing, namely at the 
$n$-th iteration we set 
$h_n= h_0/n$. {This is to provide a numerical simulation for the theoretical framework analysed in Section~\ref{sec_losingboth}, where the re-scaling of the process $(i(t))_{t\geq0}$ has the effect of progressively decreasing the learning rate.} 
The results are reported in
Figure~\ref{fig:convex_exp}.
As we can see, SGMP shows
a faster convergence than the classical
SGD scheme. In this case,  decreasing the
mass parameter $\mass$ leads to a deterioration
of the performances.
Indeed, consistently with the theoretical
predictions, the behaviour of SGMP
gets closer to the stochastic gradient process as $\mass$ diminishes.
We also study the case where the mass is non-constant and decreases over time as  $m_k=m_0(0.995)^k$. In that situation, the decrement of the mass leads to a slower convergence as well.
\begin{figure}[htb]
    \centering
    \includegraphics[scale = 0.45]{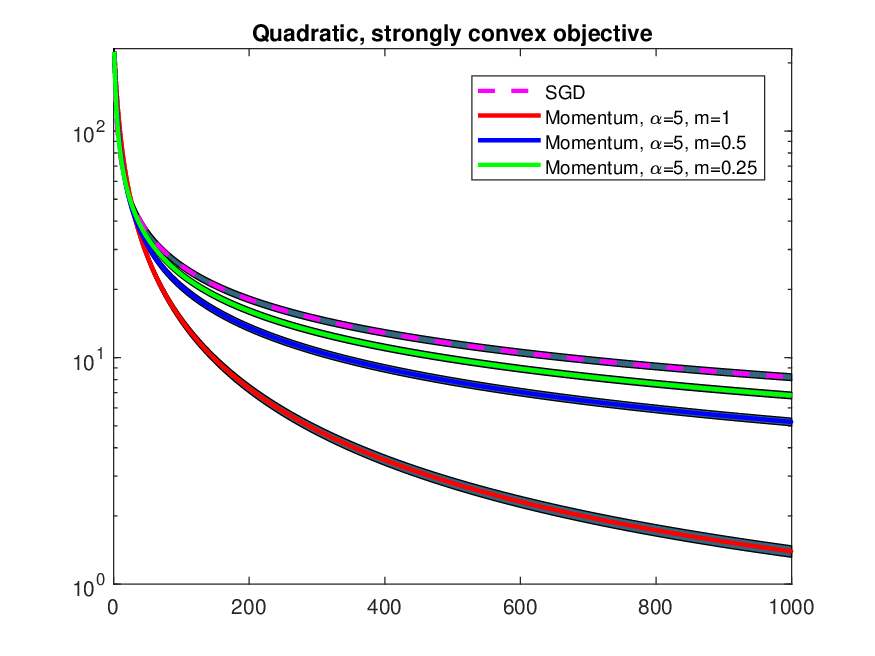}
    \includegraphics[scale = 0.45]{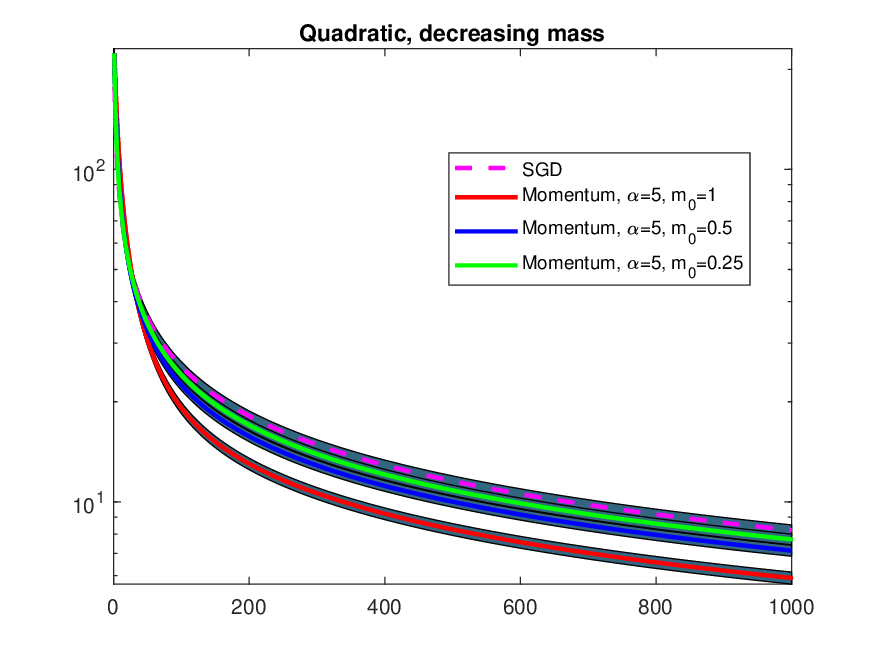}
    \caption{Convergence rate comparison.
    The plot represents the decreasing  
    distance from the true minimiser achieved
    by the SGD and SGMP
    introduced in \eqref{eq:stoc_meth}.
    We consider the case where the mass is constant, but the step-size $(h_n)_{n=0}^\infty$ is decreasing (left) and the case where the mass is {reduced} as $\mass_k=\mass_0(0.995)^k$ and the step-size decreases as before (right, see Subsection~\ref{sec_losingmomentum}).
    The experiments are repeated $100$ times
    (always resampling the potentials), and we
    report the mean distance achieved by each method, 
    and the corresponding standard deviation.
    }
    \label{fig:convex_exp}
\end{figure}

\subsection{Non-convex example: polynomial function} \label{subsec:polyn}
We studied the behaviour of our
SGMP
in the case of a polynomial non-convex function given as the sum
$\bar\Phi=\frac1N \sum_{i=1}^N\Phi_i$, where for 
every $i=1,\ldots,N$ 
\begin{equation*}
    \Phi_i(x) = -\frac{N}{2} x^TA_ix +N b_ix + \frac1{4}
    \sum_{j=1}^nx_j^4
\end{equation*}
and where $A_i\in \R^{K\times K}$ is a 
symmetric and positive definite
matrix and $b_i\in \R^K$.
The problem is non-convex, since 
the Hessian of $\bar\Phi$ (as well as
the one of each $\Phi_1,\ldots,\Phi_N$)
is negative definite at the origin $x=0$.
On the other hand, outside a large enough
compact set, the objective function
$\Phi$ is locally convex.
We considered $K=500$ (dimension of the 
domain), and $N=100$. 
For every $i=1,\ldots,N$, we have sampled the
eigenvalues of the matrix $A_i$ 
using a uniform distribution in 
$[0.05,15]$, and we obtained 
$b_i$ using a normal distribution centered
at the origin and with standard deviation
$2$.
At each iteration we use a mini-batch of 
$\ell= 5$ elements of $\{ \Phi_1,\ldots,\Phi_N\}$
to compute the stochastic 
approximation of the gradient of $\Phi${, and we keep the learning rate constant along the iterations.}
We have compared the SGMP with SGD.
In this case, the learning rate is kept
constant during the iterations.
The results are reported in
Figure~\ref{fig:non_convex}.
In this case it seems that the stochastic momentum
method tends to stabilise in correspondence of 
lower values of the objective function.
Interestingly,
    we observe that the implementations with
    smaller $\mass$ have a faster decay 
    in the initial iterations -- as opposed to the results obtained in the previous subsection.

\begin{figure}[htb]
    \centering
    \includegraphics[scale=0.45]{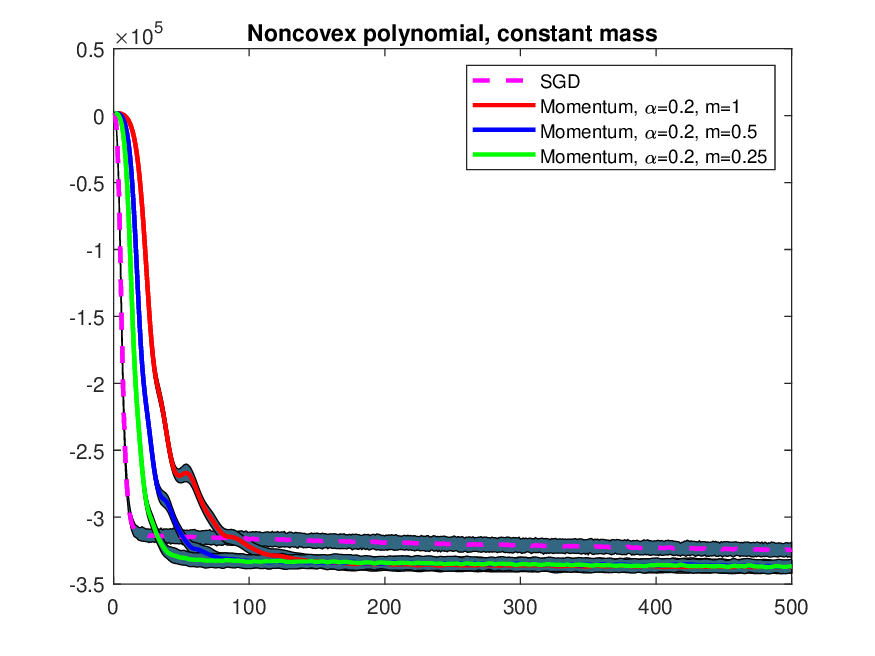}
    \includegraphics[scale=0.45]{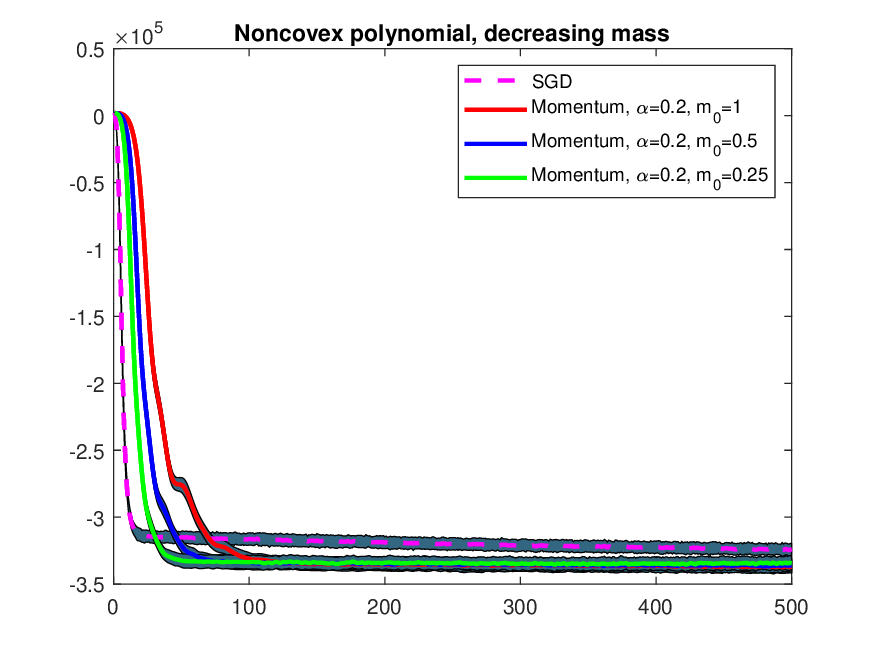}
    \caption{Decay of objective comparison.
    The plot represents the decreasing 
    objective function achieved by the SGD
    and SGMP
    introduced in \eqref{eq:stoc_meth}.
    We consider the constant-mass regime
    (left) and the exponentially decreasing
    case $\mass_k=\mass_0(0.995)^k$ (right).
    {Experiments are run $100$ times}, and
    we reported the mean objective decrease achieved
    by the methods, and the respective standard deviations.
    }
    \label{fig:non_convex}
\end{figure}

\subsection{Convolutional Neural Network (CNN)} \label{subsec:cnn}

We now employ the discrete SGMP method (\ref{eq:semi_impl_rule}) to solve the CIFAR-10 \citep{CIFAR10}, image classification task with a convolutional neural network (CNN). The CIFAR-10 data set consists of $6 \cdot 10^4$  colour images ($32\times 32$ pixels) which are split into $5 \cdot 10^4$ training images and $10^4$ test images. CIFAR-10 has 10 classes (e.g.\ airplane, dog, frog,...) with 6000 images per class. In the classification task, images with known class are used to train the CNN to automatically recognise the class of any image. We use a VGG-like CNN architecture \citep{vgg}. More precisely, we use $3\times 3$ kernels with depth $32$, $64$, and $128$. The network contains 6 convolutional layers, each of them followed by the ReLU activation, batch normalisation, max-pooling, and drop-out layers. The train data is augmented as in \cite{DataAug}, that is a random horizontal flip and a $32\times32$ random crop after a $4\times4$ padding.
The training is done with Google Colab using GPUs (often Tesla V100, sometime Tesla A100). We compare SGD, classical momentum, and SGMP (\ref{eq:semi_impl_rule}) with different parameters, constant mass, and decreasing mass. The experiments are set in the following way. We train for $800$ epochs with batch size $\ell= 100$ and no weight decay. We use constant learning rate $\eta=0.01$ for SGD and the classical momentum. In classical momentum, we set the momentum hyperparameter $\rho=0.9$. 
See Figure \ref{fig:loss} for the plots of train loss for constant mass. In Figure \ref{fig:loss}, note that with fixed $\alpha$ and $h$, hSGMP performs better with small mass $m$. We have discussed the stability of (\ref{eq:semi_impl_rule}) with small $m$ and (relatively) large stepsize in \ref{subsection_stepsizechoice}. In addition, when $m\downarrow0$, the red line (hSGMP) converges to the blue line (SGD), which aligns with Theorem \ref{second th}.
Furthermore, we observe that with fixed step-size $h$, a small ratio of $\alpha$ and $m$ gives better results suggesting a possible rule for choosing hyperparameter in practice. For decreasing mass, see Figure \ref{fig:decreasing_loss} for the plots of train loss for dmSGMP. The decreasing rate is set to be 
$m = m_0 0.995^k$, where $k$ is the number of iterations and $m_0=0.1$ is the initial mass. The train loss for each epoch is calculated by averaging over batches. {In Figure \ref{fig:decreasing_loss}, dmSGMP achieves a similar train loss as the classical momentum method. They might arrive at different minimisers, but this shows that dmSGMP is a practical algorithm for machine learning optimisation that we are able to lay a solid analytical foundation for. For this particular task, dmSGMP has attained a desirable result without decreasing learning rate -- from Corollary \ref{cor:convmi} we would expect a behaviour similar to SGD. For larger datasets and more complicated tasks that we are not able to experiment due to lack of resource, we conjecture that ddSGMP may be an effective optimisation method with proper choice of hyperparameters.} See Table \ref{tab:acc} for the train and test accuracy. Accuracy is measured using the model obtained after 800 epochs of training. 
{Overall, we conclude that hSGMP achieves competitive test accuracy compared with  classical momentum. When decreasing the mass, dmSGMP achieves competitive train loss compared to classical momentum.}

\begin{figure}[htb]
  \centering
    \includegraphics[width=1\textwidth]{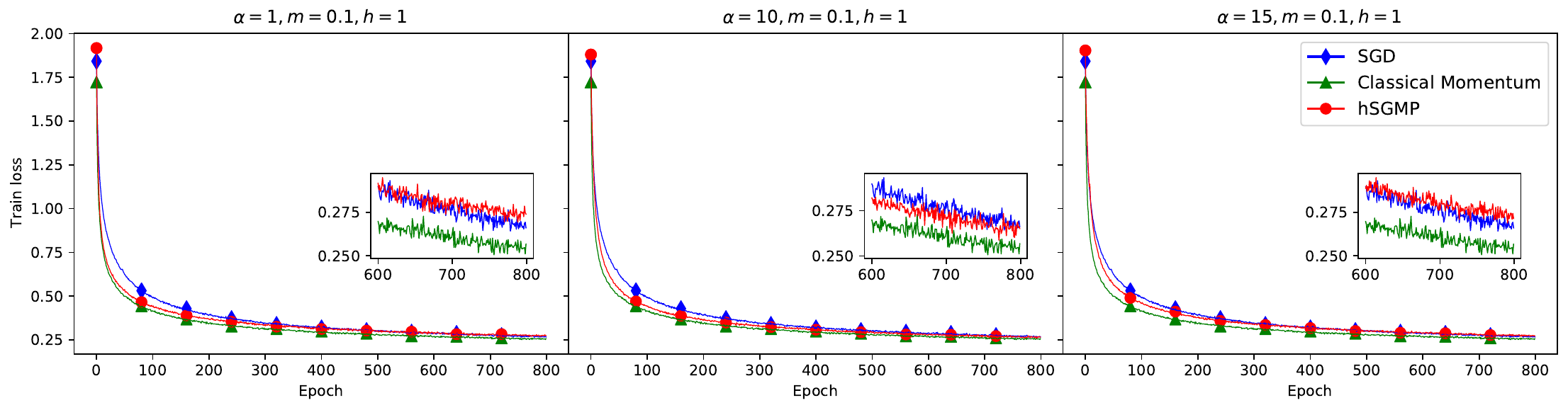}
    \includegraphics[width=1\textwidth]{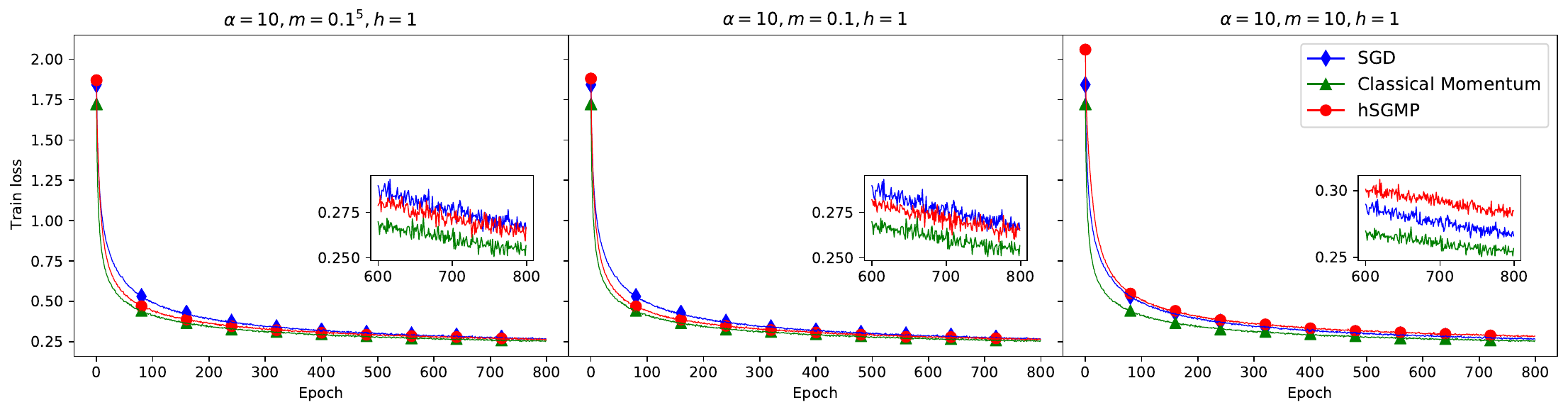}
    \includegraphics[width=1\textwidth]{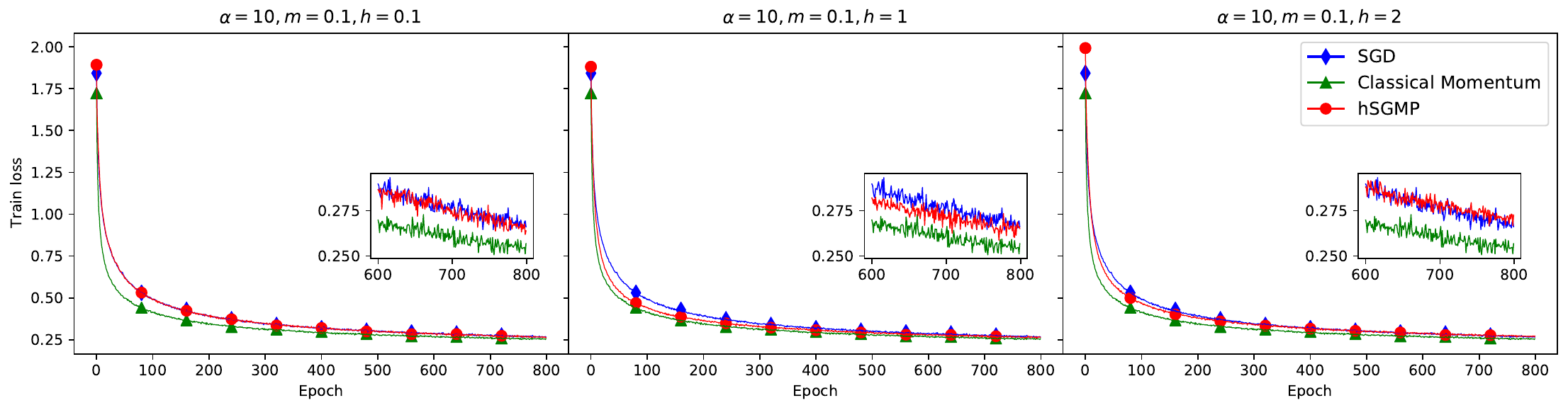}
\caption{Train loss comparison for SGD ($\eta=0.01$), classical momentum ($\eta=0.01, \rho =0.9$), and hSGMP for CNN on CIFAR-10. We vary $\alpha$, $\e$ and $h$ in each experiment, which are specified in the title of each plot. }\label{fig:loss}
\end{figure}

\begin{figure}[htb]
  \centering
    \includegraphics[width=.4\textwidth]{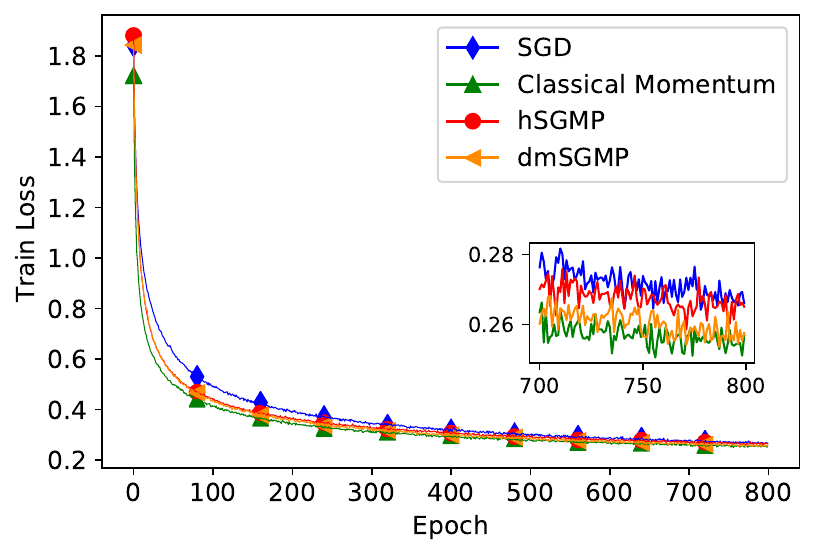}
\caption{Train loss comparison for SGD ($\eta=0.01$), classical momentum ($\eta=0.01, \rho =0.9$), hSGMP ($\alpha=10, \e=0.1, h=1$), and dmSGMP ($\alpha=10, \e_0=0.1, m = m_0 0.995^k, h=1$) for CNN on CIFAR-10.}\label{fig:decreasing_loss}
\end{figure}

\begin{table}[htb]
\begin{tabular}{l|l|l|l}
\textbf{Method} & \textbf{Parameters} & \textbf{Train Acc.} & \textbf{Test Acc.} \\ \hline
{SGD}   &    $\eta = 0.01$        &     $96.902$      &     $90.98$       \\ \hline
{Classical momentum}  &       $\eta= 0.01,\ \rho=0.9$    &    $97.236$     &  $91.33$       \\ \hline
{hSGMP} &    $\alpha=1,\ \e=0.1,\ h=1$            &    $97.096$     &  $91.09$     \\
  & $\alpha=10,\ \e=0.1,\ h=1$    &   $96.986$  &      $91.25$   \\
  &     $\alpha=15,\ \e=0.1,\ h=1$    &   $96.822$         &  $91.09$    \\
 &     $\alpha=10,\ \e=10,\ h=1$  &   $96.556$          &  $90.75$        \\
 &     $\alpha=10,\ \e=0.1^5,\ h=1$     &   97.35 &  91.35  \\
&     $\alpha=10,\ \e=0.1,\ h=2$    &   $96.984$          &  $90.87$    \\
 &     $\alpha=10,\ \e=0.1,\ h=0.1$   &   $97.108$     &  $91.26$ \\
 \hline
{dmSGMP} &    $\alpha=10,\ \e_0=0.1,\ h=1$            &    $97.288$     &  $91.32$  
\end{tabular}\caption{Comparison of train and test accuracy (\%) with different hyperparameters over the CIFAR-10 dataset.}
 \label{tab:acc}
\end{table}

\section{Conclusions}\label{sec_conc}
In this work, we have proposed and analysed the stochastic gradient-momentum process, a continuous-time dynamics representing momentum-based stochastic optimisation. We have especially analysed limiting behaviour when reducing learning rate and/or particle mass. In this context, learning rate and particle mass can either be reduced homogeneously or decrease over time. We have shown pathwise or longtime convergence to the underlying gradient flow or the stochastic gradient process,  respectively. We have then proposed a stable discretisation strategy for the stochastic gradient-momentum process and tested the strategy in several numerical examples. In those, we especially saw that the stable discretisation of the stochastic gradient-momentum process can achieve a similar accuracy in a CNN training compared with (the possibly unstable) stochastic gradient descent with classical momentum algorithm.

Most of the theoretical results we have obtained throughout this work refer to the setting of convex optimisation. Convex optimisation is vital in, e.g.\ image reconstruction. The training of neural network usually requires non-convex optimisation and momentum-based methods are especially popular in non-convex settings. Hence, a natural future research direction are non-convex optimisation problems.
The stochastic gradient-momentum process does not represent the adaptivity in the Adam algorithm.  To represent the adaptivity, we would need to study a two-sided dependence between $(i(t))_{t \geq 0}$ and $(p(t),q(t))_{t \geq 0}$ and a non-linear weighting in front of the gradient. Both these additions to the stochastic gradient-momentum process are a very interesting and challenging direction for future research.


\acks{A.S. acknowledges partial support from INdAM--GNAMPA.}


\appendix

\section{Longtime behavior of the underdamped gradient flow}\label{appendix:underdamped}
We recall that the underdamped gradient flow (\ref{eq:AS:pq}) is defined as the following dynamical system
\begin{equation*}
\left\{ \begin{array}{rl}
\de q_t &= p_t \de t,\\
\de p_t &= - \nabla \bar\Phi(q_t) \de t-\alpha p_t \de t, \\
p_0, q_0 &\in X.
\end{array} \right.
\end{equation*}
We now show that the position in the underdamped gradient flow converges to the global minimiser of the objective function $\bar\Phi$  and the velocity converges to 0. 
We shall assume that $\bar\Phi$ satisfies Assumption \ref{as1.2} with $\mathcal{A}^{\lambda, \alpha}$ for some $\lambda$ and $\alpha$ defined in \eqref{eq:AS:pq}.
First, we note that since $\bar\Phi$ achieves a global minimum at $\theta_*$, $\theta_*$ is a critical point of $\bar\Phi$. Then, (\ref{as:phi}) implies  that $\theta_*$ is the unique critical point of $\bar\Phi$.
{We use a technique similar to that of, e.g.\ \cite{doi:10.1137/21M1403990, Lojasiewicz, EGZ,  doi:10.1137/19M1272767}.} Indeed, we introduce the  Lyapunov function $V: \R^n\times \R^n\to \R$ given by
\begin{align}\label{defV}
    V(x,y):=&  \bar\Phi(x)-\bar\Phi(\theta_*)+ \frac{\alpha^2}{4}\Big(\norm{x-\theta_*+\alpha^{-1}y}^2+\norm{\alpha^{-1}y}^2-\lambda\norm{x-\theta_*}^2\Big)
\end{align} 
which we employ below. Using $V$, we can show Lyapunov exponential global stability for the underdamped gradient flow \eqref{eq:AS:pq}.
\begin{proposition}\label{th2}
Assume that $\bar\Phi$ satisfies Assumption \ref{asSGPf} and \ref{as1.2} with $\mathcal{A}^{\lambda, \alpha}$.  Let $(p_t,q_t)_{t \geq 0}$ be the solution to (\ref{eq:AS:pq}) and $V$ be the Lyapunov function defined in \eqref{defV}. Then we have the following inequality
\begin{align*}
V(q_t,p_t)\le V(q_0,p_0)e^{-\alpha\lambda t}  \qquad (t \geq 0).
\end{align*}
\end{proposition}
\begin{proof}
Notice that \eqref{as:phi} implies that $\frac{\mathrm{d}V(q_t,p_t)}{\mathrm{d}t} \leq -\alpha\lambda V(q_t,p_t)$.
By Gr\"onwall's inequality, we obtain
$
V(q_t,p_t)\le V(q_0,p_0)e^{-\alpha\lambda t}.
$
\end{proof}
In the previous proposition, we show that the Lyapunov function $V(q_t,p_t)$ is bounded above by a number that exponentially decreases in $t > 0$. $V(q_t,p_t)$ is zero, if $q_t = \theta_*$ and $p_t = 0$, i.e.\  the particle is positioned at the unique minimiser of $\bar\Phi$ and the particle velocity is 0: the dynamical system {has reached a stationary state.} However, $V$ can have multiple zeros. Thus, we need some more work to show exponential convergence of the dynamical system.
\begin{corollary} \label{cor_det_conv}
Under the conditions of Proposition \ref{th2}, we have
$$
\norm{p_t}^2+\norm{q_t-\theta_*}^2\le C_{\alpha,\lambda}e^{-\alpha\lambda t}V(q_0,p_0)  \qquad (t \geq 0).
$$
\end{corollary}
\begin{proof}
Since $\lambda\le \frac{1}{4},$ we have, by the $\varepsilon$-Young inequality:
\begin{align*}
    V(x,y)&\ge \frac{\alpha^2}{4}\Big((1-\lambda)\norm{x-\theta_*}^2+2\norm{\alpha^{-1}y}^2+2(x-\theta_*)\cdot (\alpha^{-1}y)\Big)\\
    &\ge \frac{\alpha^2}{4}\Big((1-\lambda)\norm{x-\theta_*}^2+2\norm{\alpha^{-1}y}^2-2\norm{\alpha^{-1}y}^2-\frac{1}{2}\norm{x-\theta_*}^2\Big)\\
    &\ge \frac{\alpha^2}{4}\left(\frac{1}{2}-\lambda\right)\norm{x-\theta_*}^2,
\end{align*}
which implies that
\begin{align*}
    \norm{q_t-\theta_*}^2\le \frac{8}{\alpha^2(1-2\lambda)}V(q_t,p_t)\le \frac{8}{\alpha^2(1-2\lambda)} V(q_0,p_0)e^{-\alpha\lambda t}.
\end{align*}
Similarly, again by the $\varepsilon$-Young inequality, we have  for $\norm{p_t}$ that
\begin{align*}
    V(x,y)&\ge \frac{\alpha^2}{4}\Big((1-\lambda)\norm{x-\theta_*}^2+2\norm{\alpha^{-1}y}^2+2(x-\theta_*)\cdot (\alpha^{-1}y)\Big)\\
    &\ge \frac{\alpha^2}{4}\Big((1-\lambda)\norm{x-\theta_*}^2+2\norm{\alpha^{-1}y}^2-\frac{1}{1-\lambda}\norm{\alpha^{-1}y}^2-(1-\lambda)\norm{x-\theta_*}^2\Big)\\
    &\ge \frac{\alpha^2}{4}\left(\frac{1-2\lambda}{1-\lambda}\norm{\alpha^{-1}y}^2\right),
\end{align*}
which implies 
\begin{align*}
    \norm{p_t}^2\le  \frac{4(1-\lambda)}{1-2\lambda}e^{-\alpha\lambda t} V(q_0,p_0)\le e^{-\alpha\lambda t} V(q_0,p_0).
\end{align*}
Therefore, we conclude that
\begin{align*}
    \norm{q_t-\theta_*}^2+\norm{p_t}^2\le C_{\alpha,\lambda}e^{-\alpha\lambda t} V(q_0,p_0),
\end{align*}
where $C_{\alpha,\lambda}=4+\frac{8}{\alpha^2(1-2\lambda)}.$
\end{proof}

\section{Decreasing momentum: the fixed-sample case} \label{appendix_decreasing_momentum}
We now consider the SGMP dynamics subject to a fixed sample $i \in I$.  More specifically,  we define
\begin{equation}\label{eq:AS:pqii}
\left\{ \begin{array}{rl}
\de q^{i,\e}_t &= p^{i,\e}_t\de t,\\
 \e \de p^{i,\e}_t &= - \nabla \Phi_i(q^{i,\e}_t)\de t-  p^{i,\e}_t\de t, \\
p^{i,\e}_0 &= p^i_0 \in X,\\ q^{i,\e}_0 &= q^i_0 \in X,
\end{array} \right.
\end{equation}
with mass $\e>0$.
When $\e\to 0$, we have the following formal limiting equation
\begin{equation} \label{Eq_SGPC}
\left\{ \begin{array}{ll}
    \de  \theta^i_t &= - \nabla \Phi_i(\theta^i_t) \de t,\\
    \theta^i_0 &\in X.
\end{array} \right.
\end{equation}
This is the gradient flow with respect to the potential $\Phi_i$ {for some $i \in I$}.  To prove this described limiting behaviour, we require three auxiliary results. In the first one we study, similarly to Corollary~\ref{cor_det_conv}, the longtime behaviour of the deterministic underdamped dynamical system (\ref{eq:AS:pqii}). This time with emphasis on the influence of the mass $\e$. Interestingly, we can see that the convergence rate is independent of the mass $\e$, if it is sufficiently small. We note that these results are similar to previous results by \cite{attouch2000heavy-2}.
\begin{lemma}\label{lem:dqi} Let $(p^{i,\e}_t,q^{i,\e}_t)_{t \geq 0}$ be the solution to (\ref{eq:AS:pqii}). 
Let $\Phi_i$ satisfy Assumption \ref{asSGPf} and \ref{as1.2} with $\mathcal{A}^{\lambda_i, 1}$ and critical point $\theta^i_*$. We set $\lambda:=\min_{i\in I}\lambda_i$. Then for $0<\e\le 1$, we have the following inequality
\begin{align}\label{lemma1}
    \norm{q^{i,\e}_t-\theta^i_*}^2\le 16 (L+2) e^{-\lambda t}(\norm{q^i_0-\theta^i_*}^2+m^2\norm{p^i_0}^2)  \qquad (t \geq 0),
\end{align}
where $L$ is the Lipschitz constant of the $\nabla \Phi_i$.
\end{lemma}

\begin{proof}
We recall Assumption \ref{as1.2} for $\Phi_i$ with $\alpha=1$:
$$
(x-\theta^i_*)\cdot \nabla \Phi_i(x)/2\ge \lambda (\Phi_i(x)-\Phi_i(\theta^i_*)+\norm{x-\theta^i_*}^2/4),
$$
which implies that
\begin{align*}
    (x-\theta^i_*)\cdot \nabla \Phi_i(x)/(2\e)\ge \lambda (\e^{-1}\Phi_i(x)-\e^{-1}\Phi_i(\theta^i_*)+\e^{-1}\norm{x-\theta^i_*}^2/4).
\end{align*}
And since $0<\e\le 1,$ we have 
\begin{align*}
    \lambda (\e^{-1}\Phi_i(x)-\e^{-1}\Phi_i&(\theta^i_*)+\e^{-1}\norm{x-\theta^i_*}^2/4) \\ &= \lambda\e (\e^{-2}\Phi_i(x)-\e^{-2}\Phi_i(\theta^i_*)+\e^{-2}\norm{x-\theta^i_*}^2/4)\\
    &\ge \lambda\e (\e^{-1}\Phi_i(x)-\e^{-1}\Phi_i(\theta^i_*)+\e^{-2}\norm{x-\theta^i_*}^2/4).
\end{align*}
So Assumption \ref{as1.2} implies that
$$
(x-\theta^i_*)\cdot \nabla \Phi_i(x)/(2\e)\ge\lambda\e (\e^{-1}\Phi_i(x)-\e^{-1}\Phi_i(\theta^i_*)+\e^{-2}\norm{x-\theta^i_*}^2/4)
$$
which  exactly means that $\Phi_i(x)/\e$ satisfies Assumption \ref{as1.2} with $\mathcal{A}^{\lambda m, m^{-1}}$. We define
\begin{align*}
    V^{i,\e}(x,y)=\frac{\Phi_i(x)-\Phi_i(\theta^i_*)}{\e}+ \frac{1}{4\e^2}\Big(\norm{x-\theta^i_*+\e y}^2+\norm{\e y}^2-\e\lambda\norm{x-\theta^i_*}^2\Big). 
\end{align*}
From Proposition \ref{th2}, we have
\begin{align*}
    V^{i,\e}(q^{i,\e}_t,q^{i,\e}_t)\le V^{i,\e}(q^i_0,p^i_0)e^{-\lambda t}.
\end{align*}
Next, we are going to show that
\begin{align*}
  c\e^{-2}\norm{x-\theta^i_*}^2\le V^{i,\e}(x,y)\le (L+2)\Big(\e^{-2}\norm{x-\theta^i_*}^2+\norm{y}^2\Big),
\end{align*}
where $c= \frac{1}{16}$.

Since $\Phi_i(x)-\Phi_i(\theta^i_*)\ge 0$ and $\norm{x-\theta^i_*+\e y}^2+\norm{\e y}^2\ge \frac{\norm{x-\theta^i_*}^2}{2},$ we have
\begin{align*}
    V^{i,\e}(x,y)&\ge \frac{1}{4\e^2}\Big(\norm{x-\theta^i_*+\e y}^2+\norm{\e y}^2-\e\lambda\norm{x-\theta^i_*}^2\Big) \\
    &\ge \frac{1}{4\e^2}\Big(\frac{1}{2}-\e\lambda\Big)\norm{x-\theta^i_*}^2\ge \frac{\norm{x-\theta^i_*}^2}{16\e^2}.
\end{align*}
Notice that
\begin{align*}
    \Phi_i(x)-\Phi_i(\theta^i_*)= \int_0^1 (x-\theta^i_*)\Big[\nabla \Phi_i\Big((x-\theta^i_*)s+\theta^i_*)\Big)-\nabla \Phi_i(\theta_*)\Big]ds\le \frac{L\norm{x-\theta^i_*}^2}{2}.
\end{align*}
This implies that
\begin{align*}
    V^{i,\e}(x,y)&\le \frac{L\norm{x-\theta^i_*}^2}{2\e}+ \frac{1}{2\e^2}\Big(\norm{x-\theta^i_*}^2+\norm{\e y}^2\Big)\\
    &\le (L+2)\Big(\e^{-2}\norm{x-\theta^i_*}^2+\norm{y}^2\Big).
\end{align*}
Hence, we immediately get 
\begin{align*}
    \e^{-2}\norm{q^{i,\e}_t-\theta^i_*}^2\le 16 (L+2) e^{-\lambda t}(\e^{-2}\norm{q^i_0-\theta^i_*}^2+\norm{p^i_0}^2),
\end{align*}
which implies that
\begin{align*}
    \norm{q^{i,\e}_t-\theta^i_*}^2\le 16 (L+2) e^{-\lambda t}(\norm{q^i_0-\theta^i_*}^2+\e^2\norm{p^i_0}^2).
\end{align*}
\end{proof}

In the next auxiliary result, we show boundedness of the velocity, i.e.\  $(p^{i,\e}_t)_{t \geq 0}$ that depends on the mass $\e$. Moreover, we show Lipschitz continuity of the particle position with respect to time. Note that the Lipschitz constant can be chosen independently of $\e$.

\begin{lemma}\label{lem:bound+Lip} {Let $(p^{i,\e}_t,q^{i,\e}_t)_{t \geq 0}$ be the solution to (\ref{eq:AS:pqii}). Let $\Phi_i$ satisfy Assumption \ref{asSGPf} and \ref{as1.2} with $\mathcal{A}^{\lambda_i, 1}$ and critical point $\theta^i_*$. Then for $0<\e\le 1$,} we have for any $0\le s\le t$, 
\begin{align*}
\norm{p^{i,\e}_t}&\le \Big(e^\frac{-t}{\e}+C_L\e\Big)\norm{p^i_0}+C_L  \norm{q^i_0-\theta^i_*}, \\  
    \norm{q^{i,\e}_t-q^{i,\e}_s}&\le C_L(t-s)\Big(\norm{q^i_0-\theta^i_*}+\norm{p^i_0}\Big) \qquad (t \geq 0),
\end{align*}
{where $C_L=16(L+2)L + 1$ and $L$ is the Lipschitz constant of $\nabla \Phi_i$.}
\end{lemma}
\begin{proof}
 We take the derivative of $e^\frac{t}{\e}p^{i,\e}_t$, from \eqref{eq:AS:pqii}, we get
 
\begin{align}\label{pie}
    p^{i,\e}_t= e^\frac{-t}{\e}p^i_0-\e^{-1}e^\frac{-t}{\e} \int_0^te^\frac{s}{\e}\nabla \Phi_i(q^{i,\e}_s)\de s.
\end{align}
Hence, we have
\begin{align*}
     \norm{p^{i,\e}_t}
     &\le \norm{e^\frac{-t}{\e}p^i_0}+\e^{-1}e^\frac{-t}{\e} \int_0^te^\frac{s}{\e}\norm{\nabla \Phi_i(q^{i,\e}_s)-\nabla \Phi_i(\theta^i_*)}\de s\\
     &\le e^\frac{-t}{\e}\norm{p^i_0}+\e^{-1}e^\frac{-t}{\e} L\int_0^te^\frac{s}{\e}\norm{q^{i,\e}_s-\theta^i_*}ds\\
     &\overset{(\ref{lemma1}) }{\le} e^\frac{-t}{\e}\norm{p^i_0}+16 (L+2) L \e^{-1}e^\frac{-t}{\e} (\norm{q^i_0-\theta^i_*}+\e\norm{p^i_0})\int_0^te^\frac{s}{\e}\de s\\
     &\le  \Big(e^\frac{-t}{\e}+16(L+2)L\e\Big)\norm{p^i_0}+ 16 (L+2) L  \norm{q^i_0-\theta^i_*}.
\end{align*}
Let $C_L:=16(L+2)L + 1$.
Then we have $\norm{p^{i,\e}_t}\le C_L(\norm{p^i_0}+  \norm{q^i_0-\theta^i_*}).$
From \eqref{eq:AS:pqii}, we immediately get
\begin{align*}
    \norm{q^{i,\e}_t-q^{i,\e}_s}= \norm{\int_s^tp^{i,\e}_m\de m}
    \le \int_s^t\norm{p^{i,\e}_m}\de m\le  C_L(t-s)\Big(\norm{q^i_0-\theta^i_*}+\norm{p^i_0}\Big).
\end{align*}
\end{proof}

The third auxiliary result is a bound on the time derivative of the velocity  $(p_t^{i,m})_{t \geq 0}$, i.e.\  a bound on the particle's acceleration.
\begin{lemma}\label{lem:bound_Accel} {Let $(p^{i,\e}_t,q^{i,\e}_t)_{t \geq 0}$ be the solution to (\ref{eq:AS:pqii}). Let $\Phi_i$ satisfy Assumption \ref{asSGPf} and \ref{as1.2} with $\mathcal{A}^{\lambda_i, 1}$ and critical point $\theta^i_*$. Then for any $0<\e\le 1$, we have}
\begin{align}\label{lemma3}
    \norm{\frac{\de p^{i,\e}_t}{\de t}}\le C^{(0)}_L(1+\e^{-1}e^{-t/\e})\Big(\norm{q^i_0-\theta^i_*}+\norm{p^i_0}\Big)\ \ \ \ (t\ge 0),
\end{align}
where $ C^{(0)}_L=2+ 4LC_L$ {and $L$ is the Lipschitz constant of  $\nabla \Phi_i$.}
\end{lemma}
\begin{proof}
We first recall (\ref{pie}), where we have 
\begin{align}\label{pied}
   \frac{\de p^{i,\e}_t}{\de t}= \frac{-e^\frac{-t}{\e}p^i_0}{\e}-\frac{\nabla \Phi_i(q^{i,\e}_t)}{\e}+\frac{e^\frac{-t}{\e}}{\e^2}\int_0^te^\frac{s}{\e}\nabla \Phi_i(q^i_s)\de s.
\end{align}
and notice that
\begin{align}\label{nomore}
    \frac{\nabla \Phi_i(q^{i,\e}_t)}{\e}=\frac{e^{-\frac{t}{\e}}
}{\e^2}\int_0^t e^{\frac{s}{\e}}\nabla\Phi_i(q^{i,\e}_t)\de s+
\frac{e^{-\frac{t}{\e}}\nabla \Phi_i(q^{i,\e}_t)}{\e}.
\end{align}
Combining (\ref{pied}) and (\ref{nomore}), we get
\begin{align*}
    \norm{\frac{\de p^{i,\e}_t}{\de t}}
    &= e^\frac{-t}{\e}\norm{\frac{-p^i_0}{\e}-\frac{\nabla \Phi_i(q^{i,\e}_t)-\nabla \Phi_i(\theta^i_*)}{\e}+\frac{1}{\e^2}\int_0^te^\frac{s}{\e}(\nabla \Phi_i(q^i_s)-\nabla \Phi_i(q^{i,\e}_t))\de s}\\
    &\le e^\frac{-t}{\e}\Big[\frac{\norm{p^i_0}}{\e}+\frac{L\norm{q^{i,\e}_t-\theta^i_*}}{\e}+\frac{L}{\e^2}\int_0^te^\frac{s}{\e}\norm{q^{i,\e}_s-q^{i,\e}_t}\de s\Big]\\
    &= \e^{-1}e^\frac{-t}{\e}\Big(\norm{p^i_0}+L\norm{q^{i,\e}_t-\theta^i_*}\Big)+\frac{L}{\e^2}\int_0^te^\frac{-(t-s)}{\e}\norm{q^{i,\e}_s-q^{i,\e}_t}\de s\\
    &\le \e^{-1}e^\frac{-t}{\e}\Big(\norm{p^i_0}+L\norm{q^{i,\e}_t-\theta^i_*}\Big)+\frac{LC_L}{\e^2}\int_0^te^\frac{-(t-s)}{\e}(t-s)\de s\Big(\norm{q^i_0-\theta^i_*}+\norm{p^i_0}\Big)\\
    &\underbrace{\le}_{(a_1)} \e^{-1}e^\frac{-t}{\e}\Big(\norm{p^i_0}+4L (L+2) (\norm{q^i_0-\theta^i_*}+\norm{p^i_0})\Big)+LC_L\Big(\norm{q^i_0-\theta^i_*}+\norm{p^i_0}\Big)\\
    &\le  C^{(0)}_L(1+\e^{-1}e^\frac{-t}{\e})\Big(\norm{q^i_0-\theta^i_*}+\norm{p^i_0}\Big),
\end{align*}
{
where $(a_1)$ is correct due to $(\ref{lemma1})$ and the fact that $\int_0^{\frac{t}{\e}}e^{-s}s\de s<1.$  }
\end{proof}

The previous lemmas provide the boundedness results needed for proving the following proposition. We show that the difference between the fixed sample SGMP \eqref{eq:AS:pqii} and the fixed sample stochastic gradient process  \eqref{Eq_SGPC} can be bounded by the sum of the distance between their initial values and a term that is linear in $\e$. Hence, in the fixed subsample case, we show that the solution of (\ref{eq:AS:pqii}) converges to the solution of the limiting equation (\ref{Eq_SGPC}). In addition to the previous assumptions, we now also need to assume $\Phi_i$ to be convex.
\begin{proposition}\label{inductionlem} {Let $\Phi_i$ be convex and satisfy Assumption \ref{asSGPf} and \ref{as1.2} with $\mathcal{A}^{\lambda_i, 1}$ and critical point $\theta^i_*$. Let $(p^{i,\e}_t,q^{i,\e}_t)_{t \geq 0}$ be the solution to (\ref{eq:AS:pqii}).} Then, for $0< \e \le 1$, we have
\begin{align*}
\norm{q^{i,\e}_t-\theta^i_t}\le \norm{q^i_0-\theta^i_0}+ C^{(0)}_L\e (1+t) \Big(\norm{p^i_0}+\norm{q^i_0-\theta^i_0}+\norm{\theta^i_0-\theta^i_*}\Big)  \qquad (t \geq 0),
\end{align*}
{where $ C^{(0)}_L=2+ 4LC_L$ and $L$ is the Lipschitz constant of $\nabla \Phi_i$.}
\end{proposition}

\begin{proof}
 We take the derivative of $\norm{q^{i,\e}_t-\theta^i_t}^2$ and obtain
\begin{align*}
    \frac{1}{2}\frac{\de\norm{q^{i,\e}_t-\theta^i_t}^2}{\de t}&= \ip{q^{i,\e}_t-\theta^i_t}{\frac{\de q^{i,\e}_t}{\de t}-\frac{\de \theta^i_t}{\de t}}\\
    &=-\ip{q^{i,\e}_t-\theta^i_t}{\nabla \Phi_i(q^{i,\e}_t)-\nabla \Phi_i(\theta^i_t) }-\e\ip{q^{i,\e}_t-\theta^i_t}{\frac{\de p^{i,\e}_t}{\de t}}\\
    &\le -\e\ip{q^{i,\e}_t-\theta^i_t}{\frac{\de p^{i,\e}_t}{\de t}}\le \e\norm{q^{i,\e}_t-\theta^i_t}\norm{\frac{\de  p^{i,\e}_t}{\de  t}}\\
    &\overset{(\ref{lemma3})}{\le} C^{(0)}_L\e(1+\e^{-1}e^\frac{-t}{\e})\Big(\norm{q^i_0-\theta^i_*}+\norm{p^i_0}\Big)\norm{q^{i,\e}_t-\theta^i_t},
\end{align*}
which implies 
\begin{align*}
    \frac{\de \norm{q^{i,\e}_t-\theta^i_t}}{\de t}&=\Big(2\norm{q^{i,\e}_t-\theta^i_t}\Big)^{-1}\frac{\de \norm{q^{i,\e}_t-\theta^i_t}^2}{\de t}\\
    &\le C^{(0)}_L\e(1+\e^{-1}e^\frac{-t}{\e})\Big(\norm{q^i_0-\theta^i_*}+\norm{p^i_0}\Big).
\end{align*}
Integrating both sides, we get
\begin{align*}
    \norm{q^{i,\e}_t-\theta^i_t}&\le \norm{q^i_0-\theta^i_0}+ C^{(0)}_L \int_0^t (\e+ e^\frac{-s}{\e})\de s\Big(\norm{p^i_0}+\norm{q^i_0-\theta^i_*}\Big)\\
    &=\norm{q^i_0-\theta^i_0}+ C^{(0)}_L \Big(\e t+\e(1-e^{\frac{-t}{\e}}) \Big)\Big(\norm{p^i_0}+\norm{q^i_0-\theta^i_*}\Big)\\
    &\le \norm{q^i_0-\theta^i_0}+ C^{(0)}_L\e (1+t) \Big(\norm{p^i_0}+\norm{q^i_0-\theta^i_*}\Big)\\
    &\le \norm{q^i_0-\theta^i_0}+ C^{(0)}_L\e (1+t) \Big(\norm{p^i_0}+\norm{q^i_0-\theta^i_0}+\norm{\theta^i_0-\theta^i_*}\Big).
\end{align*}
\end{proof}

\section{Other auxiliary results} \label{appendix}

\begin{example}\label{example:assp2}
Let $\bar\Phi$ be a function on $\mathbb{R}$ satisfying
\begin{equation*}
\bar\Phi(x) = \left\{ \begin{array}{l}
 x^2, \qquad  \ \ \ \ \ \  x \in (-\infty, 1], \\
2x-1, \qquad  \  x \in (1,2],\\
\frac{1}{2}x^2+1, \qquad   x \in (2,+\infty).
\end{array} \right.
\end{equation*} 
$\bar\Phi$ satisfies Assumption~\ref{as1.2} with $\lambda \in (0, 2/(8 + \bar\alpha^2)]$. {As shown in Figure \ref{figure:example1},} $\bar\Phi$ is convex but not strongly convex since it is linear on $(1,2]$. 

\begin{figure}
    \centering
    \includegraphics[scale = 0.6]{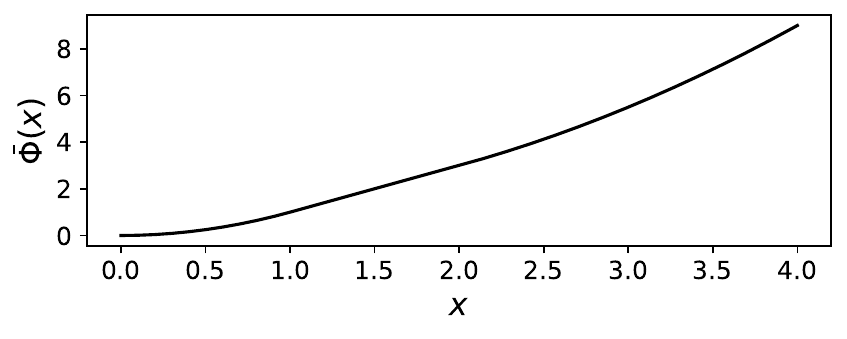}
    \vspace{-3mm}
    \caption{{Plot of function $\bar\Phi$ in Example \ref{example:assp2}.}}
    \label{figure:example1}
\end{figure}
\end{example}

\begin{lemma}\label{lem:conx}
Let $\bar\Phi\in \mathcal{C}^{1}(X,\R)$ be strongly convex with constant $\mu$. Then $\bar\Phi$ satisfies Assumption \ref{as1.2} with  $\mathcal{A}^{(\mu/\alpha^2)\land (1/4), \alpha}$.
\end{lemma}

\begin{proof}
Since $\bar\Phi$ is strongly convex, we have 
$$
(x-\theta_*)\cdot (\nabla \bar\Phi(x)-\nabla \bar\Phi(\theta_*))\ge \mu \norm{x-\theta_*}^2.
$$
By Lagrange's mean value theorem, there exist some $\xi$ lying on the line between $x$ and $\theta_*$, such  that
$$
\bar\Phi(x)-\bar\Phi(\theta_*)= (x-\theta_*)\cdot\nabla \bar\Phi(\xi).
$$
Let $\xi= \theta_*+t(x-\theta_*)$ for some $0\le t\le 1.$ By strong convexity, we have
\begin{align*}
    (x-\xi)\cdot (\nabla \bar\Phi(x)-\nabla \bar\Phi(\xi))\ge \mu \norm{x-\xi}^2.
\end{align*}
Replacing $\xi$ by $\theta_*+t(x-\theta_*)$, we get
\begin{align*}
    (x-\theta_*)\cdot (\nabla \bar\Phi(x)-\nabla \bar\Phi(\xi))\ge(1-t)(x-\theta_*)\cdot (\nabla \bar\Phi(x)-\nabla \bar\Phi(\xi))\ge \mu \norm{x-\xi}^2.
\end{align*}
This implies 
\begin{align*}
     (x-\theta_*)\cdot \nabla \bar\Phi(x)\ge (x-\theta_*)\cdot\nabla \bar\Phi(\xi)+\mu \norm{x-\xi}^2\ge \bar\Phi(x)-\bar\Phi(\theta_*) .
\end{align*}
Hence,
\begin{align*}
    (x-\theta_*)\cdot \nabla \Phi(x)/2\ge  (\bar\Phi(x)-\bar\Phi(\theta_*))/4+\mu\norm{x-\theta_*}^2/4).
\end{align*}
Then, we choose $\lambda=(\mu/\alpha^2)\land (1/4)$ and the proof is completed.
\end{proof}

\begin{lemma}\label{thetabound}
We assume that $\Phi_i$ is convex and satisfies Assumption \ref{asSGPf} for all $i\in I.$ Let $(\theta^\e_t)_{t\geq0}$ be the solution to system (\ref{Eq_SGPCed}). Then we have
\begin{align*}
    \norm{\theta^\e_t}\le \norm{\theta_0}+C_\Phi t  \qquad (t \geq 0). 
\end{align*}
\end{lemma}

\begin{proof}
We differentiate $\norm{\theta^\e_t}^2$ and obtain
\begin{align*}
    \frac{\de \norm{\theta^\e_t}^2}{\de t}&=2\ip{\theta^\e_t}{\frac{\de \theta^\e_t}{\de t}}=-2\ip{\theta^\e_t}{\nabla \Phi_{i^{\e,\delta}(t)}(\theta^\e_t)}\\
    &=-2\ip{\theta^\e_t,\nabla \Phi_{i^{\e,\delta}(t)}(\theta^\e_t)-\nabla \Phi_{i^{\e,\delta}(t)}(0)}-2\ip{\theta^\e_t}{\nabla \Phi(0)}\\
    &\le 2\norm{\theta^\e_t}\norm{\nabla \Phi_{i^{\e,\delta}(t)}(0)}\le 2C_\Phi \norm{\theta^\e_t},
\end{align*}
which implies 
$\frac{\de \norm{\theta^\e_t}}{\de t}\le C_\Phi$ and, hence,
$\norm{\theta^\e_t}\le \norm{\theta_0}+C_\Phi t.$
\end{proof}

\begin{lemma}\label{strongthetabound}
We assume that $\Phi_i$ is strongly  convex with $\mu>0$ for all $i\in I.$ Let $(\theta_t)_{t\geq0}$ be the solution to (\ref{eq:SGP_old}). Then, we have
\begin{align*}
    \norm{\theta_t}\le \norm{\theta_0}e^{-\mu t}+C_\Phi \qquad (t \geq 0). 
\end{align*}
\end{lemma}
\begin{proof}
We differentiate $\norm{\theta_t}^2$,
\begin{align*}
    \frac{\de \norm{\theta_t}^2}{\de t}
    &\le -2\mu \norm{\theta_t}^2  +2\norm{\theta_t}\norm{\nabla \Phi_{i(t)}(0)}\le -\mu\norm{\theta_t}^2+\frac{1}{\mu}\norm{\nabla \Phi_{i(t)}(0)}^2,
\end{align*}
By Gr\"onwall's inequality, we have 
\begin{align*}
   \norm{\theta_t}^2\le e^{-\mu t}\norm{\theta_0}^2+\frac{1}{\mu^2}\norm{\nabla \Phi_{i(t)}(0)}^2
\end{align*}
and choose $C_\Phi= \frac{1}{\mu^2}\sup_{i=1,...,N}\norm{\nabla \Phi_i(0)}^2$.
\end{proof}

\begin{lemma}\label{tn}
Let $C$ and $\lambda$ be two positive constants and let $(t_n)_{n\ge 0}$  be  a sequence with $t_0\ge 0$ that satisfies the following  recurrence relation
\begin{align*}
    t_{n+1}=t_n+Ce^{-\lambda t_n}.
\end{align*}
Then $\lim_{n\to+\infty}t_n=+\infty.$
\end{lemma}
\begin{proof}
It is obvious that  $\{t_n\}_{n\ge 0}$ is a strictly increasing sequence. We now assume that $\lim_{n\to+\infty}t_n<+\infty.$ let $A$ be the smallest upper bound of $(t_n)_{n\ge 0}.$ Then for any $\e>0,$ there exist $n_0$ such that, for any $n\ge n_0, $ $t_n\ge A-\e$. We have
\begin{align*}
     t_{n+1}=t_n+Ce^{-\lambda t_n}\ge A-\e +Ce^{-A}.
\end{align*}
Set $\e\le Ce^{-A}/2$, then $t_{n+1}>A,$ contradicting the assumption that $t_{n} \leq A$ $(n \in \mathbb{N})$. Hence $\lim_{n\to+\infty}t_n=+\infty.$
\end{proof}

\begin{lemma}\label{lyap:witht} Let $(p^{i}_t,q^{i}_t)_{t \geq 0}$ solve (\ref{eq:AS:pqeiei}). Let $\theta_*^i$ be the critical point of $\Phi_i$ for $i\in I$. Under Assumption \ref{asSGPf}, \ref{as1.2} with $\mathcal{A}^{\lambda_i, 2}$, and \ref{as1.21} with $\lambda=\min_{i\in I}\lambda_i$, we have
\begin{align*}
    V^i(t,q^i_t,p^i_t)\le e^{-\lambda t}V^i(0,q^i_0,p^i_0) \qquad (t \geq 0),
\end{align*}
where $V^i(t,x,y)= \e(t)[\Phi_i(x)-\Phi_i(\theta^i_*)]+ \frac{1}{4}\Big(\norm{x-\theta^i_*+\e(t)y}^2+\norm{\e(t)y}^2\Big)$ is a Lyapunov function.
\end{lemma}
\begin{proof}
First, we differentiate $V^i(t,q^i_t,p^i_t)$ and obtain by (\ref{eq:AS:pqeiei}),
\begin{align}\label{eq:longlong}
    \de V^i(t,q^i_t,p^i_t)/\de t&= \de (\e(t)(\Phi_i(q^i_t)-\Phi_i(\theta_*)))/\de t+ \frac{1}{4}\de \Big(\norm{q^i_t-\theta^i_*+\e(t)p^i_t}^2+\norm{\e(t)p^i_t}^2\Big)/\de t\nonumber\\
    &= \e'(t)(\Phi_i(q^i_t)-\Phi_i(\theta^i_*))+\e(t)\ip{\nabla \Phi_i(q^i_t)}{\de q^i_t/\de t}\nonumber\\
    &\ + \frac{1}{2}\Big(\ip{q^i_t-\theta^i_*+\e(t)p^i_t}{\de (q^i_t-\theta^i_*+\e(t)p^i_t)/\de t}+\ip{\e(t)p^i_t}{\de (\e(t)p^i_t)/\de t}\Big)\nonumber\\
     &= \e'(t)(\Phi_i(q^i_t)-\Phi_i(\theta^i_*))+\e(t)\ip{\nabla \Phi_i(q^i_t)}{p^i_t}\nonumber\\
    &\ \ \ +\frac{1}{2}\Big(\ip{q^i_t-\theta^i_*+\e(t)p^i_t}{p^i_t+\e'(t)p^i_t+\e(t)\de p^i_t/\de t}\nonumber \\
    &\ \ \ \ \ \ \ \ +\ip{\e(t)p^i_t}{\e'(t)p^i_t+\e(t)\de p^i_t/\de t}\Big)\nonumber\\
    & \overset{(\ref{eq:AS:pqeiei})}{=} \e'(t)(\Phi_i(q^i_t)-\Phi_i(\theta^i_*))+\e(t)\ip{\nabla \Phi_i(q_t)}{p^i_t}\nonumber\\
    &\ \ \ +\frac{1}{2}\Big(\ip{q^i_t-\theta^i_*+\e(t)p^i_t}{\e'(t)p^i_t-\nabla \Phi_i(q^i_t)}\nonumber \\ 
    &\ \ \ \ \ \ \ \ +\ip{\e(t)p^i_t}{\e'(t)p^i_t-\nabla \Phi_i(q^i_t)-p^i_t}\Big)\nonumber\\
    &= \underbrace{\e'(t)(\Phi_i(q^i_t)-\Phi_i(\theta^i_*))}_{(\alpha_{11})}-\frac{1}{2}\ip{q^i_t-\theta^i_*}{\nabla \Phi_i(q^i_t)}\nonumber\\
    &\ \ \ \ +\e(t)(\e'(t)-\tfrac{1}{2})\norm{p^i_t}^2
    +\underbrace{\frac{\e'(t)}{2}\ip{q^i_t-\theta^i_*}{p^i_t}}_{(\alpha_{12})}.
    \end{align}
{Note that $(\alpha_{11})+m(t)\e'(t)\norm{p_t^i}^2\le 0$ since $\e'(t)\le 0$ and $\Phi_i(q^i_t)-\Phi_i(\theta^i_*)\ge 0$. By the $\varepsilon$-Young's inequality, we have $(\alpha_{12})\le \frac{\lambda}{4}\norm{q^i_t-\theta^i_*}^2+\frac{(\e'(t))^2}{4\lambda}\norm{p^i_t}^2.$ Then we have}
    \begin{align*}
    \eqref{eq:longlong}
    \le & -\frac{1}{2}\ip{q^i_t-\theta^i_*}{\nabla \Phi_i(q^i_t)}-\frac{\e(t)}{2}\norm{p^i_t}^2+\frac{\lambda}{4}\norm{q^i_t-\theta^i_*}^2+\frac{(\e'(t))^2}{4\lambda}\norm{p^i_t}^2\\
    \underbrace{\le}_{(\beta_1)} & -\lambda \Big(\Phi_i(q^i_t)-\Phi_i(\theta^i_*)+\norm{q^i_t-\theta^i_*}^2\Big)-\frac{\e(t)}{4}\norm{p^i_t}^2+\frac{\lambda}{4}\norm{q^i_t-\theta^i_*}^2\\
    \underbrace{\le}_{(\beta_2)} & -\lambda \Big(\Phi_i(q^i_t)-\Phi_i(\theta^i_*)+\frac{\norm{q^i_t-\theta^i_*+\e(t)p^i_t}^2}{4}+\frac{\norm{\e(t)p^i_t}^2}{4}\Big)+\frac{3\lambda\e^2(t)-\e(t)}{4}\norm{p^i_t}^2\\
    \le & -\lambda V^i(t,q^i_t,p^i_t).
\end{align*}
$(\beta_1)$ holds due to  Assumption \ref{as1.2} and \ref{as1.21}. In particular, Assumption \ref{as1.21} implies $\frac{(\e'(t))^2}{4\lambda}-\frac{\e(t)}{2}\le \frac{\lambda^2(\e(t))^2}{4\lambda}-\frac{\e(t)}{2}\le \frac{\lambda \e(t)}{4}-\frac{\e(t)}{2}\le -\frac{\e(t)}{4}.$ { Moreover, $(\beta_2)$ follows from $\frac{\norm{q^i_t-\theta^i_*+\e(t)p^i_t}^2}{4}\le \frac{\norm{q^i_t-\theta^i_*}^2+\norm{\e(t)p^i_t}^2}{2}$.}
Finally, by Gr\"onwall's inequality we have
$$
V^i(t,q^i_t,p^i_t)\le e^{-\lambda t}V^i(0,q^i_0,p^i_0).
$$
\end{proof}

\begin{lemma}\label{epsilong1}
Let $\e(t)>0$ be a strictly decreasing differentiable function and satisfy Assumption \ref{as1.21} with $\lambda=\mu/4$, where $\mu>0$. Then we have the following inequalities,
\begin{align*}
    \int_0^t\e^2(s)e^{\mu s}\de s\le 2\mu^{-1} e^{\mu t}\e^2(t), \qquad 
   \int_0^te^{-2\mathcal{E}(s)}e^{\mu s}\de s\le (1+\mu)\e_0 e^{\mu t}.
\end{align*}
\end{lemma}
\begin{proof}
For the first inequality, we have
\begin{align*}
     \int_0^t\e^2(s)e^{\mu s}\de s= \mu^{-1}\int_0^t\e^2(s)\de e^{\mu s}= \mu^{-1}\Big[\e^2(t)e^{\mu t}-\e_0- 2\int_0^t\e(s)\e'(s)e^{\mu s}\de s\Big],
\end{align*}
using integration by parts.
Under the Assumption \ref{as1.21},  
\begin{align*}
    - 2\int_0^t\e(s)\e'(s)e^{\mu s}\de s \le \frac{\mu}{2}\int_0^t\e^2(s)e^{\mu s}\de s.
\end{align*}
Hence, 
\begin{align*}
    \int_0^t\e^2(s)e^{\mu s}\de s\le \mu^{-1} e^{\mu t}\e^2(t)+ \frac{1}{2}\int_0^t\e^2(s)e^{\mu s}\de s. 
\end{align*}
This implies that
$$
 \int_0^t\e^2(s)e^{\mu s}\de s\le 2\mu^{-1} e^{\mu t}\e^2(t).
$$
For the second inequality, we use integration by parts and notice that
\begin{align*}
    \int_0^te^{-2\mathcal{E}(s)}e^{\mu s}\de s&=- \frac{1}{2}\int_0^t\e(s)e^{\mu s}\de e^{-2\mathcal{E}(s)}\\
    &= \frac{1}{2}\Big(\e_0-\e(t)e^{\mu t}e^{-2\mathcal{E}(t)}\Big)+\frac{1}{2}\int_0^te^{-2\mathcal{E}(s)}e^{\mu s}(\mu \e(s)+\e'(s))\de s\\
    &\le \frac{\e_0}{2}+\frac{3\mu\e_0 e^{\mu t}}{8}\int_0^te^{-2\mathcal{E}(s)}\de s\\
    &\le \frac{\e_0}{2}+\frac{3\mu\e_0 e^{\mu t}}{8}\int_0^te^{\frac{-2s}{\e_0}}\de s\le  (1+\mu)\e_0 e^{\mu t}.
\end{align*}
This completes the proof.
\end{proof}

\begin{lemma}\label{lem:iter}
Let $(\e_n)_{n\ge 0}$ be a non-negative decreasing sequence with $\lim_{n\to \infty}\e_n=0.$ Let $(d_n)_{n\ge 0}$ be a positive sequence satisfying 
   $ d_{n+1}\le cd_n+ \e_n,$
where $0<c<1$. Then we have $\lim_{n\to \infty}d_n=0.$
Furthermore, if $\e_n\le B e^{-bn}$ for some constant $\tilde B,\tilde b>0$ we have $d_n\le \tilde B e^{-\tilde b n}.$
\end{lemma}
\begin{proof}
It is obvious that $(\e_n)_{n\ge 0}$ is bounded and we denote the upper bound by $C.$ Hence $(d_n)_{n\ge 0}$ can be bounded by $D:=d_0+C/(1-c).$
Since $\lim_{n\to+\infty}\e_n=0$, for any $\e>0,$ there exist a $N_0,$ for any $N\ge N_0,$ $\e_n\le \e.$  Hence, for any $n\ge N_0,$ 
we have 
\begin{align*}
    d_{n+1}\le cd_n+ \e_n\le cd_n+\e. 
\end{align*}
which implies that
\begin{align*}
    d_n\le c^{n-N_0}d_{N_0}+\frac{\e}{1-c}.
\end{align*}
Since $\e$ is arbitrary, we get $\lim_{n\to \infty}d_n=0,$ which finishes the proof of the first part.   

Now, assuming that $\e_n\le B e^{-bn}$,  the relation $ d_{n+1}\le cd_n+ \e_n$ implies that
\begin{align*}
    d_{n+1}\le \sum^n_{i=0}c^{n-i}\e_i\le  B\sum^n_{i=0}c^{n-i}e^{-bi}= Bc^n \sum^n_{i=0}e^{-(\log c +b)i}.
\end{align*}
If $b+\log c\ge 0,$  $d_{n+1}\le Bnc^n$ the result is obvious. For  $b+\log c< 0,$ we have 
\begin{align*}
     d_{n+1}\le Bc^n \frac{e^{-(\log c +b)(n+1)}-1}{e^{-(\log c +b)}-1}\le Bc^n \frac{e^{-(\log c +b)(n+1)}}{e^{-(\log c +b)}-1}\le \tilde B e^{-b(n+1)}.
\end{align*}
\end{proof}

\begin{lemma}\label{twoseq1}
Let $G^n_k= \prod^n_{i=k}\Big(1-\frac{b}{\sqrt{i}}\Big)$, where $0<b<1$. Then there exist $c>0$ such that  for any $0\le k<n$, we have 
$
G^n_k\le e^{-c(\sqrt{n}-\sqrt{k})}.
$
\end{lemma}

\begin{proof}
For any $0\le x<1$ we have $\log(1-x)\le -x.$ This implies that
\begin{align*}
    \log G^n_k=& \log\Big( \prod^n_{i=k}\Big(1-\frac{b}{\sqrt{i}}\Big)\Big) =\sum_{i=k}^n \log\Big(1-\frac{b}{\sqrt{i}}\Big)\le -b\sum_{i=k}^n \frac{1}{\sqrt{i}}\\
    =& -b\sum_{i=k}^n \frac{1}{\sqrt{i}}\int^{i+1}_i\de x\le -b\sum_{i=k}^n \int^{i+1}_i \frac{1}{\sqrt{x}}\de x
    = -b \int^{n+1}_{k} \frac{1}{\sqrt{x}}\de x\\
    =& \frac{-b(\sqrt{n+1}-\sqrt{k})}{2}\le \frac{-b(\sqrt{n}-\sqrt{k})}{2}.
\end{align*}
Finally, by taking exponential both side and let $ c=-b/2$, we finish the proof.
\end{proof}


\section{Kushner's Perturbed Test Function Theory}\label{appendix:kushner}

In \citet[Chapter 7]{Kushner1}, the author considers a dynamical system with state space $\mathbb{R}^d$ of the form
\begin{align}\label{dy: kushner}
\frac{\mathrm{d}x}{\mathrm{d}t}^\gamma = \bar G(x^\gamma) +\tilde G(x^\gamma,\xi^\gamma)+F(x^\gamma,\xi^\gamma)/\gamma,
\end{align}
where $\gamma >0$ is a parameter and $(\xi^\gamma_t)_{t \geq 0}$ is a stochastic process. The author then studies the limiting behaviour of this dynamical system to, e.g.\ ODEs of type
 $\frac{\mathrm{d}x}{\mathrm{d}t} = \bar G(x)$. We now cite their main result.
\begin{theorem}[\cite{Kushner1}] \label{kushner}
We assume that assumptions (A4.2) to (A4.6) (see below) hold and let $(x^\gamma(0))_{\gamma >0}$ be tight. Then $(x^\gamma(t))_{t\geq0}$ is tight with respect to $\gamma >0$ and the limit of any weakly convergent   subsequence solves the martingale problem with the associate operator $A$ defined by 
\begin{align*}
    Af(x)=\nabla_x   f\cdot \bar G(x)+B_0f(x),
\end{align*}
for appropriate test functions $f: \mathbb{R}^d \rightarrow \mathbb{R}$.
\end{theorem}
We apply this theorem only in case $F = 0$, which simplifies the discussion in the following. For instance, we skip the definition of $B_0$, as $B_0 = 0$. Then, the operator $A$ is the infinitesimal generator of $\frac{\mathrm{d}x}{\mathrm{d}t} = \bar G(x)$ and we obtain convergence. We now list Assumptions (A4.2), (A4.3), and (A4.5). We skip Assumptions (A4.4) and (A4.6) as they are redundant if $F = 0$. 

\begin{description}
    \item[(A4.2):] $\bar{G},\ \tilde{G}, \ F,\ \nabla_xF$ are continuous. 
    \item[(A4.3):] The process $(\xi_t^\gamma)_{t \geq 0}$ satisfies $\xi_{t}^\gamma = \xi_{t/\gamma^2}$ $(t \geq 0)$ for a bounded and right continuous $(\xi_t)_{t \geq 0}$.
    \item[(A4.5):] For each $x\in \R^d,$ as $t,\tau\to \infty$,
\begin{align*}
    \frac{1}{\tau}\int_t^{t+\tau} \E[\tilde G(x,\xi(s))|(\xi(s'))_{s'\le t}] \de s \to 0,
    \end{align*}
    almost surely.
\end{description}

We employ Theorem~\ref{kushner} in the proof of Theorem~\ref{wcovpq}. There, we have $\xi_t=i(t)$, $\gamma=\nu^{1/2}, $ $\bar{G}(x,y) = (y,-\nabla_x \bar\Phi(x)-\alpha y)$, $\tilde G(x,y,i) = (0,- \nabla_x \Phi_i(x) + \nabla_x \bar\Phi(x))$ and $F=0$.

\bibliography{sample}

\end{document}